\documentclass[reqno,12pt]{amsart}

\usepackage{color} 
\usepackage[pdftex]{graphicx}
\usepackage[pdftex]{hyperref}
\hypersetup{
    unicode=false,          
    pdftoolbar=true,        
    pdfmenubar=true,        
    pdffitwindow=false,     
    pdfstartview={FitH},    
    pdftitle={MCP Article},      
    pdfauthor={Michael Holst},   
    pdfsubject={Mathematics},    
    pdfcreator={Michael Holst},  
    pdfproducer={Michael Holst}, 
    pdfkeywords={PDE, analysis, mathematical physics}, 
    pdfnewwindow=true,      
    colorlinks=true,        
    linkcolor=red,          
    citecolor=blue,         
    filecolor=magenta,      
    urlcolor=cyan           
}
\usepackage{times}
\usepackage{amsfonts}
\usepackage{amsmath}
\usepackage{amsthm}
\usepackage{amssymb}
\usepackage{amsbsy}
\usepackage{amscd}
\usepackage{enumerate}
\usepackage{mathtools}
\usepackage{pgf}
\usepackage{tikz}
\usepackage{tikz-cd}
\usepackage{mdframed}
\usepackage{pbox}
\usepackage{array}
\usepackage{float}
\newtheorem{theorem}{Theorem}[section]
\newtheorem{corollary}[theorem]{Corollary}
\newtheorem{lemma}[theorem]{Lemma}

\newtheorem{definition}[theorem]{Definition}
\newtheorem{remark}[theorem]{Remark}

\numberwithin{equation}{section}

\newcounter{mnote}
\let\oldmarginpar\marginpar
  \renewcommand\marginpar[1]{\-\oldmarginpar[\raggedleft\footnotesize #1]%
  {\raggedright\footnotesize #1}}

\newcommand{\reals}{\mathbb{R}}

\newcommand{\lab}{\label}

\newcommand{\floor}[1]{\lfloor{#1}\rfloor}

\definecolor{mve}{rgb}{0.7,0.35,0.15}
\definecolor{brght}{rgb}{0.825,0.2625,0.15}
\definecolor{yello}{rgb}{1,0.925,0.65}
\definecolor{bluu}{rgb}{0.65, 0.95, 1}
\definecolor{bluu2}{rgb}{0.2, 0.5, 0.8}

\newenvironment{itemizeX}
{\begin{list}{\labelitemi}
 {\setlength{\leftmargin}{1.5em}
  \setlength{\topsep}{0.5em}
  \setlength{\itemsep}{0.5em}
  \setlength{\labelwidth}{50.0em}}}
 {\end{list}}

\newenvironment{enumerateX}
{\begin{list}{(\arabic{enumi})}
 {\usecounter{enumi}
  \setlength{\leftmargin}{2.5em}
  \setlength{\topsep}{0.5em}
  \setlength{\itemsep}{0.5em}
  \setlength{\labelwidth}{50.0em}}}
 {\end{list}}

\newenvironment{enumerateXALI}
{\begin{list}{(\arabic{enumi})}
 {\usecounter{enumi}
  \setlength{\leftmargin}{1em}
  \setlength{\topsep}{0.5em}
  \setlength{\itemsep}{0.5em}
  \setlength{\labelwidth}{50.0em}}}
 {\end{list}}
 \newenvironment{itemizeXALI}
{\begin{list}{\labelitemi}
 {\setlength{\leftmargin}{0.4em}
  \setlength{\topsep}{0.5em}
  \setlength{\itemsep}{0.5em}
  \setlength{\labelwidth}{50.0em}}}
 {\end{list}}

\begin{document}

\title[Space of Locally Sobolev Functions]
      {Some Remarks on the Space of Locally Sobolev-Slobodeckij Functions}

\author[A. Behzadan]{A. Behzadan}
\email{a.behzadan@csus.edu}
\address{Department of Mathematics and Statistics\\
        California State University Sacramento\\
        Sacramento CA 95819}

\author[M. Holst]{M. Holst}
\email{mholst@math.ucsd.edu}
\address{Department of Mathematics\\
        University of California San Diego\\
        La Jolla CA 92093}

\thanks{AB was supported by NSF Awards~1262982.}
\thanks{MH was supported in part by
        NSF Awards~2012857, 1620366, and 1262982.}

\date{}

\begin{abstract}
The study of certain differential operators between Sobolev spaces of sections of vector bundles on compact manifolds equipped with rough metric is closely related to the study of locally Sobolev functions on domains in the Euclidean space.
In this paper we present a coherent rigorous study of some of the properties of locally Sobolev-Slobodeckij functions that are especially useful in the study of differential operators between sections of vector bundles on compact manifolds with rough metric. Results of this type in published literature generally can be found only for integer order Sobolev spaces $W^{m,p}$ or Bessel potential spaces $H^s$. Here we have presented the relevant results and their detailed proofs for Sobolev-Slobodeckij spaces $W^{s,p}$ where $s$ does not need to be an integer. We also develop a number of results needed in the study of differential operators on manifolds that do not appear to be in the literature.
\end{abstract}

\maketitle
{\footnotesize \tableofcontents }


\vspace*{-0.75cm}
\section{Introduction}
 \label{sec:intro}
The study of elliptic PDEs on compact manifolds naturally leads to the study of Sobolev spaces of functions and more generally Sobolev spaces of sections of vector bundles.
As it turns out, the study of certain differential operators between Sobolev spaces of sections of vector bundles on manifolds equipped with rough metric is closely related to the study of spaces of locally Sobolev functions on domains in the Euclidean space (see \cite{holstbehzadan2017c, holstbehzadan2018b}).

In this paper we focus on certain properties of spaces of locally Sobolev functions that are particularly useful in the study of differential operators on manifolds. Our work can be viewed as a continuation of the excellent work of Antonic and Burazin \cite{Antonic1}; their work is mainly concerned with the properties of spaces of locally Sobolev functions with integer smoothness degree. In particular, they study the following fundamental questions for Sobolev spaces with integer smoothness degree:
\begin{itemize}
\item Topology, metrizability
\item Density of smooth functions
\item Reflexivity, the nature of the dual
\item Continuity of differentiation between certain spaces of locally Sobolev functions
\end{itemize}
Our main goal here is to provide a self-contained manuscript in which the known results are collected and stated in the general setting of Sobolev-Slobodeckij spaces and then develop certain other results that are useful in the study of differential operators on manifolds. In particular, we will discuss
\begin{itemize}
\item General embedding results
\item Pointwise multiplication
\item Invariance under composition
\end{itemize}
Results of this type and other related results have been used in the literature -particularly in the study of Einstein constraint equations on manifolds equipped with rough metric- without complete proof. This paper should be viewed as a part of our efforts to fill some of the gaps. Interested readers can find other results in this direction in \cite{holstbehzadan2018d, holstbehzadan2017c, holstbehzadan2018b, holstbehzadan2015b}. Our hope is that the detailed presentation of this manuscript, along with these other four manuscripts, will help in better understanding the structure of the proofs and the properties of Sobolev-Slobodeckij spaces and locally Sobolev functions.


\section{Notation and Conventions}
\label{sec:notation}

Throughout this paper, $\reals$ denotes the set of real numbers, $\mathbb{N}$ denotes the set of positive integers, and $\mathbb{N}_0$ denotes the set of nonnegative integers. For any nonnegative real number $s$, the integer part of $s$ is denoted by $\floor{s}$. The letter $n$ is a positive integer and stands for the dimension of the space. For all $k\in \mathbb{N}$, $\textrm{GL}(k,\reals)$ is the set of all $k\times k$ invertible
matrices with real entries.\\

 $\Omega$ is a nonempty open set in $\reals^n$. The collection of all compact subsets of $\Omega$ will be denoted by $\mathcal{K}(\Omega)$. If $\mathcal{F}(\Omega)$ is any function space on $\Omega$ and
$K\in \mathcal{K}(\Omega)$, then $\mathcal{F}_K(\Omega)$ denotes
the collection of elements in $\mathcal{F}(\Omega)$ whose support
is inside $K$. Also
\begin{equation*}
\mathcal{F}_{comp}(\Omega)=\bigcup_{K\in
\mathcal{K}(\Omega)}\mathcal{F}_K(\Omega)
\end{equation*}
If $\Omega'\subseteq \Omega$ and $f: \Omega'\rightarrow \reals$, we denote the extension by zero of $f$ to the entire $\Omega$ by $\textrm{ext}^0_{\Omega',\Omega}f: \Omega \rightarrow \reals$, that is,
\begin{align*}
\textrm{ext}^0_{\Omega',\Omega}f (x)=
\begin{cases}
&f(x)\quad \textrm{if $x\in \Omega'$}\\
&0\quad \textrm{otherwise}
\end{cases}
\end{align*}

Lipschitz domain in $\reals^n$ refers to a nonempty bounded open set in $\reals^n$ with Lipschitz continuous boundary. We say that a nonempty open set $\Omega\subseteq \reals^n$ has the \textbf{interior Lipschitz property} provided that for each compact set $K\in \mathcal{K}(\Omega)$ there exists a bounded open set $\Omega'\subseteq \Omega$ with Lipschitz continuous boundary such that $K\subseteq \Omega'$.\\

Each element of $\mathbb{N}_0^n$ is called a multi-index. For a multi-index $\alpha=(\alpha_1,\cdots,\alpha_n)\in \mathbb{N}_0^n$, we let $|\alpha|:=\alpha_1+\cdots+\alpha_n$. Also for sufficiently smooth functions $u:\Omega\rightarrow \reals$ (or for any distribution $u$) we define the $\alpha$th order partial derivative of $u$ as follows:
\begin{equation*}
\partial^\alpha u:=\frac{\partial^{|\alpha|}u}{\partial x_1^{\alpha_1}\cdots \partial x_n^{\alpha_n}}\,.
\end{equation*}

We use the notation $A\preceq B$ to mean $A\leq cB$, where $c$ is a positive constant that does not depend on the non-fixed parameters appearing in $A$ and $B$. We write $A\simeq B$ if $A\preceq B$ and $B\preceq A$.\\

If $X$ and $Y$ are two topological spaces, we use the
notation $X\hookrightarrow Y$ to mean $X \subseteq Y$ and the
inclusion map is continuous.
\section{Background Material}
In this section we collect some useful tools and facts we will need from topology and analysis. Statements without proof in this section are mainly taken from Rudin's functional analysis \cite{Rudi73}, Grubb's distributions and operators \cite{9}, excellent presentation of Reus \cite{Reus1}, Treves' topological vector spaces \cite{Treves1}, and \cite{holstbehzadan2018b} or are direct consequences of statements in the aforementioned references.

\subsection{Topological Vector Spaces}
\begin{definition}\lab{winter10}
A topological vector space is a vector space $X$ together with a
topology $\tau$ with the following properties:
\begin{enumerate}[i)]
\item For all $x\in X$, the singleton $\{x\}$ is a closed set.
\item The maps
\begin{align*}
& (x,y)\mapsto x+y \qquad (\textrm{from $X\times X$ into $X$})\\
& (\lambda,x)\mapsto \lambda x\qquad (\textrm{from $\reals\times X
$ into $X$})
\end{align*}
are continuous where $X\times X$ and $\reals \times X$ are
equipped with the product topology.
\end{enumerate}
\end{definition}

\begin{definition}\lab{winter11}
Suppose $(X,\tau)$ is a topological vector space and $Y\subseteq
X$.
\begin{itemizeXALI}
\item $Y$ is said to be \textbf{convex} if for all $y_1, y_2\in Y$ and $t\in
(0,1)$ it is true that $ty_1+(1-t)y_2\in Y$.
\item We say $Y$ is \textbf{bounded} if for any  neighborhood $U$ of the
origin (i.e. any open set containing the origin), there exits
$t>0$ such that $Y\subseteq tU$.
\end{itemizeXALI}
\end{definition}

\begin{definition}\lab{winter15}
Let $(X,\tau)$ be a topological vector space. $X$ is said to be
\textbf{metrizable} if there exists a metric $d: X\times X\rightarrow
[0,\infty)$ whose induced topology is $\tau$. In this case we say
that the metric $d$ is compatible with the topology $\tau$.
\end{definition}

\begin{theorem}\cite{9, Rudi73}\lab{thmmay71}
Let $(X,\tau)$ be a topological vector space. The following are equivalent:
\begin{itemize}
\item $X$ is metrizable.
\item There exists a translation invariant metric $d$ on $X$ whose collection of open sets is the same as $\tau$. Translation invariant means
    \begin{equation*}
\forall\,x,y,a\in X\qquad d(x+a,y+a)=d(x,y)
\end{equation*}
\item $X$ has a countable local base at the origin.
\end{itemize}
\end{theorem}
\noindent (Recall that a subcollection $\mathcal{B}$ of $\tau$ is said to be a \emph{local base} at the origin if for any open set $U$ containing the origin there is $B\in \mathcal{B}$ such that $0\in B\subseteq U$.)
\begin{remark}\lab{remmay72}
It can be shown that if
$d_1$ and $d_2$ are two translation invariant metrics that induce
the same topology on $X$, then the Cauchy sequences of $(X,d_1)$
will be exactly the same as the Cauchy sequences of $(X,d_2)$.
\end{remark}

\begin{definition}\lab{winter24}
Let $(X,\tau)$ be a topological vector space. We say $(X,\tau)$
is \textbf{locally convex} if it has a convex local base at the
origin.
\end{definition}

\begin{definition}\lab{winter25}
Let $(X,\tau)$ be a metrizable locally convex topological vector
space. Let $d$ be any translation invariant metric on $X$ that is
compatible with $\tau$. We say that $X$ is \textit{complete} if and only if
the metric space $(X,d)$ is a complete metric space. A complete
metrizable locally convex topological vector space is called a
\textbf{Frechet space}.
\end{definition}

\begin{definition}\lab{winter27}
A \textbf{seminorm} on a vector space $X$ is a real-valued function
$p:X\rightarrow \reals$ such that
\begin{enumerate}[i.]
\item $\forall\,x,y\in X\qquad p(x+y)\leq p(x)+p(y)$
\item $\forall\,x\in X\,\,\forall\,\alpha\in \reals\qquad p(\alpha
x)=|\alpha| p(x)$
\end{enumerate}
If $\mathcal{P}$ is a family of seminorms on $X$, then we say
$\mathcal{P}$ is \textit{separating} provided that for all $x\neq 0$ there
exists at least one $p\in\mathcal{P}$ such that $p(x)\neq 0$
(that is if $p(x)=0$ for all $p\in\mathcal{P}$, then $x=0$). It easily follows from the definition that any seminorm is a nonnegative function.
\end{definition}

\begin{theorem}\lab{thmmay3626}
Suppose that $(X,\|.\|_X)$ is a normed space. Let $p: X\rightarrow \reals$ be a seminorm on $X$. If $p$ is continuous, then there exists a constant $C>0$ such that
\begin{equation*}
\forall\,x\in X\qquad p(x)\leq C\|x\|_X
\end{equation*}
\end{theorem}

\begin{proof}
$p$ is continuous at $0$ so there exists $\delta>0$ such that if $\|x\|_X\leq \delta$ then $|p(x)|<1$. If $x\neq 0$, then $\delta\frac{x}{\|x\|_X}$ has norm equal to $\delta$ and so for all $x\neq 0$, $\displaystyle p(\delta\frac{x}{\|x\|_X})<1$. Hence for all $x\neq 0$ we have
\begin{equation*}
p(x)\leq\frac{1}{\delta}\|x\|_X
\end{equation*}
Since $p(0)=0$, clearly the above inequality also holds for $x=0$.
\end{proof}

\begin{definition}\lab{winter28}
Suppose $\mathcal{P}$ is a separating family of seminorms on a
vector space $X$. The \textbf{natural topology} induced by $\mathcal{P}$ is the smallest topology on $X$ that is translation
invariant and with respect to which every $p\in \mathcal{P}$ is
continuous function from $X$ to $\reals$. (Recall that \emph{translation invariant} means if $U\subseteq X$ is open, then $U+x$ is open for every $x\in X$.)
\end{definition}

\begin{remark}\lab{remmay2421}
Suppose that $\mathcal{P}$ and $\mathcal{P}'$ are two separating family of seminorms on a vector space $X$. Let $\tau$ and $\tau'$ be the corresponding natural topologies on $X$. It follows immediately from the definition that if 1)$p: (X,\tau')\rightarrow \reals$ is continuous for each $p\in \mathcal{P}$ and  2) $p': (X,\tau)\rightarrow \reals$ is continuous for each $p'\in \mathcal{P}'$, then $\tau=\tau'$.
\end{remark}

The following theorem can be viewed as an extension of Theorem ~\ref{thmmay3626}.
\begin{theorem}[\cite{Reus1}, Page 157]\lab{thmmay2656}
Let $X$ be a vector space and suppose $\mathcal{P}$ is a separating family of seminorms on $X$. Equip $X$  with the corresponding natural topology. Then a seminorm $q:X\rightarrow \reals$ is continuous if and only if there exist $C>0$ and $p_1,\cdots,p_m\in \mathcal{P}$ such that for all $x\in X$
\begin{equation*}
q(x)\leq C \big(p_1(x)+\cdots+p_m(x)\big)\,.
\end{equation*}
\end{theorem}

\begin{theorem}\cite{9,Rudi73}\lab{thmmay7310}
Suppose $\mathcal{P}$ is a separating family of seminorms on a
vector space $X$ and $\tau$ is the corresponding natural topology on $X$. Then $(X,\tau)$ is a locally convex topological vector space. Moreover, if $\mathcal{P}=\{p_k\}_{k\in \mathbb{N}}$ is countable, then
 the locally convex topological vector space $(X,\tau)$ is metrizable and the
 following translation invariant metric on $X$ is compatible with
 $\tau$:
 \begin{equation*}
 d(x,y)=\sum_{k=1}^\infty\frac{1}{2^k}\frac{p_k(x-y)}{1+p_k(x-y)}
 \end{equation*}
\end{theorem}

\begin{corollary}\lab{corspring181}
Suppose $\mathcal{P}$ is a countable separating family of seminorms on a
vector space $X$ and $\tau$ is the corresponding natural topology on $X$. Then $(X,\tau)$ is a Frechet space if and only if it is complete.
\end{corollary}

\begin{theorem}[\cite{Narici1985}, Sections 6.4 and 6.5] \lab{thmmay3627}
Let $(X,\tau)$ be a locally convex topological vector space. Then there exists a separating family of seminorms on $X$ whose corresponding natural topology is $\tau$.
\end{theorem}

\begin{theorem}[\cite{Rudi73}, Page 28]\lab{thmmay3814}
Suppose $\mathcal{P}$ is a separating family of seminorms on a
vector space $X$ and $\tau$ is the corresponding natural topology on $X$. Then a set $E\subseteq X$ is bounded if and only if $p(E)$ is a bounded set in $\reals$ for all $p\in \mathcal{P}$.
\end{theorem}

\begin{corollary}\lab{cormay13343}
Suppose $\mathcal{P}$ is a separating family of seminorms on a vector space $X$ and $\tau$ is the corresponding natural topology on $X$. It follows from Theorem ~\ref{thmmay2656} and Theorem ~\ref{thmmay3814} that if $E\subseteq X$ is bounded, then for any continuous seminorm $q:(X,\tau)\rightarrow \reals$, $q(E)$ is a bounded set in $\reals$.
\end{corollary}

\begin{theorem}[\cite{9}, Page 436, \cite{Narici1985}, Section 6.6]\lab{thmmay11714}
Let $(X,\tau)$ be a topological vector space. Suppose $\mathcal{Q}$ is a separating family of seminorms on a
vector space $Y$ and $\tau'$ is the corresponding natural topology on $Y$. Then a linear map $T:(X,\tau)\rightarrow (Y,\tau')$ is continuous if and only if for each $q\in \mathcal{Q}$, $q\circ T$ is continuous on $X$.
\end{theorem}

\begin{theorem}\cite{9}\lab{thmapptvconvergence2}
Let $X$ be a Frechet space and let $Y$ be a topological vector
space. When $T$ is a linear map of $X$ into $Y$, the following
two properties are equivalent
\begin{enumerate}
\item $T$ is continuous.
\item $x_n\rightarrow 0$ in $X$ $\Longrightarrow$ $Tx_n\rightarrow
0$ in $Y$.
\end{enumerate}
\end{theorem}

\begin{theorem}\cite{9, Rudi73}\lab{thmapptvconvergence1}
Let $X$ and $Y$ be two vector spaces and suppose $\mathcal{P}$
and $\mathcal{Q}$ are two separating families of seminorms on $X$
and $Y$, respectively. Equip $X$ and $Y$ with the corresponding
natural topologies. Then
\begin{enumerateXALI}
\item A sequence $x_n$ converges to $x$ in $X$ if and only if for
all $p\in \mathcal{P}$, $p(x_n-x)\rightarrow 0$.
\item A linear operator $T:X\rightarrow Y$ is continuous if and
only if
{\fontsize{10}{10}{\begin{equation*}
\forall\,q\in \mathcal{Q}\quad \exists\,c>0,\,k\in
\mathbb{N},\,p_1,\cdots,p_k\in \mathcal{P}\quad \textrm{such
that}\quad \forall\,x\in X\quad |q\circ T(x)|\leq c\max_{1\leq
i\leq k}p_i(x)
\end{equation*}}}
\item A linear operator $T:X\rightarrow \reals$ is continuous if and
only if
\begin{equation*}
\exists\,c>0,\,k\in \mathbb{N},\,p_1,\cdots,p_k\in
\mathcal{P}\quad \textrm{such that}\quad \forall\,x\in X\quad
|T(x)|\leq c\max_{1\leq i\leq k}p_i(x)
\end{equation*}
\end{enumerateXALI}
\end{theorem}

\begin{definition}\lab{winter31}
Let $(X,\tau)$ be a locally convex topological vector space.
\begin{itemizeXALI}
\item The \textbf{weak topology} on $X$ is the natural topology
induced by the separating family of seminorms $\{p_F\}_{F\in X^*}$ where
\begin{equation*}
\forall\,F\in X^*\qquad p_F: X\rightarrow \reals,\quad
p_F(x):=|F(x)|
\end{equation*}
It can be shown that this topology is the smallest (weakest) topology with respect to which all the
 linear maps in $[(X,\tau)]^*$ are continuous. A sequence $\{x_m\}$ converges to $x$ in $X$ with respect to
the weak topology if and only if $F(x_m)\rightarrow F(x)$ in
$\reals$ for all $F\in X^*$. In this case we may write $x_m\rightharpoonup x$. We denote the weak topology on $X$ by $\sigma (X,X^*)$. It can be shown that $[(X,\tau)]^*$ is the same set as $[(X,\sigma(X,X^*))]^*$.
\item The \textbf{weak$^*$ topology} on $X^*$ is the natural topology
induced by the separating family of seminorms $\{p_x\}_{x\in X}$ where
\begin{equation*}
\forall\,x\in X\qquad p_x: X^*\rightarrow \reals,\quad
p_x(f):=|f(x)|
\end{equation*}
It can be shown that this topology is the weakest topology with respect to which all the
 linear maps $\{f\mapsto f(x)\}_{x\in X}$ (from $X^*$ to $\reals$) are continuous. A sequence $\{f_m\}$ converges to $f$ in $X^*$ with respect to
the weak$^*$ topology if and only if $f_m(x)\rightarrow f(x)$ in
$\reals$ for all $x\in X$. We denote the weak$^*$ topology on $X^*$ by $\sigma (X^*,X)$.
\item The \textbf{strong topology} on $X^*$ is the natural topology induced
by the separating family of seminorms $\{p_B\}_{B\subseteq X
\textrm{bounded}}$ where for any bounded subset $B$ of $X$
\begin{equation*}
p_B: X^*\rightarrow \reals\qquad p_B(f):=\sup \{|f(x)|: x\in B\}
\end{equation*}
(it can be shown that for any bounded subset $B$ of $X$ and $f\in
X^*$, $f(B)$ is a bounded subset of $\reals$; see Theorem ~\ref{thmmay3814} and Theorem ~\ref{thmmay7320})
\end{itemizeXALI}
\end{definition}
\begin{remark}\lab{winter32}
\leavevmode
\begin{enumerateX}
\item If $X$ is a normed space, then the topology induced by the
norm
\begin{equation*}
\forall\,f\in X^*\qquad \|f\|_{op}=\sup_{\|x\|_X=1}|f(x)|
\end{equation*}
on $X^*$ is the same as the strong topology on $X^*$
(\cite{Treves1}, Page 198).
\item In this manuscript, unless otherwise stated, we consider the topological dual of a
locally convex topological vector space with the strong topology. Of course, it
is worth mentioning that for many of the spaces that we will
consider (including $X=\mathcal{E}(\Omega)$ or $X=D(\Omega)$
where $\Omega$ is an open subset of $\reals^n$) a sequence in
$X^*$ converges with respect to the weak$^*$ topology if and only
if it converges with respect to the strong topology (for more
details on this see the definition and properties of \textbf{Montel
spaces} in section 34.4, page 356 of \cite{Treves1}).
\end{enumerateX}
\end{remark}

\begin{theorem}\lab{thmmay7314}
Let $(X,\tau)$ be a locally convex topological vector space. Then the evaluation map
\begin{equation*}
J: (X,\tau)\rightarrow X^{**}:=[(X^*,\textrm{strong topology})]^*,\qquad J(x)(F):=F(x)
\end{equation*}
is a well-defined injective linear map. $X^{**}$ is called the \textbf{bidual} of $X$.
\end{theorem}

\begin{definition}\lab{defmay7316}
Let $(X,\tau)$ be a locally convex topological vector space. Let $\tau'$ denote the strong topology on $X^{**}$ as the dual of $(X^*,\textrm{strong topology})$.
\begin{itemizeX}
\item If the evaluation map $J: (X,\tau)\rightarrow (X^{**},\tau')$ is bijective, then we say that $(X,\tau)$ is a semireflexive space.
\item If the evaluation map $J: (X,\tau)\rightarrow (X^{**},\tau')$ is a linear topological isomorphism, then we say that $(X,\tau)$ is a reflexive space.
\end{itemizeX}
\end{definition}

\begin{theorem}[\cite{Petersen83}, Pages 16 and 17]\lab{thmmay6758}
\leavevmode
\begin{itemize}
\item Strong dual of a reflexive topological vector space is reflexive.
\item Every semireflexive space whose topology is defined by the inductive limit of a sequence of Banach spaces is reflexive.
 \item  Every semireflexive Frechet space is reflexive. 
\end{itemize}
\end{theorem}

\begin{theorem}\lab{thmapril2713}
Let $(X,\tau_X)$ and $(Z,\tau_Z)$ be two locally convex topological vector spaces. For all $x\in X$, let $l_x: X^*\rightarrow \reals$ be the linear map defined by $l_x(f)=f(x)$. Then
\begin{enumerate}
\item A linear map $T:(Z,\tau_Z)\rightarrow (X,\sigma(X,X^*))$ is continuous if and only if for all $F\in [(X,\tau_X)]^*$, the linear map $F\circ T: (Z,\tau_Z)\rightarrow \reals$ is continuous.
\item A linear map $T:(Z,\tau_Z)\rightarrow (X^*,\sigma(X^*,X))$ is continuous if and only if for all $x\in X$, the linear map $l_x\circ T: (Z,\tau_Z)\rightarrow \reals$ is continuous.
\end{enumerate}
\end{theorem}

\begin{theorem}[\cite{Reus1}, Page 163, \cite{9}, Page 46] \lab{thmfallinjectiveadjoint1} Let $X$ and $Y$ be locally
convex topological vector spaces and suppose $T:X\rightarrow Y$
is a continuous linear map. Either equip both $X^*$ and $Y^*$ with the strong topology or equip both with the weak$^*$ topology. Then
\begin{enumerateXALI}
\item the map
\begin{equation*}
T^*: Y^*\rightarrow X^*\qquad \langle T^*y,x \rangle_{X^*\times
X}=\langle y,Tx \rangle_{Y^*\times Y}
\end{equation*}
is well-defined, linear, and continuous. ($T^*$ is called the
\textbf{adjoint} of $T$.)
\item If $T(X)$ is dense in $Y$, then $T^*:Y^*\rightarrow X^*$ is
injective.
\end{enumerateXALI}
\end{theorem}

\begin{theorem}[\cite{Rudi73}, Page 70]\lab{thmmay7320}
Let $(X,\tau)$ be a locally convex topological vector space. Then a set $E\subseteq X$ is bounded with respect to $\tau$ if and only if it is bounded with respect to $\sigma (X,X^*)$.
\end{theorem}

\begin{corollary}\lab{cormay7322}
If $(X,\tau)$ is a locally convex topological vector space and $x_n\rightharpoonup x$  (i.e. $x_n$ converges to $x$ with respect to $\sigma(X,X^*)$), then $\{x_n\}$ is bounded with respect to both $\tau$ and $\sigma (X,X^*)$.
\end{corollary}

\begin{theorem}\lab{thmmay2}
Let $(X,\tau_X)$ and $(Y,\tau_Y)$ be two locally convex topological vector spaces. If $T:(X,\tau_X)\rightarrow (Y^*, \sigma(Y^*,Y))$ is continuous, then $T: (X,\sigma(X,X^*))\rightarrow (Y^*, \sigma(Y^*,Y))$ is continuous. In particular, if $u_n\rightharpoonup u$ (i.e. $u_n$ converges to $u$ with respect to $\sigma(X,X^*)$), then $T(u_n)\rightarrow T(u)$ in $(Y^*, \sigma(Y^*,Y))$.
\end{theorem}

\begin{proof}
For all $y\in Y$, let $l_y: Y^*\rightarrow \reals$ be the map $l_y(F)=F(y)$. By Theorem ~\ref{thmapril2713} $T: (X,\sigma(X,X^*))\rightarrow (Y^*, \sigma(Y^*,Y))$ is continuous if $l_y\circ T: (X,\sigma(X,X^*))\rightarrow \reals$ is continuous for all $y\in Y$. Let $y\in Y$.
\begin{enumerateX}
\item By definition of the weak$^*$ topology on $Y^*$, we know that the linear map $l_y: Y^*\rightarrow \reals$ is continuous.
\item By assumption $T:(X,\tau_X)\rightarrow (Y^*, \sigma(Y^*,Y))$ is a continuous linear map.
\end{enumerateX}
Therefore, $l_y\circ T$  belongs to $[(X,\tau_X)]^*$. Since $\sigma(X,X^*)$ is the weakest topology on $X$ that makes all elements of $[(X,\tau_X)]^*$ continuous, we can conclude that $l_y\circ T: (X,\sigma(X,X^*))\rightarrow \reals$ is continuous.
\end{proof}

\begin{theorem}[\cite{Vogt2000}, Page 13]\lab{thmmay7323}
Let $(X,\tau)$ be a Frechet space. Then $X$ is reflexive if and only if every bounded set $E$ in $X$ is relatively weakly compact (i.e. the closure of $E$ w.r.t $\sigma(X,X^*)$ is compact w.r.t $\sigma(X,X^*)$ ).
\end{theorem}

\begin{theorem}[\cite{Bogachev2017}, Page 167]\lab{thmmay7324}
Let $(X,\tau)$ be a separable Frechet space. If $E\subseteq X$ is relatively weakly compact, then every infinite sequence in $E$ has a subsequence that converges in $(X, \sigma(X,X^*))$.
\end{theorem}

The next theorem is an immediate consequence of the previous theorems.

\begin{theorem}\lab{thmmay2858}
Suppose that $(X,\tau)$ is a separable reflexive Frechet space. Then every bounded sequence in $(X,\tau)$ has a weakly convergent subsequence, that is, a subsequence that converges w.r.t $\sigma(X,X^*)$.
\end{theorem}

\begin{theorem}[\cite{Brezis2011}, Page 61]\lab{thmmay21022}
Let $X$ and $Y$ be two Banach spaces. Let $T:X\rightarrow Y$ be a linear map. Then $T$ is continuous if and only if it is weak-weak continuous, that is, $T:(X,\|.\|_X)\rightarrow (Y,\|.\|_Y)$ is continuous if and only if $T:(X,\sigma(X,X^*))\rightarrow (Y,\sigma(Y,Y^*))$ is continuous.
\end{theorem}


\begin{theorem}\lab{thmmay21220}
Let $X$ be a Banach space and $Y$ be a closed subspace of $X$ with the induced norm. Suppose that ${y_m}$ is a sequence in $Y$ and $y\in Y$. If $y_m\rightarrow y$ in $(X, \sigma(X,X^*))$, then $y_m\rightarrow y$ in $(Y,\sigma(Y,Y^*))$.
\end{theorem}

\begin{proof}
This is a direct consequence of the fact that the following two topologies on the space $Y$ are the same (see \cite{Brezis2011}, Page 70):
\begin{enumerate}
\item the topology induced by $\sigma(X,X^*)$
\item the topology $\sigma(Y,Y^*)$.
\end{enumerate}
\end{proof}

\begin{definition}\lab{winter36}
Let $X$ be a vector space and let $\{X_\alpha\}_{\alpha\in I}$ be a family of vector subspaces of $X$ with the property that
\begin{itemize}
\item for each $\alpha \in I$, $X_\alpha$ is equipped with a topology that makes it a locally convex topological vector space, and
\item $\bigcup_{\alpha\in I} X_\alpha=X$.
\end{itemize}
The \textbf{inductive limit topology} on $X$ with respect to the family $\{X_\alpha\}_{\alpha\in I}$ is defined to be the largest topology with
respect to which
\begin{enumerate}
\item $X$ is a locally convex topological vector space, and
\item all the inclusions $X_\alpha\subseteq
X$ are continuous.
\end{enumerate}
\end{definition}
\begin{theorem}\cite{Reus1}\lab{winter37}
Let $X$ be a vector space equipped with the inductive limit topology with respect to $\{X_\alpha\}$ as described above. If $Y$ is a locally convex vector space, then a linear map $T:X\rightarrow Y$ is continuous if and only if $T|_{X_\alpha}: X_\alpha\rightarrow Y$ is continuous for all $\alpha\in I$.
\end{theorem}

\begin{definition}\lab{defmay7606}
Let $X$ be a vector space and let $\{X_j\}_{j\in \mathbb{N}_0}$ be an increasing chain of subspaces of $X$:
\begin{equation*}
X_0\subsetneq X_1\subsetneq X_2\subsetneq \cdots
\end{equation*}
Suppose that
\begin{itemize}
\item each $X_j$ is equipped with a locally convex topology $\tau_j$;
\item for each $j$, the inclusion $(X_j,\tau_j)\hookrightarrow (X_{j+1},\tau_{j+1})$ is a linear topological embedding with closed image.
\end{itemize}
Then the inductive limit topology on $X$ with respect to the family $\{X_j\}_{j\in \mathbb{N}_0}$ is called a \textbf{strict inductive limit topology}.
\end{definition}

\begin{theorem}\cite{Reus1}\lab{thmmay7609}
Suppose that $X$ is equipped with the strict inductive limit topology with respect to the chain $\{X_j\}_{j\in \mathbb{N}_0}$. Then a subset $E$ of $X$ is bounded if and only if there exists $m\in \mathbb{N}_0$ such that $B$ is bounded in $X_m$.
\end{theorem}

\subsection{Function Spaces and Distributions}
\begin{definition}\lab{defmay7634}
Let $\Omega$ be a nonempty open set in $\reals^n$ and $m\in
\mathbb{N}_0$.
\begin{align*}
& C(\Omega)=\{f:\Omega\rightarrow \reals: \textrm{$f$ is
continuous}\}\\
& C^m(\Omega)=\{f: \Omega\rightarrow \reals:
\textrm{$\forall\,|\alpha|\leq m\quad \partial^\alpha f \in
C(\Omega)$}\}\qquad (C^0(\Omega)=C(\Omega))\\
& BC(\Omega)=\{f:\Omega\rightarrow \reals: \textrm{$f$ is
continuous and bounded on $\Omega$}\}\\
& BC^m(\Omega)=\{f\in C^m(\Omega): \textrm{$\forall\,|\alpha|\leq
m\quad \partial^\alpha f$ is bounded on $\Omega$}\}\\
&C^\infty(\Omega)=\bigcap_{m\in \mathbb{N}_0}C^m(\Omega),\quad
BC^\infty(\Omega)=\bigcap_{m\in \mathbb{N}_0}BC^m(\Omega)\\
&C_c^\infty(\Omega)=\{f\in C^\infty(\Omega): \textrm{support of $f$ is an element of $\mathcal{K}(\Omega)$}\}
\end{align*}
\end{definition}
Let $0<\lambda\leq 1$. A function $F: \Omega\subseteq \reals^n
\rightarrow \reals^k$ is called \textit{$\lambda$-Holder continuous} if
there exists a constant $L$ such that
\begin{equation*}
|F(x)-F(y)|\leq L|x-y|^\lambda\quad \forall\,x,y\in \Omega
\end{equation*}
Clearly a $\lambda$-Holder continuous function on $\Omega$ is
uniformly continuous on $\Omega$. $1$-Holder continuous functions
are also called \textit{Lipschitz continuous} functions or simply Lipschitz
functions. We define
\begin{align*}
BC^{m,\lambda}(\Omega)&=\{f:\Omega\rightarrow \reals:
\textrm{$\forall\,|\alpha|\leq m\quad
\partial^\alpha f$ is $\lambda$-Holder continuous and bounded}\}\\
&=\{f\in BC^m(\Omega): \textrm{$\forall\,|\alpha|\leq m\quad
\partial^\alpha f$ is $\lambda$-Holder continuous}\}
\end{align*}
and
\begin{equation*}
BC^{\infty,\lambda}(\Omega):=\bigcap_{m\in \mathbb{N}_0}BC^{m,\lambda}(\Omega)
\end{equation*}

\begin{theorem}\cite{9}\lab{thmmay7635}
Let $\Omega$ be a nonempty open set in $\reals^n$ and let $K\in \mathcal{K}(\Omega)$. There is a function $\psi\in C_c^\infty(\Omega)$ taking values in $[0,1]$ such that $\psi=1$ on a neighborhood containing $K$.
\end{theorem}

\begin{theorem}[Exhaustion by compact sets]\cite{9}\lab{winter2}
Let $\Omega$ be a nonempty open subset of $\reals^n$. There
exists a sequence of compact subsets $(K_j)_{j\in\mathbb{N}}$
such that $\cup_{j\in \mathbb{N}} \mathring{K}_j=\Omega$ and
\begin{equation*}
K_1\subseteq \mathring{K}_2\subseteq K_2\subseteq \cdots\subseteq
\mathring{K}_j\subseteq K_j\subseteq \cdots
\end{equation*}
Moreover, as a direct consequence, if $K$ is any compact subset
of the open set $\Omega$, then there exists an open set $V$ such
that $K\subseteq V\subseteq \bar{V}\subseteq \Omega$.
\end{theorem}

\begin{theorem}\cite{9}\lab{winter3}
Let $\Omega$ be a nonempty open subset of $\reals^n$. Let
$\{K_j\}_{j\in\mathbb{N}}$ be an exhaustion of $\Omega$ by compact
sets. Define
\begin{equation*}
V_0=\mathring{K}_4,\qquad \forall\, j\in \mathbb{N}\quad
V_j=\mathring{K}_{j+4}\setminus K_j
\end{equation*}
Then
\begin{enumerateXALI}
\item Each $V_j$ is an open bounded set and $\Omega=\cup_j V_j$.
\item The cover $\{V_j\}_{j\in\mathbb{N}_0}$ is \textbf{locally finite} in
$\Omega$, that is, each compact subset of $\Omega$ has nonempty
intersection with only a finite number of the $V_j$'s.
\item There is a family of functions $\psi_j\in
C_c^{\infty}(\Omega)$ taking values in $[0.1]$ such that
$\textrm{supp}\,\psi_j\subseteq V_j$ and
\begin{equation*}
\sum_{j\in \mathbb{N}_0}\psi_j(x)=1\qquad \textrm{for all $x\in
\Omega$}
\end{equation*}
\end{enumerateXALI}
\end{theorem}

Let $\Omega$ be a nonempty open set in $\reals^n$. For all $\varphi\in C^{\infty}(\Omega)$, $j\in \mathbb{N}$ and $K\in \mathcal{K}(\Omega)$ we define
\begin{equation*}
\|\varphi\|_{j,K}:=\sup \{|\partial^\alpha \varphi(x)|:
|\alpha|\leq j, x\in K\}
\end{equation*}
 For all $j\in \mathbb{N}$ and $K\in \mathcal{K}(\Omega)$, $\|.\|_{j,K}$
is a seminorm on $C^{\infty}(\Omega)$. We define
$\mathcal{E}(\Omega)$ to be $C^\infty(\Omega)$ equipped with the
natural topology induced by the separating family of seminorms
$\{\|.\|_{j,K}\}_{j\in \mathbb{N},K\in \mathcal{K}(\Omega)}$. It
can be shown that $\mathcal{E}(\Omega)$ is a Frechet space.\\\\
 For all $K\in \mathcal{K}(\Omega)$ we define $\mathcal{E}_K(\Omega)$ to be
$C_K^\infty(\Omega)$ equipped with the subspace topology. Since
$C^\infty_K(\Omega)$ is a closed subset of the Frechet space
$\mathcal{E}(\Omega)$, $\mathcal{E}_K(\Omega)$ is also a Frechet
space.\\\\
 We define $D(\Omega)=\bigcup_{K\in
\mathcal{K}(\Omega)}\mathcal{E}_K(\Omega)$ equipped with the
inductive limit topology with respect to the family of vector subspaces $\{\mathcal{E}_K(\Omega)\}_{K\in
\mathcal{K}(\Omega)}$.  It can be shown that if $\{K_j\}_{j\in \mathbb{N}_0}$ is an exhaustion by compacts sets of $\Omega$, then the inductive limit topology on $D(\Omega)$ with respect to the family $\{\mathcal{E}_{K_j}\}_{j\in \mathbb{N}_0}$ is exactly the same as the inductive limit topology with respect to $\{\mathcal{E}_K(\Omega)\}_{K\in
\mathcal{K}(\Omega)}$.

\begin{remark}\lab{remfallcontintotest1}
Suppose $Y$ is a topological space and the mapping $T:
Y\rightarrow D(\Omega)$ is such that $T(Y)\subseteq
\mathcal{E}_K(\Omega)$ for some $K\in \mathcal{K}(\Omega)$. Since
$\mathcal{E}_K(\Omega)\hookrightarrow D(\Omega)$, if $T:
Y\rightarrow \mathcal{E}_K(\Omega)$ is continuous, then $T:
Y\rightarrow D(\Omega)$ will be continuous.
\end{remark}
\begin{theorem}[Convergence and Continuity for
$\mathcal{E}(\Omega)$]\lab{winter56} Let $\Omega$ be a nonempty
open set in $\reals^n$. Let $Y$ be a topological vector space
whose topology is induced by a separating family of seminorms $\mathcal{Q}$.
\begin{enumerateXALI}
\item A sequence $\{\varphi_m\}$ converges to $\varphi$ in
$\mathcal{E}(\Omega)$ if and only if
$\|\varphi_m-\varphi\|_{j,K}\rightarrow 0$ for all $j\in
\mathbb{N}$ and $K\in \mathcal{K}(\Omega)$.
\item  Suppose $T:\mathcal{E}(\Omega)\rightarrow Y$ is a linear map.
Then the followings are equivalent
\begin{itemize}
\item $T$ is continuous.
\item For every $q\in\mathcal{Q}$, there exist $j\in \mathbb{N}$ and $K\in\mathcal{K}(\Omega)$, and $C>0$ such that
\begin{equation*}
\forall\,\varphi\in \mathcal{E}(\Omega)\qquad q(T(\varphi))\leq C\|\varphi\|_{j,K}
\end{equation*}
\item If $\varphi_m\rightarrow 0$ in $\mathcal{E}(\Omega)$, then $T(\varphi_m)\rightarrow 0$ in $Y$.
\end{itemize}
\item In particular, a linear map $T:\mathcal{E}(\Omega)\rightarrow
\reals$ is continuous if and only if there exist $j\in \mathbb{N}$
and $K\in\mathcal{K}(\Omega)$, and $C>0$ such that
\begin{equation*}
\forall\,\varphi\in \mathcal{E}(\Omega)\qquad |T(\varphi)|\leq
C\|\varphi\|_{j,K}
\end{equation*}
\item A linear map $T: Y\rightarrow \mathcal{E}(\Omega)$ is continuous if and only if
{\fontsize{10}{10}{\begin{equation*}
\forall\, j\in \mathbb{N},\,\forall\, K\in
\mathcal{K}(\Omega)\qquad \exists\,C>0,\,k\in \mathbb{N}\,,
q_1,\cdots,q_k \in \mathcal{Q}\quad \textrm{such that $\forall\,y$} \quad
\|T(y)\|_{j,K}\leq C \max_{1\leq i\leq k} q_i(y)
\end{equation*}}}
\end{enumerateXALI}
\end{theorem}
\begin{theorem}[Convergence and Continuity for
$\mathcal{E}_K(\Omega)$]\lab{winter57} Let $\Omega$ be a nonempty
open set in $\reals^n$ and $K\in \mathcal{K}(\Omega)$. Let $Y$ be
a topological vector space whose topology is induced by a separating family
of seminorms $\mathcal{Q}$.
\begin{enumerateXALI}
\item A sequence $\{\varphi_m\}$ converges to $\varphi$ in
$\mathcal{E}_K(\Omega)$ if and only if
$\|\varphi_m-\varphi\|_{j,K}\rightarrow 0$ for all $j\in
\mathbb{N}$.
\item Suppose $T:\mathcal{E}_K(\Omega)\rightarrow Y$ is a linear map.
Then the followings are equivalent
\begin{itemize}
\item $T$ is continuous.
\item For every $q\in\mathcal{Q}$, there exists $j\in \mathbb{N}$ and $C>0$ such that
\begin{equation*}
\forall\,\varphi\in \mathcal{E}_K(\Omega)\qquad q(T(\varphi))\leq
C\|\varphi\|_{j,K}
\end{equation*}
\item If $\varphi_m\rightarrow 0$ in $\mathcal{E}_K(\Omega)$, then $T(\varphi_m)\rightarrow 0$ in $Y$.
\end{itemize}
\end{enumerateXALI}
\end{theorem}
\begin{theorem}[Convergence and Continuity for $D(\Omega)$]\lab{thmfallconvcont13}
Let $\Omega$ be a nonempty open set in $\reals^n$. Let $Y$ be a topological vector space whose topology is induced by a separating family of seminorms $\mathcal{Q}$.
\begin{enumerateXALI}
\item A sequence $\{\varphi_m\}$ converges to $\varphi$ in
$D(\Omega)$ if and only if there is a $K\in\mathcal{K}(\Omega)$
such that $\textrm{supp} \varphi_m\subseteq K$ and
$\varphi_m\rightarrow \varphi$ in $\mathcal{E}_K(\Omega)$.
\item Suppose $T:D(\Omega)\rightarrow Y$ is a linear map.
Then the followings are equivalent
\begin{itemize}
\item $T$ is continuous.
\item For all $K\in \mathcal{K}(\Omega)$, $T: \mathcal{E}_K(\Omega)\rightarrow
Y$ is continuous.
\item For every $q\in\mathcal{Q}$ and $K\in \mathcal{K}(\Omega)$, there exists $j\in \mathbb{N}$ and $C>0$ such that
\begin{equation*}
\forall\,\varphi\in \mathcal{E}_K(\Omega)\qquad q(T(\varphi))\leq
C\|\varphi\|_{j,K}
\end{equation*}
\item If $\varphi_m\rightarrow 0$ in $D(\Omega)$, then $T(\varphi_m)\rightarrow 0$ in $Y$.
\end{itemize}
\item In particular, a linear map $T: D(\Omega)\rightarrow \reals$
 is continuous if and only if for every $K\in \mathcal{K}(\Omega)$, there exists $j\in \mathbb{N}$ and $C>0$ such that
\begin{equation*}
\forall\,\varphi\in \mathcal{E}_K(\Omega)\qquad |T(\varphi)|\leq
C\|\varphi\|_{j,K}
\end{equation*}
\end{enumerateXALI}
\end{theorem}
\begin{remark}\lab{remfallcontintotest2}
Let $\Omega$ be a nonempty open set in $\reals^n$. Here are two immediate consequences of the previous theorems and remark:
\begin{enumerateXALI}
\item The identity map
\begin{equation*}
i_{D,\mathcal{E}}:D(\Omega)\rightarrow \mathcal{E}(\Omega)
\end{equation*}
is continuous (that is, $D(\Omega)\hookrightarrow \mathcal{E}(\Omega)$ ).
\item If $T:\mathcal{E}(\Omega)\rightarrow \mathcal{E}(\Omega)$ is a continuous linear map such that $\textrm{supp}(T \varphi)\subseteq \textrm{supp}\varphi$ for all $\varphi\in \mathcal{E}(\Omega)$ (i.e. $T$ is a \textbf{local} continuous linear map), then $T$ restricts to a continuous linear map from $D(\Omega)$ to $D(\Omega)$. Indeed, the assumption $\textrm{supp}(T \varphi)\subseteq \textrm{supp}\varphi$ implies that $T(D(\Omega))\subseteq D(\Omega)$. Moreover $T:D(\Omega)\rightarrow D(\Omega)$ is continuous if and only if for $K\in \mathcal{K}(\Omega)$ $T: \mathcal{E}_K(\Omega)\rightarrow D(\Omega)$ is continuous. Since $T(\mathcal{E}_K(\Omega))\subseteq \mathcal{E}_K(\Omega)$, this map is continuous if and only if $T: \mathcal{E}_K(\Omega)\rightarrow \mathcal{E}_K(\Omega)$ is continuous (see Remark \ref{remfallcontintotest1}). However, since the topology of $\mathcal{E}_K(\Omega)$ is the induced topology from $\mathcal{E}(\Omega)$, the continuity of the preceding map follows from the continuity of $T:\mathcal{E}(\Omega)\rightarrow \mathcal{E}(\Omega)$.
\end{enumerateXALI}
\end{remark}

\begin{theorem}
Let $\Omega$ be a nonempty open set in $\reals^n$. Then $D(\Omega)$ is separable.
\end{theorem}

\begin{definition}\lab{winter60}
Let $\Omega$ be a nonempty open set in $\reals^n$. The topological
dual of $D(\Omega)$, denoted $D'(\Omega)$ $(D'(\Omega)
=[D(\Omega)]^*)$, is called the \textbf{space of distributions} on
$\Omega$. Each element of $D'(\Omega)$ is called a \textbf{distribution} on
$\Omega$. The action of a distribution $u\in D'(\Omega)$ on a function $\varphi\in D(\Omega)$ is sometimes denoted by $\langle u,\varphi\rangle_{D'(\Omega)\times D(\Omega)}$ or simply $\langle u,\varphi\rangle$.
\end{definition}
\begin{remark}\lab{winter61}
Every function $f\in L^1_{loc}(\Omega)$ defines a distribution
$u_f\in D'(\Omega)$ as follows
\begin{equation}\lab{eqnwinter61}
\forall\, \varphi\in D(\Omega)\qquad u_f(\varphi):=\int_\Omega
f\varphi dx
\end{equation}
In particular, every function $\varphi\in \mathcal{E}(\Omega)$
defines a distribution $u_\varphi$. It can be shown that the map
$i:\mathcal{E}(\Omega)\rightarrow D'(\Omega)$ which sends
$\varphi$ to $u_\varphi$ is an injective linear continuous map
(\cite{Reus1}, Page 11). Therefore we can identify
$\mathcal{E}(\Omega)$ with a subspace of $D'(\Omega)$; we sometimes refer to the map $i$ as the "identity map".
\end{remark}

\begin{theorem}[\cite{9}, Page 47] \lab{thmmay7643} Let $\Omega$ be a nonempty open set in $\reals^n$. Equip $D'(\Omega)$ with the weak$^*$ topology. Then under the above identification, $C_c^\infty(\Omega)$ is dense in $D'(\Omega)$.
\end{theorem}

\begin{theorem}[\cite{Treves1}, Page 302]\lab{thmmay111006}
Let $\Omega$ be a nonempty open set in $\reals^n$. Equip $D'(\Omega)$ with the strong topology. Then under the identification described in Remark ~\ref{winter61}, $C_c^\infty(\Omega)$ is sequentially dense in $D'(\Omega)$.
\end{theorem}

\begin{remark}\lab{remmay111008}
\leavevmode
\begin{itemizeX}
\item Clearly sequential density is a stronger notion than density. So $C_c^\infty(\Omega)$ is dense in $(D'(\Omega), \textrm{strong topology})$.
\item Recall that, according to Remark ~\ref{winter32}, a sequence converges in $(D'(\Omega), \textrm{weak$^*$})$ if and only if it converges in $(D'(\Omega), \textrm{strong topology})$. This together with the fact that weak$^*$ topology is weaker than the strong topology implies that convergent sequences in both topologies converge to the same limit. Therefore it follows from Theorem ~\ref{thmmay111006} that $C_c^\infty(\Omega)$ is sequentially dense in $(D'(\Omega), \textrm{weak$^*$})$. Hence Theorem ~\ref{thmmay7643} can be viewed as a corollary of Theorem ~\ref{thmmay111006}.
\end{itemizeX}
\end{remark}

\begin{theorem}[\cite{Reus1}, Page 9]\lab{thmmay111020}
$D(\Omega)$ is reflexive. So $[(D'(\Omega), \textrm{strong topology})]^*$ can be identified with the topological vector space $D(\Omega)$.
\end{theorem}

\begin{definition}[Restriction of a Distribution]\lab{winter64}
Let $\Omega$ be an open subset of $\reals^n$ and $V$ be an open
susbset of $\Omega$. We define the restriction map
$\textrm{res}_{\Omega,V}: D'(\Omega)\rightarrow D'(V)$ as follows
\begin{equation*}
\langle \textrm{res}_{\Omega,V} u,\varphi \rangle_{D'(V)\times
D(V)}:=\langle u,
\textrm{ext}_{V,\Omega}^0\varphi\rangle_{D'(\Omega)\times
D(\Omega)}
\end{equation*}
This is well-defined; indeed, $\textrm{res}_{\Omega,V}:
D'(\Omega)\rightarrow D'(V)$ is a continuous linear map as it is
the adjoint of the continuous map
$\textrm{ext}_{V,\Omega}^0:D(V)\rightarrow D(\Omega)$. Given
$u\in D'(\Omega)$, we sometimes write $u|_V$ instead of
$\textrm{res}_{\Omega,V} u$.
\end{definition}

\begin{definition}[Support of a Distribution]\lab{winter66}
Let $\Omega$ be a nonempty open set in $\reals^n$. Let $u\in
D'(\Omega)$.
\begin{itemizeXALI}
\item We say $u$ is equal to zero on some open subset $V$ of
$\Omega$ if $u|_V=0$.
\item Let $\{V_i\}_{i\in I}$ be the collection of all open subsets of
$\Omega$ such that $u$ is equal to zero on $V_i$. Let
$V=\bigcup_{i\in I} V_i$. The support of $u$ is defined as follows
\begin{equation*}
\textrm{supp}\,u:=\Omega\setminus V
\end{equation*}
Note that $\textrm{supp}\,u$ is closed in $\Omega$ but it is not
necessarily closed in $\reals^n$.
\end{itemizeXALI}
\end{definition}

\begin{theorem}\cite{Reus1}\lab{thmmay31127}
Let $\Omega$ be a nonempty open set in $\reals^n$ and let $u\in D'(\Omega)$. If $\varphi\in
D(\Omega)$ vanishes on a neighborhood containing $\textrm{supp}\, u$,
then $u(\varphi)=0$.
\end{theorem}

\begin{theorem}\cite{Reus1}\lab{thmmay11130}
Let $\{u_i\}$ be a sequence in $D'(\Omega)$, $u\in D(\Omega)$, and $K\in\mathcal{K}(\Omega)$ such that $u_i\rightarrow u$ in $D'(\Omega)$ and $\textrm{supp}\,u_i\subseteq K$ for all $i$. Then also $\textrm{supp}\,u\subseteq K$.
\end{theorem}

\begin{theorem}[\cite{Blanchard2003}, Page 38]\lab{thmapril141} Let $\Omega$ be a nonempty open set in $\reals^n$. Suppose that $\{T_i\}$ is a sequence in $D'(\Omega)$ with the property that for all $\varphi\in D(\Omega)$, $\lim_{i\rightarrow \infty}\langle T_i,\varphi\rangle _{D'(\Omega)\times D(\Omega)}$ exists. Then there exists $T\in D'(\Omega)$ such that
\begin{equation*}
\forall\,\varphi\in D(\Omega)\qquad  \langle T,\varphi\rangle _{D'(\Omega)\times D(\Omega)}=\lim_{i\rightarrow \infty}\langle T_i,\varphi\rangle _{D'(\Omega)\times D(\Omega)}
\end{equation*}
\end{theorem}

\begin{definition}[Sobolev-Slobodeckij spaces]\lab{defa1}
Let $\Omega$ be a nonempty open set in $\reals^n$. Let $s\in \reals$ and $p\in (1,\infty)$.
\begin{itemize}
\item If $s=k\in \mathbb{N}_0$,
\begin{equation*}
W^{k,p}(\Omega)=\{u\in L^p (\Omega):
\|u\|_{W^{k,p}(\Omega)}:=\sum_{|\nu|\leq
k}\|\partial^{\nu}u\|_{L^p(\Omega)}<\infty\}
\end{equation*}
\item If $s=\theta\in(0,1)$,
\begin{equation*}
W^{\theta,p}(\Omega)=\{u\in L^p (\Omega):
 |u|_{W^{\theta,p}(\Omega)}:=\big(\int\int_{\Omega\times
\Omega}\frac{|u(x)-u(y)|^p}{|x-y|^{n+\theta p}}dx
dy\big)^{\frac{1}{p}} <\infty\}
\end{equation*}
\item If $s=k+\theta,\, k\in \mathbb{N}_0,\, \theta\in(0,1)$,
\begin{equation*}
W^{s,p}(\Omega)=\{u\in
W^{k,p}(\Omega):\|u\|_{W^{s,p}(\Omega)}:=\|u\|_{W^{k,p}(\Omega)}+\sum_{|\nu|=k}
|\partial^{\nu}u|_{W^{\theta,p}(\Omega)}<\infty\}
\end{equation*}
\item $W^{s,p}_0(\Omega)$ is defined as the closure of $C_c^\infty(\Omega)$ in $W^{s,p}(\Omega)$.
\item If $s<0$,
\begin{equation*}
W^{s,p}(\Omega)=(W^{-s,p'}_{0}(\Omega))^{*} \quad
(\frac{1}{p}+\frac{1}{p'}=1)
\end{equation*}
\item For all compact sets $K\subset \Omega$ we define
\begin{equation*}
W^{s,p}_K(\Omega)=\{u\in W^{s,p}(\Omega):
\textrm{supp}\,u\subseteq K\}
\end{equation*}
with $\|u\|_{W^{s,p}_K(\Omega)}:=\|u\|_{W^{s,p}(\Omega)}$. Note that for $s<0$, $W^{s,p}(\Omega)$ can be viewed as a subspace of $D'(\Omega)$ (see Theorem ~\ref{thmmay7701}) and the support of $u\in W^{s,p}(\Omega)$ is interpreted as the support of a distribution.
\item $ W^{s,p}_{comp}(\Omega):=\bigcup_{K\in\mathcal{K}(\Omega)}W^{s,p}_K(\Omega)$.
 $W^{s,p}_{comp}(\Omega)$ is equipped with the inductive limit topology with respect to the family of vector subspaces $\{W^{s,p}_K(\Omega)\}_{K\in\mathcal{K}(\Omega)}$. It can be shown that if $\{K_j\}_{j\in \mathbb{N}_0}$ is an exhaustion by compacts sets of $\Omega$, then the inductive limit topology on $W^{s,p}_{comp}(\Omega)$ with respect to the family $\{W^{s,p}_{K_j}(\Omega)\}_{j\in \mathbb{N}_0}$ is exactly the same as the inductive limit topology with respect to $\{W^{s,p}_K(\Omega)\}_{K\in\mathcal{K}(\Omega)}$.
\end{itemize}
\end{definition}

\begin{theorem}\lab{thmmay141029}
Let $\Omega$ be a nonempty open set in $\reals^n$, $s\geq 1$ and $1<p<\infty$. Then $u\in W^{s,p}(\Omega)$ if and only if $u\in L^p(\Omega)$ and for all $1\leq i\leq n$, $\displaystyle \frac{\partial u}{\partial x^i}\in W^{s-1,p}(\Omega)$.
\end{theorem}

\begin{proof}
We consider two cases:
\begin{itemizeX}
\item \textbf{Case 1: $s=k\in \mathbb{N}$}
\begin{align*}
u\in W^{k,p}(\Omega)&\Longleftrightarrow \textrm{$u\in L^p(\Omega)$ and $\partial^\alpha u\in L^p(\Omega)\quad \forall\, 1\leq |\alpha|\leq k$}\\
&\Longleftrightarrow \textrm{$u\in L^p(\Omega)$ and $\partial^\beta \big[\frac{\partial u}{\partial x^i}\big]\in L^p(\Omega)\quad \forall\, 0\leq |\beta|\leq k-1,\,1\leq i\leq n$}\\
&\Longleftrightarrow \textrm{$u\in L^p(\Omega)$ and $\frac{\partial u}{\partial x^i}\in W^{k-1,p}(\Omega)
\quad \forall\, 1\leq i\leq n$}
\end{align*}
\item \textbf{Case 2: $s=k+\theta$, $k\in \mathbb{N}$, $0<\theta<1$}
{\fontsize{9}{10}{\begin{align*}
&u\in W^{s,p}(\Omega)\Longleftrightarrow \textrm{$u\in W^{k,p}(\Omega)$ and $\frac{\partial^\nu u(x)-\partial^\nu u(y)}{|x-y|^{\frac{n}{p}+\theta}}\in L^p(\Omega\times \Omega)\quad \forall\, |\nu|= k$}\\
&\Longleftrightarrow \textrm{$u\in L^{p}(\Omega)$ and $\frac{\partial u}{\partial x^i}\in W^{k-1,p}(\Omega)
\quad \forall\, 1\leq i\leq n$ and $\frac{\partial^\nu u(x)-\partial^\nu u(y)}{|x-y|^{\frac{n}{p}+\theta}}\in L^p(\Omega\times \Omega)\quad \forall\, |\nu|= k$}\\
&\Longleftrightarrow \textrm{$u\in L^{p}(\Omega)$ and $\frac{\partial u}{\partial x^i}\in W^{k-1,p}(\Omega)
$ and $\frac{\partial^\beta \frac{\partial u}{\partial x^i}(x)-\partial^\beta \frac{\partial u}{\partial x^i}(y)}{|x-y|^{\frac{n}{p}+\theta}}\in L^p(\Omega\times \Omega)\quad \forall\, |\beta|= k-1\,\forall\, 1\leq i\leq n$}\\
&\Longleftrightarrow \textrm{$u\in L^p(\Omega)$ and $\frac{\partial u}{\partial x^i}\in W^{s-1,p}(\Omega)
\quad \forall\, 1\leq i\leq n$}
\end{align*}}}
\end{itemizeX}
\end{proof}

\begin{remark}\lab{remmay7646}
Let $\Omega$ be a nonempty open set in $\reals^n$, $s\in\reals$ and $1<p<\infty$. Clearly for $s\geq 0$, $C_c^\infty(\Omega)\subseteq W^{s,p}(\Omega)$. For $s<0$, it is easy to see that for all $\varphi\in C_c^\infty(\Omega)$, the map $l_\varphi: W^{-s,p'}_0(\Omega)\rightarrow \reals$ which sends $u\in W^{-s,p'}_0(\Omega)$ to $\int_\Omega u\varphi\,dx$ belongs to $[W^{-s,p'}_0(\Omega)]^*=W^{s,p}(\Omega)$. The map $\varphi\mapsto l_\varphi$ is one-to-one and we can use it to identify $C_c^\infty(\Omega)$ with a subspace of $W^{s,p}(\Omega)$; we sometimes refer to the map that sends $\varphi$ to $l_\varphi$ as the "identity map". So we can talk about the identity map from $C_c^\infty(\Omega)$ to $W^{s,p}(\Omega)$ for all $s\in \reals$.
\end{remark}

\begin{theorem}\cite{holstbehzadan2018b}\lab{thmmay7658}
Let $\Omega$ be a nonempty open set in $\reals^n$, $s\geq 0$, and $1<p<\infty$. Then $W^{s,p}(\Omega)$ is a reflexive Banach space.
\end{theorem}

\begin{corollary}\lab{cormay7659}
Let $\Omega$ be a nonempty open set in $\reals^n$, $s\geq 0$, and $1<p<\infty$. A closed subspace of a reflexive space is reflexive, so $W^{s,p}_0(\Omega)$ is reflexive. Dual of a reflexive Banach space is a reflexive Banach space, so $W^{-s,p'}(\Omega)$ is a reflexive Banach space.
\end{corollary}

\begin{remark}\lab{remmay77}
Let $\Omega$ be a nonempty open set in $\reals^n$, $s\geq 0$, and $1<p<\infty$. Since $W^{s,p}_0(\Omega)$ is reflexive, it can be identified with $[W^{s,p}_0(\Omega)]^{**}$ and we may write $[W^{-s,p'}(\Omega)]^*=W^{s,p}_0(\Omega)$ and talk about the duality pairing $\langle u, f \rangle_{W^{s,p}_0(\Omega)\times W^{-s,p'}(\Omega)}$. To be more precise, we notice that, the identification of $[W^{s,p}_0(\Omega)]^{**}$ and $W^{s,p}_0(\Omega)$ is done by the evaluation map
\begin{equation*}
J: W^{s,p}_0(\Omega)\rightarrow [W^{s,p}_0(\Omega)]^{**}\qquad J(u)[f]=f(u)
\end{equation*}
Therefore for all $u\in W^{s,p}_0(\Omega)$ and $f\in W^{-s,p'}(\Omega)$
\begin{equation*}
\langle u, f \rangle_{W^{s,p}_0(\Omega)\times W^{-s,p'}(\Omega)}=
\langle J(u), f \rangle_{[W^{s,p}_0(\Omega)]^{**}\times W^{-s,p'}(\Omega)}=f(u)=
\langle f, u\rangle_{W^{-s,p'}(\Omega)\times W^{s,p}_0(\Omega)}
\end{equation*}
\end{remark}

\begin{theorem}\lab{thmmay7647}
Let $\Omega$ be a nonempty open set in $\reals^n$, $s\geq 0$, and $1<p<\infty$. Then $C_c^\infty(\Omega)$ is dense in $W^{-s,p'}(\Omega)$. We may write this as $W^{-s,p'}_0(\Omega)=W^{-s,p'}(\Omega)$.
\end{theorem}

\begin{proof}
Our proof will be based on a similar argument given in page 65 of \cite{32}. Let $\varphi\mapsto l_\varphi$ be the mapping introduced in Remark ~\ref{remmay7646}. Our goal is to show that the set
\begin{equation*}
V:=\{l_\varphi: \varphi\in C_c^\infty(\Omega)\}
\end{equation*}
is dense in $W^{-s,p'}(\Omega)$. To this end it is enough to show that if $F\in [W^{-s,p'}(\Omega)]^*$ is such that $F(l_\varphi)=0$ for all $\varphi\in C_c^\infty(\Omega)$, then $F=0$. Indeed, let $F$ be such an element. By reflexivity of $W^{s,p}_0(\Omega)$ there exists $f\in W^{s,p}_0(\Omega)$ such that
\begin{equation*}
\forall\,v\in W^{-s,p'}(\Omega)\qquad F(v)=v(f)
\end{equation*}
Thus for all $\varphi\in C_c^\infty(\Omega)$ we have
\begin{equation*}
0=F(l_\varphi)=l_\varphi(f)=\int_\Omega f(x)\varphi(x)\,dx
\end{equation*}
So by the fundamental lemma of the calculus of variations (see \cite{Brezis2011}, Page 110) we have $f=0$ (as an element of $W^{s,p}(\Omega)\subseteq L^1_{loc}(\Omega)$) and therefore $F=0$.
\end{proof}

\begin{theorem}\lab{thmmay7701}
Let $\Omega$ be a nonempty open set in $\reals^n$, $s\in \reals$, and $1<p<\infty$. Equip $D'(\Omega)$ with weak$^{*}$ topology or strong topology. Then
\begin{equation*}
D(\Omega)\hookrightarrow W^{s,p}(\Omega)\hookrightarrow D'(\Omega)
\end{equation*}
\end{theorem}
\begin{proof}
Recall that the convergent sequences in $D'(\Omega)$ equipped with strong topology are exactly the same as the convergent sequence of $D'(\Omega)$ equipped with the weak$^*$ topology (see Remark ~\ref{winter32}). This together with the Theorem ~\ref{thmapptvconvergence2} imply that in the study of the continuity of the inclusion map from $W^{s,p}(\Omega)$ to $D'(\Omega)$, it does not matter whether we equip $D'(\Omega)$ with the strong topology or weak$^*$ topology. In the proof, as usual, we assume $D'(\Omega)$ is equipped with the strong topology. We consider two cases:
\begin{itemizeX}
\item \textbf{Case 1: $s\geq 0$} The continuity of the embedding $D(\Omega)\hookrightarrow W^{s,p}(\Omega)$ has been studied in \cite{holstbehzadan2018b}. Also clearly $W^{s,p}(\Omega)\hookrightarrow L^p(\Omega)\hookrightarrow D'(\Omega)$. The former continuous embedding holds by the definition of $W^{s,p}(\Omega)$ and the latter embedding is continuous because if $u_m\rightarrow 0$ in $L^p(\Omega)$, then for all $\varphi\in D(\Omega)$
    \begin{equation*}
    |\langle u_m,\varphi\rangle_{D'(\Omega)\times D(\Omega)}-0|=|\int_\Omega u_m\varphi\,dx|\leq \|u_m\|_p\|\varphi\|_\infty\rightarrow 0
    \end{equation*}
    So $u_m\rightarrow 0$ in $D'(\Omega)$. This implies the continuity of the inclusion map from $L^p(\Omega)$ to $D'(\Omega)$ by Theorem ~\ref{thmapptvconvergence2}.
\item \textbf{Case 2: $s<0$} Since $W^{-s,p'}_0(\Omega)\hookrightarrow W^{-s,p'}(\Omega)$, it follows from previous case that $W^{-s,p'}_0(\Omega)\hookrightarrow D'(\Omega)$. Also since $D(\Omega)\subseteq W^{-s,p'}_0(\Omega)$ is dense in $D'(\Omega)$ (see Theorem ~\ref{thmmay7643}, Theorem ~\ref{thmmay111006}, and Remark ~\ref{remmay111008}), it follows that the inclusion map from $W^{-s,p'}_0(\Omega)$ to $D'(\Omega)$ is continuous with dense image. Thus, by Theorem ~\ref{thmfallinjectiveadjoint1}, $D(\Omega)\hookrightarrow W^{s,p}(\Omega)$. Here we used the facts that 1) the strong dual of the normed space $W^{-s,p'}_0(\Omega)$ is $W^{s,p}(\Omega)$ and that 2) the dual of $(D'(\Omega),\textrm{strong topology})$ is $D(\Omega)$ (see Theorem ~\ref{thmmay111020}).
      It remains to show that $W^{s,p}(\Omega)\hookrightarrow D'(\Omega)$. It follows from Case 1 that $D(\Omega)\hookrightarrow W^{-s,p'}_0(\Omega)$ and by definition $D(\Omega)$ is dense in $W^{-s,p'}_0(\Omega)$. So, by Theorem ~\ref{thmfallinjectiveadjoint1}, $W^{s,p}(\Omega)\hookrightarrow D'(\Omega)$.
\end{itemizeX}
\end{proof}

\begin{remark}\lab{remapril28610}
Note that for $s\leq 0$, $W^{s,p}_0(\Omega)$ is the same as $W^{s,p}(\Omega)$. For $s> 0$, $W^{s,p}_0(\Omega)$ is a subspace of $W^{s,p}(\Omega)$ which contains $C_c^\infty(\Omega)$. So it follows from the previous theorem that
\begin{equation*}
D(\Omega)\hookrightarrow W^{s,p}_0(\Omega)\hookrightarrow D'(\Omega)
\end{equation*}
To be more precise we should note that for $s<0$, we identify $\varphi\in D(\Omega)$ with the corresponding distribution in $D'(\Omega)$. Under this identification, for all $s\in \reals$ the "identity map" $i: D(\Omega)\rightarrow W^{s,p}_0(\Omega)$ is continuous with dense image and so its adjoint $i^*: [W^{s,p}_0(\Omega)]^*\rightarrow D'(\Omega)$ will be an injective continuous map (Theorem ~\ref{thmfallinjectiveadjoint1}) and we have
\begin{equation*}
\langle i^* u,\varphi\rangle_{D'(\Omega)\times D(\Omega)}=\langle u,i\,\varphi\rangle_{[W^{s,p}_{0}(\Omega)]^*\times W^{s,p}_{0}(\Omega)}=\langle u,\varphi\rangle_{[W^{s,p}_{0}(\Omega)]^*\times W^{s,p}_{0}(\Omega)}
\end{equation*}
We usually identify $[W^{s,p}_{0}(\Omega)]^*$ with its image under $i^*$ and view $[W^{s,p}_{0}(\Omega)]^*$ as a subspace of $D'(\Omega)$. So, under this identification, we can rewrite the above equality as follows:
\begin{equation*}
\forall\, u\in [W^{s,p}_{0}(\Omega)]^*\,\, \forall\,\varphi\in D(\Omega)\qquad \langle  u,\varphi\rangle_{D'(\Omega)\times D(\Omega)}=
\langle u,\varphi\rangle_{[W^{s,p}_{0}(\Omega)]^*\times W^{s,p}_{0}(\Omega)}
\end{equation*}
Finally noting that for all $s\in \reals$ and $1<p<\infty$, $[W^{s,p}_{0}(\Omega)]^*=W^{-s,p'}_0(\Omega)$ (see Definition ~\ref{defa1}, Theorem ~\ref{thmmay7647}, and Corollary ~\ref{cormay7659}), we can write
\begin{equation*}
\forall\, u\in W^{-s,p'}_{0}(\Omega)\,\, \forall\,\varphi\in D(\Omega)\qquad \langle  u,\varphi\rangle_{D'(\Omega)\times D(\Omega)}=
\langle u,\varphi\rangle_{W^{-s,p'}_{0}(\Omega)\times W^{s,p}_{0}(\Omega)}
\end{equation*}
\end{remark}

\begin{theorem}\lab{thmapril271}
Let $\Omega$ be a nonempty open set in $\reals^n$, $s\geq 0$, and $1<p<\infty$. Then
\begin{enumerate}
\item The mapping $F\mapsto F|_{C_c^\infty(\Omega)}$ is an isometric isomorphism between $W^{-s,p'}(\Omega)$ and $[C_c^\infty(\Omega),\|.\|_{s,p}]^*$.
\item Suppose $u\in D'(\Omega)$. If $u: (C_c^\infty(\Omega),\|.\|_{-s,p'})\rightarrow \reals$ is continuous, then $u\in W^{s,p}_0(\Omega)$ (more precisely, there is a unique element in $W^{s,p}_0(\Omega)$ whose corresponding distribution is $u$). Moreover,
    \begin{equation*}
    \|u\|_{W^{s,p}_0(\Omega)}=\sup_{0\not\equiv\varphi\in C_c^\infty(\Omega)}\frac{\langle u,\varphi\rangle_{D'(\Omega)\times D(\Omega)}}{\|\varphi\|_{W^{-s,p'}(\Omega)}}
    \end{equation*}
\end{enumerate}
\end{theorem}

\begin{proof}
The first item has been studied in \cite{holstbehzadan2018b}. Here we will prove the second item. Since $u: (C_c^\infty(\Omega),\|.\|_{-s,p'})\rightarrow \reals$ is continuous, it can be extended to a continuous linear map $\tilde{u}: W^{-s,p'}(\Omega)\rightarrow \reals$. So $\tilde{u}\in [W^{-s,p'}(\Omega)]^*$. However, $W_0^{s,p}(\Omega)$ is reflexive, therefore there exists a unique $v\in W^{s,p}_0(\Omega)$ such that $\tilde{u}=J(v)$ where $J(v): W^{-s,p'}(\Omega)\rightarrow \reals$ is the evaluation map defined by $J(v)(F)=\langle F,v\rangle_{W^{-s,p'}(\Omega)\times W_0^{s,p}(\Omega)}$. To finish the proof, it is enough to show that $v=u$ as elements of $D'(\Omega)$. For all $\varphi\in C_c^\infty(\Omega)$ we have
\begin{align*}
\langle v,\varphi\rangle_{D'(\Omega)\times D(\Omega)} &= \int_\Omega v\varphi\,dx\stackrel{\textrm{Remark ~\ref{remmay7646}}}{=}\langle\varphi, v \rangle_{W^{-s,p'}(\Omega)\times W_0^{s,p}(\Omega)}\\
&=J(v)(\varphi)=\tilde{u}(\varphi)=u(\varphi)=\langle u,\varphi\rangle_{D'(\Omega)\times D(\Omega)}
\end{align*}
Also
\begin{align*}
\|u\|_{W^{s,p}_0(\Omega)}&=\|v\|_{W^{s,p}_0(\Omega)}=\|J(v)\|_{[W^{-s,p'}(\Omega)]^*}\\
&= \|\tilde{u}\|_{[W^{-s,p'}(\Omega)]^*}=\sup_{0\not\equiv\varphi\in C_c^\infty(\Omega)}\frac{\langle \tilde{u},\varphi\rangle_{D'(\Omega)\times D(\Omega)}}{\|\varphi\|_{W^{-s,p'}(\Omega)}}\\
&= \sup_{0\not\equiv\varphi\in C_c^\infty(\Omega)}\frac{\langle u,\varphi\rangle_{D'(\Omega)\times D(\Omega)}}{\|\varphi\|_{W^{-s,p'}(\Omega)}}
\end{align*}
\end{proof}

\begin{corollary}\lab{corapril2812}
Let $\Omega$ be a nonempty open set in $\reals^n$, $s\geq 0$, and $1<p<\infty$. Suppose that $u\in D'(\Omega)$. As a direct consequence of Theorem ~\ref{thmapril271} we have
\begin{itemizeX}
\item If $\displaystyle \sup_{0\not\equiv\varphi\in C_c^\infty(\Omega)}\frac{\langle u,\varphi\rangle_{D'(\Omega)\times D(\Omega)}}{\|\varphi\|_{W^{s,p}(\Omega)}}<\infty$, then $u\in W^{-s,p'}(\Omega)$
     and
     \begin{equation*}
      \|u\|_{W^{-s,p'}(\Omega)}=\sup_{0\not\equiv\varphi\in C_c^\infty(\Omega)}\frac{\langle u,\varphi\rangle_{D'(\Omega)\times D(\Omega)}}{\|\varphi\|_{W^{s,p}(\Omega)}}\,.
     \end{equation*}
\item If $\displaystyle \sup_{0\not\equiv\varphi\in C_c^\infty(\Omega)}\frac{\langle u,\varphi\rangle_{D'(\Omega)\times D(\Omega)}}{\|\varphi\|_{W^{-s,p'}(\Omega)}}<\infty$, then $u\in W^{s,p}_0(\Omega)$
    and
    \begin{equation*}
     \|u\|_{W^{s,p}(\Omega)}=\sup_{0\not\equiv\varphi\in C_c^\infty(\Omega)}\frac{\langle u,\varphi\rangle_{D'(\Omega)\times D(\Omega)}}{\|\varphi\|_{W^{-s,p'}(\Omega)}}\,.
     \end{equation*}
\end{itemizeX}
That is, for any $e\in \reals$ and $1<q<\infty$ in order to show that $u\in D'(\Omega)$ belongs to $W^{e,q}_0(\Omega)$, it is enough to prove that
\begin{equation*}
\sup_{0\not\equiv\varphi\in C_c^\infty(\Omega)}\frac{\langle u,\varphi\rangle_{D'(\Omega)\times D(\Omega)}}{\|\varphi\|_{W^{-e,q'}(\Omega)}}<\infty
\end{equation*}
and in fact $\displaystyle \|u\|_{W^{e,q}(\Omega)}=\sup_{0\not\equiv\varphi\in C_c^\infty(\Omega)}\frac{\langle u,\varphi\rangle_{D'(\Omega)\times D(\Omega)}}{\|\varphi\|_{W^{-e,q'}(\Omega)}}$.
\end{corollary}

\begin{theorem}
Let $\Omega$ be a nonempty open set in $\reals^n$, $s\in \reals$, and $1<p<\infty$. Suppose that $K\in \mathcal{K}(\Omega)$. Then $W^{s,p}_K(\Omega)$ is a closed subspace of $W^{s,p}(\Omega)$.
\end{theorem}

\begin{proof}
It is enough to show that if $\{u_i\}$ is a sequence of elements in $W^{s,p}_{K}(\Omega)$ such that $u_i\rightarrow u$ in $W^{s,p}(\Omega)$, then $u\in W^{s,p}_K(\Omega)$, i.e., $\textrm{supp}\, u\subseteq K$. By Theorem ~\ref{thmmay7701}, we have $u_i\rightarrow u$ in $D'(\Omega)$. Now it follows from Theorem ~\ref{thmmay11130} that $\textrm{supp}\,u\subseteq K$. Note that for any $s\geq 0$, we have $W^{s,p}(\Omega)\subseteq L^p(\Omega)\subseteq L^1_{loc}(\Omega)$; in this proof we implicitly used the fact that for functions in $L^1_{loc}(\Omega)$, the usual definition of support agrees with the distributional definition of support.
\end{proof}

Next we list several embedding theorems for Sobolev-Slobodeckij spaces.

\begin{theorem}[\cite{36}, Section 2.8.1]
\lab{winter86} Suppose $1< p\leq q<\infty$ and $-\infty< t\leq
s<\infty$ satisfy $s-\frac{n}{p}\geq t-\frac{n}{q}$. Then
$W^{s,p}(\reals^n)\hookrightarrow W^{t,q}(\reals^n)$. In
particular, $W^{s,p}(\reals^n)\hookrightarrow W^{t,p}(\reals^n)$.
\end{theorem}

\begin{theorem}\cite{33,holstbehzadan2015b}\lab{thm3.4} Let $\Omega$ be a nonempty bounded open subset of $\mathbb{R}^n$ with Lipschitz continuous
boundary. Suppose $1\leq p, q<\infty$ ($p$ does NOT need to be
less than or equal to $q$) and $0\leq t\leq s$ satisfy
$s-\frac{n}{p}\geq t-\frac{n}{q}$. If $s\not\in \mathbb{N}_0$, additionally assume that $s\neq t$. Then
$W^{s,p}(\Omega)\hookrightarrow W^{t,q}(\Omega)$. Furthermore, if $s>t$, then the embedding $W^{s,p}(\Omega)\hookrightarrow W^{t,p}(\Omega)$ is compact.
\end{theorem}

\begin{theorem}\cite{holstbehzadan2018b}\lab{thm3.2}
Let $\Omega\subseteq \reals^n$ be an arbitrary nonempty open set.
\begin{enumerateXALI}
\item Suppose $1\leq p\leq q<\infty$
 and $0\leq t\leq s$
satisfy $s-\frac{n}{p}\geq t-\frac{n}{q}$. Then
$W^{s,p}_{K}(\Omega)\hookrightarrow W^{t,q}_{K}(\Omega)$ for all
$K\in \mathcal{K}(\Omega)$.
\item For all $k_1, k_2\in \mathbb{N}_0$  with $k_1\leq k_2$ and $1<p<\infty$,
$W^{k_2,p}(\Omega)\hookrightarrow W^{k_1,p}(\Omega)$.
\item If $0\leq t \leq s <1$ and $1<p<\infty$, then $W^{s,p}(\Omega)\hookrightarrow
W^{t,p}(\Omega)$.
\item If $0\leq t \leq s <\infty$ are such that $\floor{s}=\floor{t}$ and
$1<p<\infty$, then $W^{s,p}(\Omega)\hookrightarrow
W^{t,p}(\Omega)$.
\item If $0\leq t \leq s <\infty$, $t\in \mathbb{N}_0$, and
$1<p<\infty$, then $W^{s,p}(\Omega)\hookrightarrow
W^{t,p}(\Omega)$.
\end{enumerateXALI}
\end{theorem}

\begin{theorem}\cite{Gris85}\lab{thm3.3}
Let $\Omega$ be a nonempty bounded open subset of $\mathbb{R}^n$ with
Lipschitz continuous boundary or $\Omega=\mathbb{R}^n$. If
$sp>n$, then $W^{s,p}(\Omega)\hookrightarrow
L^{\infty}(\Omega)\cap C^{0}(\Omega)$ and $W^{s,p}(\Omega)$ is a
Banach algebra.
\end{theorem}

In the next several theorems we will list certain multiplication properties of Sobolev spaces. Suppose $\varphi\in C^\infty(\Omega)$ and $u\in W^{s,p}(\Omega)$. If $s\geq 0$, then the product $\varphi u$ has a clear meaning. What if $s<0$? In this case, $u|_{D(\Omega)}$ is a distribution and by the product $\varphi u$ we mean the distribution $(\varphi)(u|_{D(\Omega)})$; then $\varphi u$ is in $W^{s,p}(\Omega)$ if $(\varphi)(u|_{D(\Omega)}): (C_c^\infty(\Omega),\|.\|_{-s,p'})\rightarrow \reals$ is continuous. Because then it possesses a unique extension to a continuous linear map from $W^{-s,p'}_0(\Omega)$ to $\reals$ and so it can be viewed as an element of $[W^{-s,p'}_0(\Omega)]^*=W^{s,p}(\Omega)$. See Theorem ~\ref{thmapril271} and Corollary ~\ref{corapril2812}. Also see Remark ~\ref{remapril301014}.

\begin{theorem}[Multiplication by smooth functions I, \cite{Trie92}, Page
203] \lab{winter87} Let $s\in \reals$, $1<p<\infty$, and
$\varphi\in BC^\infty(\reals^n)$. Then the linear map
\begin{equation*}
m_\varphi: W^{s,p}(\reals^n)\rightarrow W^{s,p}(\reals^n),\qquad
u\mapsto \varphi u
\end{equation*}
is well-defined and bounded.
\end{theorem}

\begin{theorem}[Multiplication by smooth
functions II, \cite{holstbehzadan2018b}]\lab{thmfallmultsmooth20} Let $\Omega$ be a nonempty
bounded open set in $\reals^n$ with Lipschitz continuous boundary.
\begin{enumerateX}
\item Let $k\in \mathbb{N}_0$ and $1<p<\infty$. If $\varphi\in
BC^k(\Omega)$, then the linear map $W^{k,p}(\Omega)\rightarrow
W^{k,p}(\Omega)$ defined by $u\mapsto \varphi u$ is well-defined
and bounded.
\item Let $s\in \reals$ and $1<p<\infty$. If $\varphi\in
BC^{\infty}(\Omega)$, then the linear map
$W^{s,p}(\Omega)\rightarrow W^{s,p}(\Omega)$ defined by $u\mapsto
\varphi u$ is well-defined and bounded.
\end{enumerateX}
\end{theorem}

\begin{theorem}[Multiplication by smooth functions III, \cite{holstbehzadan2018b}] \lab{thmfallmultsmooth21}
Let $\Omega$ be any nonempty open set in $\reals^n$. Let $p\in
(1,\infty)$.
\begin{enumerateX}
\item If $0\leq s<1$ and $\varphi\in BC^{0,1}(\Omega)$ (that is, $\varphi\in L^\infty(\Omega)$ and $\varphi$ is Lipschitz), then
\begin{equation*}
m_\varphi: W^{s,p}(\Omega)\rightarrow W^{s,p}(\Omega),\qquad
u\mapsto \varphi u
\end{equation*}
is a well-defined bounded linear map.
\item If $k\in \mathbb{N}_0$ and $\varphi\in BC^k(\Omega)$, then
\begin{equation*}
m_\varphi: W^{k,p}(\Omega)\rightarrow W^{k,p}(\Omega),\qquad
u\mapsto \varphi u
\end{equation*}
is a well-defined bounded linear map.
\item If $-1<s<0$ and $\varphi\in BC^{\infty,1}(\Omega)$ or $s\in
\mathbb{Z}^-$ and $\varphi\in BC^{\infty}(\Omega)$, then
\begin{equation*}
m_\varphi: W^{s,p}(\Omega)\rightarrow W^{s,p}(\Omega),\qquad
u\mapsto \varphi u
\end{equation*}
is a well-defined bounded linear map.
\end{enumerateX}
\end{theorem}
\begin{theorem}[Multiplication by smooth functions IV, \cite{holstbehzadan2018b}]\lab{lemapp3} 
Let $\Omega$ be a nonempty open set in $\reals^n$,
$K\in\mathcal{K}(\Omega)$, $p\in (1,\infty)$, and $-1<s<0$ or $s\in\mathbb{Z}^{-}$ or $s\in [0,\infty)$. If $\varphi\in C^{\infty}(\Omega)$, then the linear
map
\begin{equation*}
W^{s,p}_{K}(\Omega)\rightarrow W^{s,p}_{K}(\Omega),\qquad u\mapsto
\varphi u
\end{equation*}
is well-defined and bounded.
\end{theorem}
\begin{theorem}[Multiplication by smooth functions V, \cite{holstbehzadan2018b}]\lab{corollarywinter92a}
Let $\Omega$ be a nonempty bounded open set in $\reals^n$ with
Lipschitz continuous boundary. Let $K\in \mathcal{K}(\Omega)$.
Suppose $s\in \reals$ and $p\in (1,\infty)$. If $\varphi\in
C^\infty(\Omega)$, then the linear map
$W^{s,p}_K(\Omega)\rightarrow W^{s,p}_K(\Omega)$ defined by
$u\mapsto \varphi u$ is well-defined and bounded.
\end{theorem}

In the next definition we introduce the notion of \emph{smooth multiplication triple} which will play a key role in several theorems that will follow.

\begin{definition}[Smooth multiplication triple]\lab{defmay7722}
Let $\Omega$ be a nonempty open set in $\reals^n$, $s\in \reals$ and $1<p<\infty$.
\begin{itemize}
\item  We say that the triple $(s,p,\Omega)$ is a \textbf{smooth multiplication triple} if for all $\varphi\in C_c^\infty(\Omega)$, the map
\begin{equation*}
m_\varphi: W^{s,p}(\Omega)\rightarrow W^{s,p}(\Omega)\qquad u\mapsto \varphi\,u
\end{equation*}
is well-defined and bounded.
\item We say that the triple $(s,p,\Omega)$ is an \textbf{interior smooth multiplication triple} if for all $\varphi\in C_c^\infty(\Omega)$ and $K\in \mathcal{K}(\Omega)$, the map
  \begin{equation*}
m_\varphi: W^{s,p}_K(\Omega)\rightarrow W^{s,p}_K(\Omega)\qquad u\mapsto \varphi\,u
\end{equation*}
is well-defined and bounded.
\end{itemize}
\end{definition}

\begin{remark}\lab{remapril161}
\leavevmode
\begin{itemizeX}
\item Every smooth multiplication triple is also an interior smooth multiplication triple.
\item It is a direct consequence of theorems ~\ref{winter87}, ~\ref{thmfallmultsmooth20}, and ~\ref{thmfallmultsmooth21} that
\begin{enumerate}
\item if $\Omega=\reals^n$ or $\Omega$ is bounded with Lipschitz continuous boundary, then for all $s\in \reals$ and $1<p<\infty$, $(s,p,\Omega)$ is a smooth multiplication triple.
    \item if $\Omega$ is any open set in $\reals^n$,  $1<p<\infty$, and $s\in \reals$ is not a noninteger with magnitude greater than $1$, then $(s,p,\Omega)$ is a smooth multiplication triple.
\end{enumerate}
\item It is a direct consequence of Theorem ~\ref{lemapp3} and Theorem ~\ref{corollarywinter92a} that
\begin{enumerate}
\item if $\Omega=\reals^n$ or $\Omega$ is bounded with Lipschitz continuous boundary, then for all $s\in \reals$ and $1<p<\infty$, $(s,p,\Omega)$ is an interior smooth multiplication triple.
    \item if $\Omega$ is any open set in $\reals^n$,  $1<p<\infty$, and $s\in \reals$ is not a noninteger less than $-1$, then $(s,p,\Omega)$ is an interior smooth multiplication triple.
\end{enumerate}
\item If $(s,p,\Omega)$ is a smooth multiplication triple and $K\in \mathcal{K}(\Omega)$, then $W^{s,p}_K(\Omega)\subseteq W^{s,p}_0(\Omega)$ (see the proof of Theorem 7.31 in \cite{holstbehzadan2018b}). Of course, if $s<0$, then $W^{s,p}(\Omega)=W^{s,p}_0(\Omega)$ and so $W^{s,p}_K(\Omega)\subseteq W^{s,p}_0(\Omega)$ holds for all $s<0$, $1<p<\infty$ and open sets $\Omega\subseteq \reals^n$.
\end{itemizeX}
\end{remark}

\begin{theorem}\lab{thmmay7724}
Let $\Omega$ be a nonempty open set in $\reals^n$, $s\geq 0$ and $1<p<\infty$. If $(s,p,\Omega)$ is a smooth multiplication triple so is $(-s,p',\Omega)$.
\end{theorem}

\begin{proof}
Let $\varphi\in C_c^\infty(\Omega)$. For all $u\in W^{-s,p'}(\Omega)=W^{-s,p'}_0(\Omega)$ and $\psi\in D(\Omega)$ we have
\begin{align*}
|\langle \varphi u,\psi\rangle_{D'(\Omega)\times D(\Omega)}|&=|\langle  u,\varphi\psi\rangle_{D'(\Omega)\times D(\Omega)}|\stackrel{\textrm{Remark ~\ref{remapril28610}}}{=}|\langle  u,\varphi\psi\rangle_{W^{-s,p'}(\Omega)\times W^{s,p}_0(\Omega)}|\\
&\leq \|u\|_{W^{-s,p'}(\Omega)}\|\varphi\psi\|_{W^{s,p}(\Omega)}\\
&\preceq \|u\|_{W^{-s,p'}(\Omega)}\|\psi\|_{W^{s,p}(\Omega)}
\end{align*}
The last inequality holds because $(s,p,\Omega)$ is a smooth multiplication triple. It follows from Corollary ~\ref{corapril2812} that $\varphi u\in W^{-s,p'}_0(\Omega)$ and $\|\varphi u\|_{W^{-s,p'}(\Omega)}\preceq \|u\|_{W^{-s,p'}(\Omega)}$, that is, $m_\varphi: W^{-s,p'}(\Omega)\rightarrow W^{-s,p'}(\Omega)$ is well-defined and continuous.
\end{proof}

\begin{theorem}\lab{thmapril181}
Let $\Omega$ be a nonempty open set in $\reals^n$, $s\in \reals$ and $1<p<\infty$. If $s<0$, further assume that $(-s,p',\Omega)$ is a smooth multiplication triple. Suppose that $\Omega'\subseteq \Omega$ and $K\in \mathcal{K}(\Omega')$. Then
    \begin{enumerate}
    \item for all $u\in W^{s,p}_K(\Omega)$, $\|u\|_{W^{s,p}(\Omega)}\simeq \|u|_{\Omega'}\|_{W^{s,p}(\Omega')}$,
    \item for all $u\in W^{s,p}_K(\Omega')$, $\|\textrm{ext}^0_{\Omega',\Omega}u\|_{W^{s,p}(\Omega)}\simeq \|u\|_{W^{s,p}(\Omega')}$.
    \end{enumerate}
\end{theorem}
\begin{proof}
The claim follows from the argument presented in the proofs of Corollary 7.39 and Theorem 7.46 in \cite{holstbehzadan2018b}.
\end{proof}

\begin{theorem} [(\cite{33}, Pages 598-605), (\cite{Gris85}, Section 1.4)]\lab{winter88} Let $s\in \reals$, $1<p<\infty$, and $\alpha\in
\mathbb{N}_0^n$. Suppose $\Omega$ is a nonempty open set in $\reals^n$. Then
\begin{enumerateX}
\item the linear operator $\partial^\alpha: W^{s,p}(\reals^n)\rightarrow
W^{s-|\alpha|,p}(\reals^n)$ is well-defined and bounded;
\item for $s<0$, the linear operator $\partial^\alpha: W^{s,p}(\Omega)\rightarrow
W^{s-|\alpha|,p}(\Omega)$ is well-defined and bounded;
\item for $s\geq 0$ and $|\alpha|\leq s$, the linear operator $\partial^\alpha: W^{s,p}(\Omega)\rightarrow
W^{s-|\alpha|,p}(\Omega)$ is well-defined and bounded;
\item if $\Omega$ is bounded with Lipschitz continuous boundary, and if $s\geq 0$, $s-\frac{1}{p}\neq \textrm{integer}$ (i.e. the fractional part of $s$ is not equal to $\frac{1}{p}$), then the linear operator $\partial^\alpha: W^{s,p}(\Omega)\rightarrow W^{s-|\alpha|,p}(\Omega)$ for $|\alpha|>s$ is well-defined and bounded.
\end{enumerateX}
\end{theorem}

\begin{theorem}\lab{thmamar191}${}$\\
\textbf{Assumptions:}
\begin{itemize}
\item $\Omega=\reals^n$ or $\Omega$ is a bounded domain with Lipschitz continuous boundary
\item $s_i, s\in \reals$,$\, s_i\geq s\geq 0$ for $i=1,2$
\item $1< p_i\leq p <\infty$  for $i=1,2$
\item $\displaystyle s_i-s\geq n(\frac{1}{p_i}-\frac{1}{p})$
\item $\displaystyle s_1+s_2-s>n(\frac{1}{p_1}+\frac{1}{p_2}-\frac{1}{p})$
\end{itemize}
\textbf{Claim:} If $u\in W^{s_1,p_1}(\Omega)$ and $v\in W^{s_2,p_2}(\Omega)$, then $uv\in W^{s,p}(\Omega)$ and moreover the pointwise multiplication of functions is a continuous bilinear map
\begin{equation*}
W^{s_1,p_1}(\Omega)\times W^{s_2,p_2}(\Omega)\rightarrow W^{s,p}(\Omega)
\end{equation*}
\end{theorem}

\begin{remark}\lab{remmar7758}
A number of other results concerning the sufficient conditions on the exponents $s_i, p_i, s, p$ that guarantee the multiplication $W^{s_1,p_1}(\Omega)\times W^{s_2,p_2}(\Omega)\hookrightarrow W^{s,p}(\Omega)$ is well-defined and continuous are discussed in detail in \cite{holstbehzadan2015b}.
\end{remark}


\begin{remark}\lab{remapril301014}  Suppose that $(s,p,\Omega)$ is a smooth multiplication triple with $s\geq 0$. $W^{-s,p'}(\Omega)=W^{-s,p'}_0(\Omega)$ is the dual of $W^{s,p}_0(\Omega)$ and $\langle u,f\rangle_{W^{-s,p'}_{0}(\Omega)\times W^{s,p}_{0}(\Omega)}$ is the action of the functional $u$ on the function $f$. As it was discussed before, if $\psi$ is a function in $C_c^\infty(\Omega)$, $(\psi) (u|_{D(\Omega)})$ is defined as a product of a smooth function and a distribution. Since $(s,p,\Omega)$ is a smooth multiplication triple, $(-s,p',\Omega)$ will also be a smooth multiplication triple, and that means $(\psi) (u|_{D(\Omega)}): (C_c^\infty(\Omega),\|.\|_{s,p})\rightarrow \reals$ is continuous (see the Note right after Theorem ~\ref{thm3.3}). We interpret $\psi u$ as an element of $W^{-s,p'}(\Omega)=[W^{s,p}_0(\Omega)]^*$ to be the unique continuous linear extension of $\psi (u|_{D(\Omega)})$ to the entire $W^{s,p}_0(\Omega)$. It is easy to see that, this unique linear extension is given by
    \begin{equation*}
    \langle \psi u,f\rangle_{W^{-s,p'}(\Omega)\times W^{s,p}_0(\Omega)}:= \langle u,\psi f\rangle_{W^{-s,p'}(\Omega)\times W^{s,p}_0(\Omega)}
    \end{equation*}
    that is, the above map is linear continuous and its restriction to $D(\Omega)$ is the same as $\psi (u|_{D(\Omega)})$. (Note that since $(s,p,\Omega)$ is a smooth multiplication triple, $\psi f$ is indeed an element of $W^{s,p}_0(\Omega)$.)
\end{remark}

\begin{theorem}\cite{38}\lab{winter106}
Let $s\in [1,\infty)$, $1<p<\infty$, and let
\begin{align*}
m=\begin{cases}
&s,\quad \textrm{if $s$ is an integer}\\
& \floor{s}+1,\quad \textrm{otherwise}
\end{cases}
\end{align*}
If $F\in C^m(\reals)$ is such that $F(0)=0$ and $F,F',\cdots,F^{(m)}\in L^{\infty}(\reals)$ (in particular, note that every $F\in C_c^\infty(\reals)$ with $F(0)=0$ satisfies these conditions), then the map $u\mapsto F(u)$ is well-defined and continuous from $W^{s,p}(\reals^n)\cap W^{1,sp}(\reals^n)$ into $W^{s,p}(\reals^n)$.
\end{theorem}
\begin{corollary}\lab{winter107}
Let $s$, $p$, and $F$ be as in the previous theorem. Moreover suppose $sp>n$. Then the map $u\mapsto F(u)$ is well-defined and continuous from $W^{s,p}(\reals^n)$ into $W^{s,p}(\reals^n)$. The reason is that when $sp>n$, we have $W^{s,p}(\reals^n)\hookrightarrow W^{1,sp}(\reals^n)$.
\end{corollary}

In the remaining of this section we will state certain useful properties of the topological vector space $W^{s,p}_{comp}$. The properties we will discuss here echo the ones stated in \cite{Petersen83} for spaces  $H^s_{comp}$.
\begin{theorem}\lab{thmmay7759}
Let $\Omega$ be a nonempty open set in $\reals^n$, $s\in \reals$, and $1<p<\infty$. Then $D(\Omega)$ is continuously embedded in $W^{s,p}_{comp}(\Omega)$.
\end{theorem}
\begin{proof}
For all $K\in\mathcal{K}(\Omega)$ we have
\begin{equation*}
\mathcal{E}_K(\Omega)\hookrightarrow D(\Omega)\hookrightarrow W^{s,p}(\Omega)
\end{equation*}
This together with the fact that the image of $\mathcal{E}_K(\Omega)$ under the identity map is inside $W^{s,p}_K(\Omega)$, implies that
\begin{equation}\lab{eqnapril261}
\mathcal{E}_K(\Omega)\hookrightarrow W^{s,p}_K(\Omega)
\end{equation}
Also, by the definition of the inductive limit topology on $W^{s,p}_{comp}(\Omega)$, we have
\begin{equation}\lab{eqnapril262}
W^{s,p}_K(\Omega)\hookrightarrow W^{s,p}_{comp}(\Omega)
\end{equation}
It follows from (~\ref{eqnapril261}) and (~\ref{eqnapril262}) that for all $K\in \mathcal{K}(\Omega)$
\begin{equation*}
\mathcal{E}_K(\Omega)\hookrightarrow W^{s,p}_{comp}(\Omega)
\end{equation*}
which, by Theorem ~\ref{winter37}, implies that $D(\Omega)\hookrightarrow W^{s,p}_{comp}(\Omega)$.
\end{proof}

\begin{theorem}\lab{thmmay78}
Let $(s,p,\Omega)$ be a smooth multiplication triple. Then $C_c^\infty(\Omega)$ is dense in $W^{s,p}_{comp}(\Omega)$.
\end{theorem}

\begin{proof}
We will follow the proof given in \cite{Petersen83} for spaces $H^s_{comp}$. Let $u\in W^{s,p}_{comp}(\Omega)$. It is enough to show that there exists a sequence in $C_c^\infty(\Omega)$ that converges to $u$ in $W^{s,p}_{comp}(\Omega)$ (this proves sequential density which implies density). By Meyers-Serrin theorem there exists a sequence $\varphi_m\in C^\infty(\Omega)\cap W^{s,p}(\Omega)$ such that $\varphi_m\rightarrow u$ in $W^{s,p}(\Omega)$. Let $\chi\in C_c^\infty(\Omega)$ be such that $\chi=1$ on a neighborhood containing $\textrm{supp}\,u$ (see Theorem ~\ref{thmmay7635}). Let $K=\textrm{supp}\,\chi$. Since $(s,p,\Omega)$ is a smooth multiplication triple, multiplication by $\chi$ is a linear continuous map on $W^{s,p}(\Omega)$ and so $\chi \varphi_m\rightarrow \chi u$ in $W^{s,p}(\Omega)$. Now we note that $\chi u=u$ and for all $m$, $\chi \varphi_m$ are in $C_c^\infty(\Omega)$ with support inside $K$. Consequently, $\chi \varphi_m \rightarrow u$ in $W^{s,p}_K(\Omega)$. Now since $W^{s,p}_K(\Omega)\hookrightarrow W^{s,p}_{comp}(\Omega)$ we may conclude that $\chi \varphi_m$ is a sequence in $C_c^\infty(\Omega)$ that converges to $u$ in $W^{s,p}_{comp}(\Omega)$.
\end{proof}

\begin{remark}\lab{remapril281}
As a consequence, if $(s,p,\Omega)$ is a smooth multiplication triple, then $[W^{s,p}_{comp}(\Omega)]^*$ (equipped with the strong topology) is continuously embedded in $D'(\Omega)$. More precisely, the identity map $i: D(\Omega)\rightarrow W^{s,p}_{comp}(\Omega)$ is continuous with dense image, and therefore, by Theorem ~\ref{thmfallinjectiveadjoint1}, the adjoint $i^*: [W^{s,p}_{comp}(\Omega)]^*\rightarrow D'(\Omega)$ is a continuous injective map. We have
\begin{equation*}
\langle i^* u,\varphi\rangle_{D'(\Omega)\times D(\Omega)}=\langle u,i\,\varphi\rangle_{[W^{s,p}_{comp}(\Omega)]^*\times W^{s,p}_{comp}(\Omega)}=\langle u,\varphi\rangle_{[W^{s,p}_{comp}(\Omega)]^*\times W^{s,p}_{comp}(\Omega)}
\end{equation*}
We usually identify $[W^{s,p}_{comp}(\Omega)]^*$ with its image under $i^*$ and view $[W^{s,p}_{comp}(\Omega)]^*$ as a subspace of $D'(\Omega)$. So, under this identification, we can rewrite the above equality as follows:
\begin{equation*}
\forall\, u\in [W^{s,p}_{comp}(\Omega)]^*\qquad \langle  u,\varphi\rangle_{D'(\Omega)\times D(\Omega)}=
\langle u,\varphi\rangle_{[W^{s,p}_{comp}(\Omega)]^*\times W^{s,p}_{comp}(\Omega)}
\end{equation*}
\end{remark}

Next we will prove that if $(s,p,\Omega)$ is a smooth multiplication triple, then $W^{s,p}_{comp}(\Omega)$ is separable. To this end, we need the following lemma.
\begin{lemma}\lab{lemapril281}
Let $(X,\tau)$ and $(Y,\tau')$ be two topological spaces. Suppose that
\begin{enumerateX}
\item $A$ is dense in $(X,\tau)$.
\item $T: (X,\tau)\rightarrow (Y,\tau')$ is continuous.
\item $T(X)$ is dense in $(Y,\tau')$.
\end{enumerateX}
Then $T(A)$ is dense in $(Y,\tau')$.
\end{lemma}

\begin{proof}
It is enough to show that $T(A)$ intersects every nonempty open set in $(Y,\tau')$. So let $O\in \tau'$ be nonempty. Since $T(X)$ is dense in $(Y,\tau')$, we have $O\cap T(X)\neq\emptyset$ and so $T^{-1}(O)$ is nonempty. Also since $T$ is continuous, $T^{-1}(O)\in \tau$. $A$ is dense in $(X,\tau)$, so $A\cap T^{-1}(O)\neq \emptyset$. Therefore
\begin{equation*}
T(A)\cap O\supseteq T(A\cap T^{-1}(O))\neq \emptyset\,.
\end{equation*}
\end{proof}

\begin{theorem}\lab{thmmay7801}
Let $(s,p,\Omega)$ be a smooth multiplication triple. Then $W^{s,p}_{comp}(\Omega)$ is separable.
\end{theorem}

\begin{proof}
According to theorems ~\ref{thmmay7759} and ~\ref{thmmay78}, $D(\Omega)$ is continuously embedded in $W^{s,p}_{comp}(\Omega)$ and it is dense in $W^{s,p}_{comp}(\Omega)$. Since $D(\Omega)$ is separable, it follows from Lemma ~\ref{lemapril281} that  $W^{s,p}_{comp}(\Omega)$ is separable.



\end{proof}

\begin{theorem}\lab{thmmay3803}
Let $(s,p,\Omega)$ be an interior smooth multiplication triple. Let $\{\psi_j\}_{j\in \mathbb{N}_0}$ be the partition of unity introduced in Theorem ~\ref{winter3}.
Let $S$ be the collection of all sequences whose terms are nonnegative integers. For all sequences $a=(a_0,a_1,\cdots)\in S$ define $q_{a,s,p}: W^{s,p}_{comp}(\Omega)\rightarrow \reals$ by
\begin{equation*}
q_{a,s,p}(u)=\sum_{j=0}^\infty a_j\|\psi_j u\|_{W^{s,p}(\Omega)}
\end{equation*}
Then $\{q_{a,s,p}\}_{a\in S}$ is a separating family of seminorms on $W^{s,p}_{comp}(\Omega)$ and the natural topology induced by this family on $W^{s,p}_{comp}(\Omega)$ is the same as the inductive limit topology on $W^{s,p}_{comp}(\Omega)$.
\end{theorem}

\begin{proof}
Note that support of every $u\in W^{s,p}_{comp}$ is compact, so for each $u$ only finitely many of $\psi_j u$'s are nonzero. Thus the sum in the definition of $q_{a,s,p}$ is a finite sum. Now it is not hard to show that each $q_{a,s,p}$ is a seminorm and $\{q_{a,s,p}\}_{a\in S}$ is separating. Here we will show that the topologies are the same. Let's denote the inductive limit topology on $W^{s,p}_{comp}(\Omega)$ by $\tau$ and the natural topology induced by the given family of seminorms $\tau'$. \\

In what follows we implicitly use the fact that both topologies are locally convex and translation invariant.
\begin{itemizeX}
\item \textbf{Step 1: ($\tau'\subseteq \tau$)}  We will prove that for each $K\in \mathcal{K}(\Omega)$, $W^{s,p}_K(\Omega)\hookrightarrow (W^{s,p}_{comp}(\Omega),\tau')$. This together with the definition of $\tau$ (the biggest topology with this property) implies that $\tau'\subseteq \tau$. Let $K\in \mathcal{K}(\Omega)$. By Theorem ~\ref{thmmay11714} it is enough to show that for all $a\in S$, $q_{a,s,p}\circ \textrm{Id}: W^{s,p}_K(\Omega)\rightarrow \reals$ is continuous. Since $K$ is compact, there are only finitely may $\psi_j$'s such that $K\cap \textrm{supp}\,\psi_j\neq \emptyset$; let's call them $\psi_{j_1},\cdots,\psi_{j_l}$. So for all $u\in W^{s,p}_K(\Omega)$
    \begin{equation*}
q_{a,s,p}(u)= a_{j_1}\|\psi_{j_1} u\|_{W^{s,p}(\Omega)}+\cdots+a_{j_l}\|\psi_{j_l} u\|_{W^{s,p}(\Omega)}
\end{equation*}
By assumption $(s,p,\Omega)$ is an interior smooth multiplication triple, so for each $j\in \{j_1,\cdots,j_l\}$ the mapping $u\mapsto \|\psi_j u\|_{W^{s,p}(\Omega)}$ from $W^{s,p}_K(\Omega)\rightarrow \reals$ is continuous. Hence $q_{a,s,p}\circ \textrm{Id}: W^{s,p}_K(\Omega)\rightarrow \reals$ must be continuous.
\item \textbf{Step 2: ($\tau\subseteq \tau'$)} Since $(W^{s,p}_{comp}(\Omega),\tau)$ is a locally convex topological vector space, there exists a separating family of seminorms $\mathcal{P}$ whose corresponding natural topology is $\tau$ (see Theorem ~\ref{thmmay3627}). We will prove that for all $\tilde{p}\in \mathcal{P}$, $\tilde{p}: (W^{s,p}_{comp}(\Omega),\tau')\rightarrow \reals$ is continuous. This together with the fact that $\tau$ is the smallest topology with this property, shows that $\tau\subseteq \tau'$. Let $\tilde{p}\in \mathcal{P}$. By Theorem ~\ref{thmmay2656}, it is enough to prove that there exists $a\in S$ such that
    \begin{equation*}
    \forall\,u\in W^{s,p}_{comp}(\Omega)\qquad \tilde{p}(u)\leq  q_{a,s,p}(u)
    \end{equation*}
    For all $u\in W^{s,p}_{comp}(\Omega)$ we have $\displaystyle\tilde{p}(u)=\tilde{p}\big(\sum_{j}\psi_j\,u\big)$. Since $u$ has compact support, only finitely many terms in the sum are nonzero, and so by the finite subadditivity of a seminorm we get
    \begin{align*}
    \tilde{p}(u)=\tilde{p}\big(\sum_{j}\psi_j\,u\big)\leq \sum_{j}\tilde{p}(\psi_j\,u)
     \end{align*}
     Now note that $\psi_j u$ belongs to the normed space $W^{s,p}_{\textrm{supp}\,\psi_j}(\Omega)$. Since $\tilde{p}: (W^{s,p}_{comp}(\Omega),\tau)\rightarrow \reals$ is continuous and $W^{s,p}_{\textrm{supp}\,\psi_j}(\Omega)\hookrightarrow (W^{s,p}_{comp}(\Omega),\tau)$, we can conclude that $\tilde{p}: W^{s,p}_{\textrm{supp}\,\psi_j}(\Omega)\rightarrow \reals$ is continuous. Thus, by Theorem ~\ref{thmmay3626}, there exists a positive integer $a_j$ such that
     \begin{equation*}
     \forall\,u\in W^{s,p}_{comp}(\Omega)\qquad  \tilde{p}(\psi_j u)\leq a_j\|\psi_j u\|_{W^{s,p}(\Omega)}
     \end{equation*}
     It follows that for all $u\in W^{s,p}_{comp}(\Omega)$
     \begin{equation*}
     \tilde{p}(u)\leq \sum_{j}\tilde{p}(\psi_j\,u)\leq \sum_{j} a_j\|\psi_j u\|_{W^{s,p}(\Omega)}=q_{a,s,p}(u)
     \end{equation*}
     where $a=(a_0,a_1,\cdots)$.
\end{itemizeX}
\end{proof}

\section{Spaces of Locally Sobolev Functions}
Let $s\in \reals$, $1<p<\infty$. Let $\Omega$ be a nonempty open set in
$\reals^n$. We define
\begin{align*}
 W^{s,p}_{loc}(\Omega):=\{u\in D'(\Omega): \forall \varphi\in
C_c^\infty(\Omega)\quad \varphi u\in W^{s,p}(\Omega)\}
\end{align*}
We equip $W^{s,p}_{loc}(\Omega)$ with the natural topology
induced by the separating family of seminorms $\{|.|_{\varphi,s,p}\}_{\varphi\in
C_c^{\infty}(\Omega)}\}$ (see Definition \ref{winter28}) where
\begin{equation*}
\forall\, u\in W^{s,p}_{loc}(\Omega)\quad \varphi\in
C_c^{\infty}(\Omega)\qquad |u|_{\varphi,s,p}:=\|\varphi
u\|_{W^{s,p}(\Omega)}
\end{equation*}
When $s$ and $p$ are clear from the context, we may just write $|u|_\varphi$ or $p_\varphi(u)$ instead of $|u|_{\varphi,s,p}$. It is easy to show that for all $\varphi\in C_c^\infty(\Omega)$, $|.|_{\varphi,s,p}$ is a seminorm on $W^{s,p}_{loc}(\Omega)$. The fact that the family of seminorms $\{|.|_{\varphi,s,p}\}_{\varphi\in
C_c^{\infty}(\Omega)}\}$ is separating will be proved in Theorem ~\ref{thmapp7b}.

\begin{remark}
Note that, by item 1. of Theorem \ref{thmapptvconvergence1}, $u_i\rightarrow u$  in $W^{s,p}_{loc}(\Omega)$ if and only if $\varphi u_i\rightarrow \varphi u$ in $W^{s,p}(\Omega)$ for all $\varphi \in C_c^\infty(\Omega)$.
\end{remark}
\begin{remark}\lab{remmay7803}
Clearly if $(s,p,\Omega)$ is a smooth multiplication triple, then $W^{s,p}(\Omega)\subseteq W^{s,p}_{loc}(\Omega)$.
\end{remark}
An equivalent description of locally Sobolev functions
is described in the following theorem.
\begin{theorem}\lab{thmapp7a}
Suppose that $(s,p,\Omega)$ is a smooth multiplication triple. Then $u\in D'(\Omega)$ is in $W^{s,p}_{loc}(\Omega)$ if and only
if for every precompact open set $V$ with $\bar{V}\subseteq
\Omega$ there is $w\in W^{s,p}(\Omega)$ such that $w|_V=u|_V$.
\end{theorem}
\begin{proof}
\leavevmode
$(\Rightarrow)$ Suppose $u\in W^{s,p}_{loc}(\Omega)$ and let
$V$ be a precompact open set such that $\bar{V}\subseteq \Omega$.
Let $\varphi \in C_c^{\infty}(\Omega)$ be such that $\varphi=1$
on a neighborhood containing $\bar{V}$. Let $w=\varphi u$. $u$ is a locally Sobolev function, so $w\in W^{s,p}(\Omega)$; also clearly $w|_V=u|_V$.\\
$(\Leftarrow)$ Suppose $u\in D'(\Omega)$ has the property that
for every precompact open set $V$ with $\bar{V}\subseteq \Omega$
there is $w\in W^{s,p}(\Omega)$ such that $w|_V=u|_V$. Let $\varphi \in
C_c^{\infty}(\Omega)$. We need to show that $\varphi u\in
W^{s,p}(\Omega)$. Note that $\textrm{supp}\,\varphi$ is compact,
so there exists a bounded open set $V$ such that
\begin{equation*}
\textrm{supp}\, \varphi\subseteq V\subseteq \bar{V}\subseteq
\Omega
\end{equation*}
By assumption there exists $w\in W^{s,p}(\Omega)$ such
that $w|_V=u|_V$. It follows from the hypothesis of the theorem that $\varphi
 w\in W^{s,p}(\Omega)$. Clearly
 $\varphi w=\varphi u$ on $\Omega$. Therefore $\varphi u\in W^{s,p}(\Omega)$.
\end{proof}

\section{Overview of the Basic Properties}
Material of this section is mainly an adaptation of the material presented in the excellent work of Antonic and Burazin \cite{Antonic1}, which is restricted to integer order Sobolev spaces, and Peterson \cite{Petersen83}, which is restricted to Hilbert spaces $H^s$. We have added certain details to the statements of the theorems and their proofs to ensure all the arguments are valid for both integer and noninteger order Sobolev-Slobodeckij spaces.

\begin{definition}
If $A$ is a subset of $C_c^{\infty}(\Omega)$ with the
following property:
\begin{equation*}
\forall\,x\in \Omega\quad \exists \varphi \in A\quad \textrm{such
that}\quad \varphi\geq 0\quad \textrm{and}\quad \varphi(x)\neq 0
\end{equation*}
then we say $A$ is an \textbf{admissible} family of functions.
\end{definition}

\begin{remark}\lab{remmay7805}
Note that if $A$ is an admissible family of functions, then for
all $m\in \mathbb{N}$, the set $\{\varphi^m: \varphi\in A\}$ is
also an admissible family of functions.
\end{remark}

\begin{theorem}\lab{thmapp7b}
Let $(s,p,\Omega)$ be an interior smooth multiplication triple.
  If $A$ is an admissible family of functions then
\begin{enumerateX}
\item $
W^{s,p}_{loc}(\Omega)=\{u\in D'(\Omega): \forall \varphi\in
A\quad \varphi u\in W^{s,p}(\Omega)\}
$
\item The collection $\{|.|_\varphi: \varphi\in A\}$ is a separating family of seminorms on $W^{s,p}_{loc}(\Omega)$.
\item The natural topology induced by the separating family of seminorms $\{|.|_\varphi: \varphi\in A\}$ is the same as the topology of $W^{s,p}_{loc}(\Omega)$.
\end{enumerateX}
\end{theorem}
\begin{proof}
\leavevmode
\begin{enumerateX}
\item Let $u\in D'(\Omega)$ be such that $\varphi u\in W^{s,p}(\Omega)$ for all
$\varphi \in A$. We need to show that if $\psi\in
C_c^{\infty}(\Omega)$, then $\psi u\in W^{s,p}(\Omega)$. By the
definition of $A$, for all $x\in \textrm{supp}\,\psi$ there
exists $\varphi_x \in A$ such that $\varphi_x(x)>0$. Define
\begin{equation*}
U_x:=\{y\in \Omega: \varphi_x(y)>0\}
\end{equation*}
Clearly $x\in U_x$ and since $\varphi_x$ is continuous, $U_x$ is
an open set. $\{U_x\}_{x\in \textrm{supp}\psi}$ is an open cover
of the compact set $\textrm{supp}\psi$. So there exist points
$x_1,\cdots,x_k$ such that $\textrm{supp}\psi\subseteq
U:=U_{x_1}\cup  \cdots \cup U_{x_k}$. If $y\in U$, then there
exists $1\leq i\leq k$ such that $y\in U_{x_i}$ and so
$\varphi_{x_i}(y)>0$. So the smooth function $\sum_{i=1}^k
\varphi_{x_i}$ is nonzero on $U$. Thus on $U$ we have
\begin{equation*}
\psi u=\frac{\psi}{\sum_{i=1}^k
\varphi_{x_i}}\big(\sum_{i=1}^k \varphi_{x_i} u\big)
\end{equation*}
Indeed, if we define
\begin{equation*}
\xi(z)=
\begin{cases}
\frac{\psi(z)}{\sum_{i=1}^k
\varphi_{x_i}(z)}\quad &\textrm{if $z\in U$}\\
 0\quad\quad &\textrm{otherwise}
\end{cases}
\end{equation*}
then $\xi$ is smooth with compact support in $U$ and
\begin{equation*}
\psi u=\xi \sum_{i=1}^k \varphi_{x_i} u
\end{equation*}
on the entire $\Omega$. Now note that for each $i$, $\varphi_{x_i} u$ is in $W^{s,p}(\Omega)$
(because by assumption $\varphi u\in W^{s,p}(\Omega)$ for all
$\varphi \in A$). So
$\sum_{i=1}^k \varphi_{x_i} u\in W^{s,p}(\Omega)$. Since $\xi\in C_c^\infty(\Omega)$ and $\sum_{i=1}^k \varphi_{x_i} u$ has compact support and $(s,p,\Omega)$ is an interior smooth multiplication triple, it follows that  $\xi\sum_{i=1}^k \varphi_{x_i} u\in
W^{s,p}(\Omega)$.
\item Now we prove that $\{|.|_\varphi: \varphi\in A\}$ is a separating family of seminorms. We need to show that if $u\in W^{s,p}_{loc}(\Omega)$ is such that for all $\varphi\in A$ $|u|_\varphi=\|\varphi\,u\|_{W^{s,p}(\Omega)}= 0$, then $u=0$. By definition of locally Sobolev functions, $u$ is an element of $D'(\Omega)$. So, in order to show that $u=0$, it is enough to prove that for all $\eta \in C_c^\infty(\Omega)$, $\langle u,\eta\rangle_{D'(\Omega)\times D(\Omega)}=0$. We consider two cases:
    \begin{itemize}
    \item \textbf{Case 1:} $A=C_c^\infty(\Omega)$.\\
    Let $\varphi\in A$ be such that $\varphi=1$ on a neighborhood containing $\textrm{supp}\,\eta$. By assumption $\varphi\,u=0$ in $W^{s,p}(\Omega)$ and so it is zero in $D'(\Omega)$. Now we have
    \begin{equation*}
    \langle u,\eta\rangle_{D'(\Omega)\times D(\Omega)}=\langle u,\varphi\eta\rangle_{D'(\Omega)\times D(\Omega)}
    =\langle \varphi u,\eta\rangle_{D'(\Omega)\times D(\Omega)}=0
    \end{equation*}
    which is exactly what we wanted to prove.
    \item \textbf{Case 2:} $A\subset C_c^\infty(\Omega)$.\\
    We claim that if $\|\varphi u\|_{W^{s,p}(\Omega)}=0$ for all $\varphi\in A$, then for any $\psi\in C_c^\infty(\Omega)$, $\|\psi u\|_{W^{s,p}(\Omega)}=0$ and so this case reduces to the previous case. Indeed, if $\psi$ is an arbitrary element of $C_c^\infty(\Omega)$, then by what was proved in the first item
    \begin{equation*}
    \psi u=\xi \sum_{i=1}^k \varphi_{x_i} u
    \end{equation*}
    where by assumption for each $i$, $\varphi_{x_i}u$ is zero as an element of $W^{s,p}(\Omega)$. Hence $\psi u=0$ in $W^{s,p}(\Omega)$.
    \end{itemize}
\item Finally we show that the natural topology $\tau_{\mathcal{P}}$ induced by $\mathcal{P}=\{|.|_\varphi: \varphi\in A\}$ is the same as the natural topology $\tau_{\mathcal{Q}}$ induced by $\mathcal{Q}=\{|.|_\varphi: \varphi\in C_c^\infty(\Omega)\}$. Obviously $\mathcal{P}$ is a subset of $\mathcal{Q}$, so it follows from the definition of natural topology induced by a family of seminorms (see Definition ~\ref{winter28}) that $\tau_{\mathcal{P}}\subseteq \tau_{\mathcal{Q}}$. In order to show that $\tau_{\mathcal{Q}}\subseteq \tau_{\mathcal{P}}$, it is enough to prove that for all $\psi\in C_c^\infty(\Omega)$, the map $|.|_\psi: (W^{s,p}_{loc}(\Omega),\tau_P)\rightarrow \reals$ is continuous. By what was shown in the first item, we can write
    \begin{align*}
    \forall\, u\in W^{s,p}_{loc}(\Omega)\qquad |u|_\psi &=\|\psi\, u\|_{W^{s,p}(\Omega)}=\|\xi \sum_{i=1}^k \varphi_{x_i} u\|_{W^{s,p}(\Omega)}\\
    &\preceq \sum_{i=1}^k \| \varphi_{x_i} u\|_{W^{s,p}(\Omega)}=\sum_{i=1}^k |u|_{\varphi_{x_i}}
    \end{align*}
    where the implicit constant does not depend on $u$. In the last inequality we used the assumption that $(s,p,\Omega)$ is an interior smooth multiplication triple. Now it follows from Theorem ~\ref{thmapptvconvergence1} that $|.|_\psi: (W^{s,p}_{loc}(\Omega),\tau_P)\rightarrow \reals$ is continuous.
\end{enumerateX}
\end{proof}

\begin{lemma}\lab{lemapril171}
There exists an admissible family $A\subseteq C_c^\infty(\Omega)$ that has only countably many elements.
\end{lemma}
\begin{proof}
Let $\{K_j\}_{j\in\mathbb{N}}$ be an exhaustion by compact sets for $\Omega$. For each $j\in \mathbb{N}$, let $\varphi_j\in C_c^\infty(\Omega)$ be a nonnegative function such that $\varphi_j=1$ on $K_j$ and $\varphi_j=0$ outside $\mathring{K}_{j+2}$. Clearly $A=\{\varphi_j\}_{j\in\mathbb{N}}$ is a countable admissible family of functions.
\end{proof}
\begin{corollary}\lab{cormay7806}
Let $(s,p,\Omega)$ be an interior smooth multiplication triple. Considering Theorem ~\ref{thmmay7310}, it follows from the previous lemma and Theorem ~\ref{thmapp7b} that $W^{s,p}_{loc}(\Omega)$ is metrizable. Indeed, if $A=\{\varphi_j\}_{j=1}^\infty$ is a countable admissible family, then
\begin{equation}\lab{eqnspring181}
d(u,v)=\sum_{j=1}^\infty\frac{1}{2^j}\frac{|u-v|_{\varphi_j}}{1+|u-v|_{\varphi_j}}
\end{equation}
is a compatible translation invariant metric on $W^{s,p}_{loc}(\Omega)$.
\end{corollary}

\begin{theorem}\lab{thmmay7819}
Let $(s,p,\Omega)$ be an interior smooth multiplication triple. Then $W^{s,p}_{loc}(\Omega)$ is a Frechet space.
\end{theorem}
\begin{proof}
By Corollary ~\ref{corspring181} it is enough to show that $W^{s,p}_{loc}(\Omega)$ equipped with the metric in (\ref{eqnspring181}) is complete. Note that all admissible families result in equivalent topologies in $W^{s,p}_{loc}(\Omega)$. So we can choose the functions $\varphi_j$'s in the definition of $d$ to be the partition of unity introduced in Theorem ~\ref{winter3}. Now suppose $\{u_m\}$ is a Cauchy sequence with respect to $d$. In what follows we will prove that $\{u_m\}$ converges to a distribution $u$ in $D'(\Omega)$. For now let's assume this is true. We need to show that $u$ is an element of $W^{s,p}_{loc}(\Omega)$, that is we need to show that for all $j$, $\varphi_j\,u\in W^{s,p}(\Omega)$.\\
It follows from the definition of $d$ that for each $j\in \mathbb{N}$, $\{\varphi_j u_m\}_{m\in \mathbb{N}}$ is a Cauchy sequence in $W^{s,p}(\Omega)$. Since $W^{s,p}(\Omega)$ is a Banach space, there exists $f_j$ in $W^{s,p}(\Omega)$ such that $\varphi_j u_m\rightarrow f_j$ in $W^{s,p}(\Omega)$. Note that $W^{s,p}(\Omega)\hookrightarrow D'(\Omega)$, so $\varphi_j u_m\rightarrow f_j$ in $D'(\Omega)$ and thus for all $\psi\in D(\Omega)$ we have
\begin{align*}
\langle f_j,\psi\rangle_{D'(\Omega)\times D(\Omega)}&=\lim_{m\rightarrow \infty}\langle \varphi_j\,u_m,\psi\rangle_{D'(\Omega)\times D(\Omega)}=
\lim_{m\rightarrow \infty}\langle u_m,\varphi_j\psi\rangle_{D'(\Omega)\times D(\Omega)}\\
&=
\langle u,\varphi_j \psi\rangle_{D'(\Omega)\times D(\Omega)}=\langle \varphi_j u,\psi\rangle_{D'(\Omega)\times D(\Omega)}
\end{align*}
Hence $\varphi_j u=f_j$ in $D'(\Omega)$. Since $f_j\in W^{s,p}(\Omega)$ we can conclude that $\varphi_j u\in W^{s,p}(\Omega)$.\\
It remains to show that $\{u_m\}$ converges in $D'(\Omega)$. To this end it is enough to show that for all $\psi\in D(\Omega)$, the sequence $\{\langle u_m,\psi\rangle\}$ converges in $\reals$ (see Theorem ~\ref{thmapril141}). Let $\psi\in D(\Omega)$. Since $\textrm{supp}\,\psi$ is compact, there are only finitely many of $\varphi_j$'s that are nonzero on the support of $\psi$ (see Theorem ~\ref{winter3}) which we denote by $\varphi_{j_1},\cdots,\varphi_{j_l}$. So for each $x\in \textrm{supp}\,\psi$, $\varphi_{j_1}(x)+\cdots+\varphi_{j_l}(x)=1$. We have
\begin{equation*}
\langle u_m,\psi\rangle=\langle u_m,(\varphi_{j_1}+\cdots+\varphi_{j_l})\psi\rangle=
\langle (\varphi_{j_1}+\cdots+\varphi_{j_l})u_m,\psi\rangle=\langle \varphi_{j_1}u_m,\psi\rangle+\cdots +\langle \varphi_{j_l}u_m,\psi\rangle
\end{equation*}
$\lim_{m\rightarrow \infty} \langle \varphi_{j_1}u_m,\psi\rangle,\cdots, \lim_{m\rightarrow \infty}\langle \varphi_{j_l}u_m,\psi\rangle$ all exist (since $\varphi_{j_r}u_m$ is Cauchy in $W^{s,p}(\Omega)$, it is convergent in $W^{s,p}(\Omega)$, and so it is convergent in $D'(\Omega)$). Therefore $\lim_{m\rightarrow \infty} \langle u_m,\psi\rangle$ exists.
\end{proof}

\begin{theorem}\lab{thmmay7820}
Let $(s,p,\Omega)$ be a smooth multiplication triple (so we know that $W^{s,p}(\Omega)\subseteq W^{s,p}_{loc}(\Omega)$ and $W^{s,p}_{loc}(\Omega)$ is metrizable). Then $W^{s,p}(\Omega)\hookrightarrow W^{s,p}_{loc}(\Omega)$.
\end{theorem}

\begin{proof}
Since both spaces are metrizable, it suffices to show that if $u_i\rightarrow u$ in $W^{s,p}(\Omega)$, then $u_i\rightarrow u$ in $W^{s,p}_{loc}(\Omega)$. To this end, let $\varphi$ be an arbitrary element of $C_c^\infty(\Omega)$. We need to show that if $u_i\rightarrow u$ in $W^{s,p}(\Omega)$, then $\varphi u_i\rightarrow \varphi u$ in $W^{s,p}(\Omega)$. But this is a consequence of the fact that $(s,p,\Omega)$ is a smooth multiplication triple.
\end{proof}

\begin{theorem}\lab{thmapril291}
Let $\Omega$ be a nonempty open set in $\reals^n$, $s\in \reals$ and $1<p<\infty$. Then $\mathcal{E}(\Omega)$ is continuously embedded in $W^{s,p}_{loc}(\Omega)$, i.e., the "identity map" from $\mathcal{E}(\Omega)$ to $W^{s,p}_{loc}(\Omega)$ is continuous.
\end{theorem}

\begin{proof}
By Theorem ~\ref{winter56} it is enough to show that if $\varphi_m\rightarrow 0$ in $\mathcal{E}(\Omega)$, then $\varphi_m\rightarrow 0$ in $W^{s,p}_{loc}(\Omega)$, that is, for all $\psi\in C_c^\infty(\Omega)$, $\psi \varphi_m\rightarrow 0$ in $W^{s,p}(\Omega)$.\\
 Let $\psi\in C_c^\infty(\Omega)$ and let $m_\psi$ denote multiplication by $\psi$. Multiplication by smooth functions is a continuous linear operator on $\mathcal{E}(\Omega)$ (\cite{Reus1}). So $m_\psi: \mathcal{E}(\Omega)\rightarrow \mathcal{E}(\Omega)$ is continuous. The range of this map is in the subspace $\mathcal{E}_{\textrm{supp}\,\psi}(\Omega)$. So $m_\psi: \mathcal{E}(\Omega)\rightarrow \mathcal{E}_{\textrm{supp}\,\psi}(\Omega)$ is continuous. However, $\mathcal{E}_{\textrm{supp}\,\psi}(\Omega)\hookrightarrow D(\Omega)$. Hence $m_\psi: \mathcal{E}(\Omega)\rightarrow D(\Omega)$ is continuous. As a consequence, since $\varphi_m\rightarrow 0$ in $\mathcal{E}(\Omega)$, $\psi\,\varphi_m\rightarrow 0$ in $D(\Omega)$. Finally, since $D(\Omega)\hookrightarrow W^{s,p}(\Omega)$, we can conclude that $\psi \varphi_m\rightarrow 0$ in $W^{s,p}(\Omega)$.
\end{proof}

\begin{corollary}\lab{cormay7821}
Since $D(\Omega)\hookrightarrow \mathcal{E}(\Omega)$, it follows that under the hypotheses of Theorem  ~\ref{thmapril291}, $D(\Omega)$ is continuously embedded in $W^{s,p}_{loc}(\Omega)$.
\end{corollary}

\begin{theorem}\lab{thmmay7822}
Let $(s,p,\Omega)$ be a smooth multiplication triple. Then $C_c^\infty(\Omega)$ is dense in $W^{s,p}_{loc}(\Omega)$.
\end{theorem}

\begin{proof}
Let $u\in W^{s,p}_{loc}(\Omega)$. It is enough to show that there exists a sequence $\{\psi_j\}$ in $C_c^\infty(\Omega)$ such that $\psi_j\rightarrow u$ in $W^{s,p}_{loc}(\Omega)$, i.e.,
\begin{equation*}
\forall\,\xi\in C_c^\infty(\Omega)\qquad \xi \psi_j \rightarrow \xi u\quad \textrm{in $W^{s,p}(\Omega)$}
\end{equation*}
First note that, since $(s,p,\Omega)$ is a smooth multiplication triple, for all $\xi\in C_c^\infty(\Omega)$, there exists a constant $C_{\xi,s,p,\Omega}$ such that
\begin{equation*}
\forall\, v\in W^{s,p}(\Omega)\qquad \|\xi\,v\|_{W^{s,p}(\Omega)}\leq
C_{\xi,s,p,\Omega}\|v\|_{W^{s,p}(\Omega)}
\end{equation*}
Let $\{\varphi_j\}_{j\in \mathbb{N}}$ be the admissible family introduced in the proof of Lemma ~\ref{lemapril171}. For each $\xi\in C_c^\infty(\Omega)$, there exists a number $J_\xi$ such that for all $j\geq J_{\xi}$, $\varphi_j=1$ on $\textrm{supp}\,\xi$. So
\begin{equation*}
\forall\, j\geq J_\xi\qquad \varphi_j\,\xi=\xi
\end{equation*}
Clearly, by definition of $W^{s,p}_{loc}(\Omega)$, for each $j$, $\varphi_j u\in W^{s,p}(\Omega)$, also $\varphi_j u$ has compact support, so $\varphi_j u\in W^{s,p}_0(\Omega)$ (see Remark ~\ref{remapril161}). Hence for each $j$, there exists $\psi_j\in C_c^\infty(\Omega)$ such that $\|\psi_j-\varphi_j u\|<\frac{1}{j}$. We claim that $\xi \psi_j \rightarrow \xi u$ in $W^{s,p}(\Omega)$. Indeed, given $\epsilon>0$ and $\xi\in C_c^\infty(\Omega)$, let  $J>J_\xi$ be such that $\frac{1}{J}<\frac{\epsilon}{C_{\xi,s,p,\Omega}}$. Then for $j\geq J$ we have
\begin{align*}
\|\xi \psi_j-\xi u\|_{W^{s,p}(\Omega)}&=\|\xi \psi_j-\xi\varphi_j u\|_{W^{s,p}(\Omega)}
=\|\xi (\psi_j-\varphi_j u)\|_{W^{s,p}(\Omega)}\\
&\leq C_{\xi,s,p,\Omega}\|\psi_j-\varphi_j u\|_{W^{s,p}(\Omega)}
< C_{\xi,s,p,\Omega}\frac{1}{J}<\epsilon\,.
\end{align*}
\end{proof}
\begin{remark}\lab{remapril282}
As a consequence, if $(s,p,\Omega)$ is a smooth multiplication triple, then $[W^{s,p}_{loc}(\Omega)]^*$ (equipped with the strong topology) is continuously embedded in $D'(\Omega)$. More precisely, the identity map $i: D(\Omega)\rightarrow W^{s,p}_{loc}(\Omega)$ is continuous with dense image, and therefore, by Theorem ~\ref{thmfallinjectiveadjoint1}, the adjoint $i^*: [W^{s,p}_{loc}(\Omega)]^*\rightarrow D'(\Omega)$ is a continuous injective map. We have
\begin{equation*}
\langle i^* u,\varphi\rangle_{D'(\Omega)\times D(\Omega)}=\langle u,i\,\varphi\rangle_{[W^{s,p}_{loc}(\Omega)]^*\times W^{s,p}_{loc}(\Omega)}=\langle u,\varphi\rangle_{[W^{s,p}_{loc}(\Omega)]^*\times W^{s,p}_{loc}(\Omega)}
\end{equation*}
We usually identify $[W^{s,p}_{loc}(\Omega)]^*$ with its image under $i^*$ and view $[W^{s,p}_{loc}(\Omega)]^*$ as a subspace of $D'(\Omega)$. So, under this identification, we can rewrite the above equality as follows:
\begin{equation*}
\forall\, u\in [W^{s,p}_{loc}(\Omega)]^*\qquad \langle  u,\varphi\rangle_{D'(\Omega)\times D(\Omega)}=
\langle u,\varphi\rangle_{[W^{s,p}_{loc}(\Omega)]^*\times W^{s,p}_{loc}(\Omega)}
\end{equation*}
\end{remark}

\begin{theorem}\lab{thmmay7823}
Let $(s,p,\Omega)$ be a smooth multiplication triple. Then $W^{s,p}_{loc}(\Omega)$ is separable.
\end{theorem}

\begin{proof}
$D(\Omega)$ is continuously embedded in $W^{s,p}_{loc}(\Omega)$ and it is dense in $W^{s,p}_{loc}(\Omega)$. Since $D(\Omega)$ is separable, it follows from Lemma ~\ref{lemapril281} that $W^{s,p}_{loc}(\Omega)$ is separable.
\end{proof}

As a direct consequence of the definitions, locally Sobolev functions and Sobolev functions with compact support are both subsets of the space of distributions. The next two theorems establish a duality connection between the two spaces. But first we need to state a simple lemma.

\begin{lemma}\lab{lemmay12408}
Let $X$ and $Y$ be two topological spaces. Suppose that $Y$ is Hausdorff. Let $f: X\rightarrow Y$ and $g: X\rightarrow Y$ be two continuous functions that agree on a dense subset $A$ of $X$. Then $f=g$ everywhere. (So, in particular, in order to show that two continuous mappings from $X$ to $Y$ are equal, we just need to show that they agree on some dense subset.)
\end{lemma}

\begin{proof}
Suppose that there exists $x_0\in X$ such that $f(x_0)\neq g(x_0)$. Since $Y$ is Hausdorff, there exist open neighborhoods $U$ and $V$ of $f(x_0)$ and $g(x_0)$, respectively, such that $U\cap V=\emptyset$. $f^{-1}(U)\cap g^{-1}(V)$ is a nonempty ($x_0$ is in it) open set in $X$ so its intersection with $A$ is nonempty. Let $z$ be a point in the intersection of $f^{-1}(U)\cap g^{-1}(V)$ and $A$. Clearly $f(z)\in U$ and $g(z)\in V$; but since $z\in A$, we have $f(z)=g(z)$. This contradicts the assumption that $U\cap V=\emptyset$.
\end{proof}

\begin{theorem}\lab{thmmay7824}
Suppose that $(s,p,\Omega)$ and $(-s,p',\Omega)$ are smooth multiplication triples. Define the mapping $T: W^{-s,p'}_{loc}(\Omega)\rightarrow [W^{s,p}_{comp}(\Omega)]^*$ by
\begin{equation*}
 \forall\,u\in W^{-s,p'}_{loc}(\Omega)\,\,\forall\,f\in W^{s,p}_{comp}(\Omega)\qquad [T(u)](f):=\langle  \psi_f u, f\rangle_{W^{-s,p'}_0(\Omega)\times W^{s,p}_0(\Omega)}
\end{equation*}
where $\psi_f$ is any function in $C_c^\infty(\Omega)$ that is equal to $1$ on a neighborhood containing the support of $f$. Then
\begin{enumerate}
\item $[T(u)](f)$ does not depend on the choice of $\psi_f$.
\item For all $u\in W^{-s,p'}_{loc}(\Omega)$, $T(u)$ is indeed an element of $[W^{s,p}_{comp}(\Omega)]^*$.
\item $T: W^{-s,p'}_{loc}(\Omega)\rightarrow [W^{s,p}_{comp}(\Omega)]^*$ is bijective.
\item Suppose $[W^{s,p}_{comp}(\Omega)]^*$ is equipped with the strong topology. Then the bijective linear map $T: W^{-s,p'}_{loc}(\Omega)\rightarrow [W^{s,p}_{comp}(\Omega)]^*$ is a topological isomorphism, i.e. it is continuous with continuous inverse. So $[W^{s,p}_{comp}(\Omega)]^*$ can be identified with $W^{-s,p'}_{loc}(\Omega)$ as topological vector spaces.
\end{enumerate}
\end{theorem}
\begin{proof}
\begin{enumerateX}
\item For the first item, it is enough to show that if $\psi\in C_c^\infty(\Omega)$ is equal to zero on a neighborhood $U$ containing the $\textrm{supp}\,f$, then $\langle \psi u, f\rangle_{W^{-s,p'}_0(\Omega)\times W^{s,p}_0(\Omega)}=0$. Note that $f$ is not necessarily in $C_c^\infty(\Omega)$, so we cannot directly apply the duality pairing identity stated in Remark ~\ref{remapril28610}. Let $\{f_m\}$ be sequence in $C_c^\infty(\Omega)$ such that $f_m\rightarrow f$ in $W^{s,p}_0(\Omega)$. Let $\xi\in C_c^\infty(\Omega)$ be such that $\xi=1$ on $\textrm{supp}\,f$ and $\xi=0$ outside $U$. By assumption $(s,p,\Omega)$ is a smooth multiplication triple and so $\xi f_m\rightarrow \xi f=f$ in $W^{s,p}_0(\Omega)$. Since elements of dual are continuous, we have
    \begin{align*}
    \langle \psi u, f\rangle_{W^{-s,p'}_0(\Omega)\times W^{s,p}_0(\Omega)}&=
    \lim_{m\rightarrow \infty} \langle \psi u, \xi f_m\rangle_{W^{-s,p'}_0(\Omega)\times W^{s,p}_0(\Omega)}\\
    & \stackrel{\textrm{Remark ~\ref{remapril28610}}}{=} \lim_{m\rightarrow \infty} \langle \psi u, \xi f_m\rangle_{D'(\Omega)\times D(\Omega)}\\
    &= \lim_{m\rightarrow \infty} \langle  u,\psi \xi f_m\rangle_{D'(\Omega)\times D(\Omega)}=\lim_{m\rightarrow \infty}
    u(0)=0
    \end{align*}
    Note that $\xi f_m$ is zero outside $U$ and $\psi=0$ in $U$, so $\psi \xi f_m=0$ everywhere.
\item In order to show that $T(u)$ is an element of $[W^{s,p}_{comp}(\Omega)]^*$, we need to prove that $T(u): W^{s,p}_{comp}(\Omega)\rightarrow \reals$ is linear and continuous. Linearity is obvious. In order to prove continuity, we need to show that for all $K\in \mathcal{K}(\Omega)$, $T(u)|_{W^{s,p}_K(\Omega)}$ is continuous (see Theorem ~\ref{winter37}). Let $K\in \mathcal{K}(\Omega)$ and fix a function $\psi\in C_c^\infty(\Omega)$ which is equal to $1$ on a neighborhood containing $K$. For all $f\in W^{s,p}_K(\Omega)$ we have
    \begin{equation*}
    |[T(u)](f)|=|\langle \psi u, f\rangle_{W^{-s,p'}_0(\Omega)\times W^{s,p}_0(\Omega)}|\leq \|\psi u\|_{W^{-s,p'}(\Omega)}\times\|f\|_{W^{s,p}(\Omega)}
    \end{equation*}
    which proves the continuity of the linear map $T(u)$.
\item In order to prove that $T$ is bijective we give an explicit formula for the inverse. Recall that by definition $W^{-s,p'}_{loc}(\Omega)$ is a subspace of $D'(\Omega)$ and by Remark ~\ref{remapril281}, $[W^{s,p}_{comp}(\Omega)]^*$ can also be viewed as a subspace of $D'(\Omega)$. More precisely, if we let $i: D(\Omega)\rightarrow W^{s,p}_{comp}(\Omega)$ be the "identity map" and $i^*: [W^{s,p}_{comp}(\Omega)]^*\rightarrow D'(\Omega)$ be the adjoint of $i$, then $i^*$ is a continuous injective linear map and
\begin{align}\lab{eqnmar6257}
& \forall\, u\in [W^{s,p}_{comp}(\Omega)]^*\,\,\forall\,\varphi\in D(\Omega)\qquad \langle  i^* u,\varphi\rangle_{D'(\Omega)\times D(\Omega)}=
\langle u,\varphi\rangle_{[W^{s,p}_{comp}(\Omega)]^*\times W^{s,p}_{comp}(\Omega)}
\end{align}
Moreover, if $K\in \mathcal{K}(\Omega)$, then $W^{s,p}_K(\Omega)\hookrightarrow W^{s,p}_{comp}(\Omega)$ and therefore if $u\in [W^{s,p}_{comp}(\Omega)]^*$, then $u|_{W^{s,p}_K(\Omega)}\in [W^{s,p}_{K}(\Omega)]^*$ and
\begin{equation*}
\forall\,g\in W^{s,p}_K(\Omega)\qquad \langle u,g\rangle_{[W^{s,p}_{comp}(\Omega)]^*\times W^{s,p}_{comp}(\Omega)}=\langle u|_{W^{s,p}_K(\Omega)},g\rangle_{[W^{s,p}_{K}(\Omega)]^*\times W^{s,p}_{K}(\Omega)}
\end{equation*}
Now we claim that the image of $i^*$ is in $W^{-s,p'}_{loc}(\Omega)$ and in fact $i^*$ is the inverse of $T$. Let us first prove that the image of $i^*$ is in $W^{-s,p'}_{loc}(\Omega)$. Let $u\in [W^{s,p}_{comp}(\Omega)]^*$. We need to show that for all $\varphi\in C_c^\infty(\Omega)$, $(\varphi) (i^*u)\in W^{-s,p'}(\Omega)$. To this end we make use of Corollary ~\ref{corapril2812}. Let $\varphi\in C_c^\infty(\Omega)$ and let $K=\textrm{supp}\,\varphi$. For all $\psi\in D(\Omega)$ we have
    \begin{align*}
    \big|\langle \varphi i^*u, \psi\rangle_{D'(\Omega)\times D(\Omega)}\big|&=\big|\langle i^*u, \varphi\psi\rangle_{D'(\Omega)\times D(\Omega)}\big|=
    \big|\langle u, \varphi\psi\rangle_{[W^{s,p}_{comp}(\Omega)]^*\times W^{s,p}_{comp}(\Omega)}\big|\\
    &=\big|\langle u|_{W^{s,p}_K(\Omega)}, \varphi\psi\rangle_{[W^{s,p}_{K}(\Omega)]^*\times W^{s,p}_{K}(\Omega)}\big|\\
    &\leq \|u|_{W^{s,p}_K(\Omega)}\|_{[W^{s,p}_{K}(\Omega)]^*}\|\varphi\psi\|_{W^{s,p}_K(\Omega)}\\
    & =\|u|_{W^{s,p}_K(\Omega)}\|_{[W^{s,p}_{K}(\Omega)]^*}\|\varphi\psi\|_{W^{s,p}(\Omega)}\\
    &\preceq \|u|_{W^{s,p}_K(\Omega)}\|_{[W^{s,p}_{K}(\Omega)]^*}\|\psi\|_{W^{s,p}(\Omega)}
    \end{align*}
    which, by Corollary ~\ref{corapril2812}, proves that $\varphi i^*u\in W^{-s,p'}(\Omega)$. \\
    Now we prove $i^*$ is the inverse of $T$. Note that for all $u\in W^{-s,p'}_{loc}(\Omega)\subseteq D'(\Omega)$ and $\varphi\in D(\Omega)$
    \begin{align*}
    \langle (i^*\circ T)(u),\varphi\rangle_{D'(\Omega)\times D(\Omega)}&\stackrel{\textrm{Equation ~\ref{eqnmar6257}}}{=}
    \langle T(u),\varphi\rangle_{[W^{s,p}_{comp}(\Omega)]^*\times W^{s,p}_{comp}(\Omega)}\\
    &\stackrel{\textrm{Definition of $T$}}{=}\langle  \psi_\varphi u, \varphi\rangle_{W^{-s,p'}_0(\Omega)\times W^{s,p}_0(\Omega)}\\
    &\stackrel{\textrm{Remark ~\ref{remapril28610}}}{=} \langle  \psi_\varphi u, \varphi\rangle_{D'(\Omega)\times D(\Omega)}\\
    &= \langle   u, \psi_\varphi\varphi\rangle_{D'(\Omega)\times D(\Omega)}\\
    &=\langle   u, \varphi\rangle_{D'(\Omega)\times D(\Omega)}
    \end{align*}
    Therefore $i^*\circ T$ is identity. Next we show that for all $v\in [W^{s,p}_{comp}(\Omega)]^*$, $(T\circ i^*)(v)=v$. Note that $(T\circ i^*)(v)$ and $v$ both are in  $[W^{s,p}_{comp}(\Omega)]^*$ and so they are continuous functions from $W^{s,p}_{comp}(\Omega)$ to $\reals$. Since $D(\Omega)$ is dense in $W^{s,p}_{comp}(\Omega)$, according to Lemma ~\ref{lemmay12408} it is enough to show that for all $f\in D(\Omega)$ we have $[(T\circ i^*)(v)](f)=v(f)$.
    \begin{align*}
    \langle (T\circ i^*)(v), f\rangle_{[W^{s,p}_{comp}(\Omega)]^*\times W^{s,p}_{comp}(\Omega)}
    &\stackrel{\textrm{Definition of $T$}}{=}\langle  \psi_f (i^*v), f\rangle_{W^{-s,p'}_0(\Omega)\times W^{s,p}_0(\Omega)}\\
    & \stackrel{\textrm{Remark ~\ref{remapril28610}}}{=}\langle  \psi_f (i^*v), f\rangle_{D'(\Omega)\times D(\Omega)}\\
    &=\langle  i^*v, \psi_f f\rangle_{D'(\Omega)\times D(\Omega)}\\
    &=\langle  i^*v,  f\rangle_{D'(\Omega)\times D(\Omega)}\\
    &\stackrel{\textrm{Equation ~\ref{eqnmar6257}}}{=} \langle v, f\rangle_{[W^{s,p}_{comp}(\Omega)]^*\times W^{s,p}_{comp}(\Omega)}
    \end{align*}
 \item Let's denote the topology of $W^{-s,p'}_{loc}(\Omega)$ by $\tau$ and the strong topology on $[W^{s,p}_{comp}(\Omega)]^*$ by $\tau'$. Our goal is to show that $T: (W^{-s,p'}_{loc}(\Omega),\tau)\rightarrow ([W^{s,p}_{comp}(\Omega)]^*,\tau')$ and also $T^{-1}=i^*: ([W^{s,p}_{comp}(\Omega)]^*,\tau')\rightarrow (W^{-s,p'}_{loc}(\Omega),\tau)$ are both continuous maps. To this end we make use of Theorem ~\ref{thmapptvconvergence1}. Recall that $\tau$ is induced by the family of seminorms $\{p_\varphi: W^{-s,p'}_{loc}(\Omega)\rightarrow \reals\}_{\varphi\in C_c^\infty(\Omega)}$ where $p_\varphi(u)=\|\varphi u\|_{W^{-s,p'}(\Omega)}$. Also $\tau'$ is induced by the family of seminorms
$\{p'_B: [W^{s,p}_{comp}(\Omega)]^*\rightarrow \reals\}$ where $B$ varies over all bounded sets in $W^{s,p}_{comp}(\Omega)$ and $p'_B(u)=\sup_{f\in B}{|u(f)|}$.
\begin{itemizeX}
\item \textbf{Step 1:} Let $B$ be a bounded subset of $W^{s,p}_{comp}(\Omega)$. Since $B$ is bounded, there exists $K\in \mathcal{K}(\Omega)$ such that $B$ is bounded in $W^{s,p}_K(\Omega)$ (See Theorem ~\ref{thmmay7609}; note  that the topology of $W^{s,p}_{comp}(\Omega)$ can be constructed as the inductive limit of $W^{s,p}_{K_j}(\Omega)$ where $\{K_j\}$ is an increasing chain of compact subsets of $\Omega$). So there exists a constant $C$ such that for all $f\in B$, $\|f\|_{W^{s,p}(\Omega)}\leq C$. Let $\psi$ be a function in $C_c^\infty(\Omega)$ which is equal to $1$ on a neighborhood containing $K$. For all $u\in W^{-s,p'}_{loc}(\Omega)$ we have
    \begin{align*}
    (p'_B\circ T)(u)=\sup_{f\in B}{|[T(u)](f)|}&\stackrel{\textrm{Definition of $T$}}{=}\sup_{f\in B}{|\langle \psi u, f\rangle_{W^{-s,p'}_{0}(\Omega)\times W^{s,p}_{0}(\Omega)}|}\\
    &\leq \sup_{f\in B}\|\psi u\|_{W^{-s,p'}(\Omega)}\|f\|_{W^{s,p}(\Omega)}\\
    &\leq C p_\psi (u)
    \end{align*}
    It follows from Theorem ~\ref{thmapptvconvergence1} that $T: (W^{-s,p'}_{loc}(\Omega),\tau)\rightarrow ([W^{s,p}_{comp}(\Omega)]^*,\tau')$ is continuous.
\item \textbf{Step 2:} Let $\varphi\in C_c^\infty(\Omega)$. Let $K$ be a compact set whose interior contains $\textrm{supp}\,\varphi$. Since $(s,p,\Omega)$ is a smooth multiplication triple, there exists a constant $C_\varphi>0$ such that for all $f\in W^{s,p}(\Omega)$ we have $\|\varphi f\|_{W^{s,p}(\Omega)}\leq C_\varphi \|f\|_{W^{s,p}(\Omega)}$.\\

     We have
    \begin{align*}
    (p_\varphi\circ i^*)(u)&=\|\varphi i^*u\|_{W^{s,p}_0(\Omega)}=\|\varphi i^*u\|_{[W^{-s,p'}_0(\Omega)]^*}\\
    &=\sup_{\xi \in C_c^\infty(\Omega), \|\xi\|_{W^{s,p}(\Omega)}\leq 1}|\langle \varphi i^*u, \xi\rangle_{W^{s,p}_0(\Omega)\times W^{-s,p'}_0(\Omega)}|\\
    &=\sup_{\xi \in C_c^\infty(\Omega), \|\xi\|_{W^{s,p}(\Omega)}\leq 1}|\langle  \varphi i^*u, \xi\rangle_{D'(\Omega)\times D(\Omega)}|\\
    &=\sup_{\xi \in C_c^\infty(\Omega), \|\xi\|_{W^{s,p}(\Omega)}\leq 1}|\langle   i^*u, \varphi\xi\rangle_{D'(\Omega)\times D(\Omega)}|\\
    &\leq\sup_{\eta \in C_{\textrm{supp}\,\varphi}^\infty(\Omega), \|\eta\|_{W^{s,p}(\Omega)}\leq C_\varphi}|\langle i^*u, \eta\rangle_{D'(\Omega)\times D(\Omega)}|\\
    &\stackrel{\textrm{Equation ~\ref{eqnmar6257}}}{=}\sup_{\eta \in C_{\textrm{supp}\,\varphi}^\infty(\Omega), \|\eta\|_{W^{s,p}(\Omega)}\leq C_\varphi}|\langle u, \eta\rangle_{[W^{s,p}_{comp}(\Omega)]^*\times W^{s,p}_{comp}(\Omega)}|
    \end{align*}
    So if we let $B$ be the ball of radius $2C_\varphi$ centered at $0$ in $W^{s,p}_{K}(\Omega)$ (clearly $B$ is a bounded set in $W^{s,p}_{comp}(\Omega)$) we get
    \begin{align*}
    (p_\varphi\circ i^*)(u)&\leq \sup_{f\in B}|\langle u, f\rangle_{[W^{s,p}_{comp}(\Omega)]^*\times W^{s,p}_{comp}(\Omega)}|\\
    &= p'_B(u)
    \end{align*}
\end{itemizeX}
\end{enumerateX}
\end{proof}

\begin{corollary}\lab{cormay13615}
Suppose that $(s,p,\Omega)$ and $(-s,p',\Omega)$ are both smooth multiplication triples. By the previous Theorem $[W^{-s,p'}_{comp}(\Omega)]^*$ can be identified with $W^{s,p}_{loc}(\Omega)$. Also, by Remark ~\ref{remapril281}, $[W^{-s,p'}_{comp}(\Omega)]^*$ is continuously embedded in $D'(\Omega)$. Therefore $W^{s,p}_{loc}(\Omega)$ is continuously embedded in $D'(\Omega)$. Since $W^{s,p}_{loc}(\Omega)$ is a Frechet space, it follows from Theorem ~\ref{thmapptvconvergence2} and Remark ~\ref{winter32} that the preceding statement remains true even if we consider $D'(\Omega)$ equipped with the weak$^*$ topology. So
\begin{equation*}
W^{s,p}_{loc}(\Omega)\hookrightarrow (D'(\Omega), \textrm{strong topology})\qquad \textrm{and}\qquad
W^{s,p}_{loc}(\Omega)\hookrightarrow (D'(\Omega), \textrm{weak$^*$ topology})
\end{equation*}
\end{corollary}

\begin{theorem}\lab{thmmay7825}
Suppose that $(s,p,\Omega)$ and $(-s,p',\Omega)$ are smooth multiplication triples. Define the mapping $R: W^{-s,p'}_{comp}(\Omega)\rightarrow [W^{s,p}_{loc}(\Omega)]^*$ by
\begin{equation*}
 \forall\,u\in W^{-s,p'}_{comp}(\Omega)\,\,\forall\,f\in W^{s,p}_{loc}(\Omega)\qquad [R(u)](f):=\langle   u, \psi_u f\rangle_{W^{-s,p'}_0(\Omega)\times W^{s,p}_0(\Omega)}
\end{equation*}
where $\psi_u$ is any function in $C_c^\infty(\Omega)$ that is equal to $1$ on a neighborhood containing the support of $u$. Then
\begin{enumerate}
\item $[R(u)](f)$ does not depend on the choice of $\psi_u$.
\item For all $u\in W^{-s,p'}_{comp}(\Omega)$, $R(u)$ is indeed an element of $[W^{s,p}_{loc}(\Omega)]^*$.
\item $R: W^{-s,p'}_{comp}(\Omega)\rightarrow [W^{s,p}_{loc}(\Omega)]^*$ is bijective.
\item Suppose $[W^{s,p}_{loc}(\Omega)]^*$ is equipped with the strong topology. Then the linear map $R$ is  bijective and continuous. In particular, $[W^{s,p}_{loc}(\Omega)]^*$ and $W^{-s,p'}_{comp}(\Omega)$ are isomorphic vector spaces.
\end{enumerate}
\end{theorem}

\begin{proof}
\begin{enumerateX}
\item Note that since $(s,p,\Omega)$ is a smooth multiplication triple, $\psi_u f$ is in $W^{s,p}_0(\Omega)$. Also by assumption $(-s,p',\Omega)$ is a smooth multiplication triple. Therefore for each $K\in \mathcal{K}(\Omega)$, $W^{-s,p'}_{K}(\Omega)\hookrightarrow W^{-s,p'}_0(\Omega)$ and hence $W^{-s,p'}_{comp}(\Omega)\hookrightarrow W^{-s,p'}_0(\Omega)$. So the pairing in the definition of $[R(u)](f)$ makes sense. The fact that the output is independent of the choice of $\psi_u$ follows directly from Theorem ~\ref{thmmay31127}.
\item  Clearly $R(u)$ is linear. Also $R(u)$ is continuous (so it is an element of $[W^{s,p}_{loc}(\Omega)]^*$). The reason is as follows: for all $f\in W^{s,p}_{loc}(\Omega)$ we have
   \begin{equation*}
      |[R(u)](f)|=|\langle  u, \psi_u f\rangle_{W^{-s,p'}_0(\Omega)\times W^{s,p}_0(\Omega)}|\leq
      \|u\|_{W^{-s,p'}_0(\Omega)}\|\psi_u f\|_{W^{s,p}_0(\Omega)}
   \end{equation*}
   That is for all $f\in W^{s,p}_{loc}(\Omega)$ we have $|[R(u)](f)|\preceq \|\psi_u f\|_{W^{s,p}(\Omega)}$. It follows from Theorem ~\ref{thmapptvconvergence1} that $R(u): W^{s,p}_{loc}(\Omega)\rightarrow \reals$ is continuous.
\item
In order to prove that $R$ is bijective we give an explicit formula for the inverse. Recall that by Remark ~\ref{remapril282}, $[W^{s,p}_{loc}(\Omega)]^*$ can also be viewed as a subspace of $D'(\Omega)$. More precisely, if we let $i: D(\Omega)\rightarrow W^{s,p}_{loc}(\Omega)$ be the "identity map" and $i^*: [W^{s,p}_{loc}(\Omega)]^*\rightarrow D'(\Omega)$ be the adjoint of $i$, then $i^*$ is a continuous injective linear map and
\begin{align}\lab{eqnmar6654}
& \forall\, u\in [W^{s,p}_{loc}(\Omega)]^*\,\,\forall\,\varphi\in D(\Omega)\qquad \langle  i^* u,\varphi\rangle_{D'(\Omega)\times D(\Omega)}=
\langle u,\varphi\rangle_{[W^{s,p}_{loc}(\Omega)]^*\times W^{s,p}_{loc}(\Omega)}
\end{align}
Now we claim that the image of $i^*$ is in $W^{-s,p'}_{comp}(\Omega)$ and in fact $i^*$ is the inverse of $R$. Let us first prove that the image of $i^*$ is in $W^{-s,p'}_{comp}(\Omega)$. $\mathcal{E}(\Omega)$ is continuously and densely  embedded in $W^{s,p}_{loc}(\Omega)$ (continuity is proved in Theorem ~\ref{thmapril291} and density is a direct consequence of the density of $C_c^\infty(\Omega)$ in $W^{s,p}_{loc}(\Omega)$). Therefore $i^*\big([W^{s,p}_{loc}(\Omega)]^*\big)$ is indeed a subspace of $\mathcal{E}'(\Omega)\subseteq D'(\Omega)$ and so elements of $i^*\big([W^{s,p}_{loc}(\Omega)]^*\big)$ can be identified with distributions in $D'(\Omega)$ that have compact support. It remains to show that if $u\in [W^{s,p}_{loc}(\Omega)]^*$, then $i^*u\in W^{-s,p'}(\Omega)$. To this end we make use of Corollary ~\ref{corapril2812}. For all $\varphi\in D(\Omega)$ we have
        \begin{align*}
        |\langle i^*u,\varphi\rangle_{D'(\Omega)\times D(\Omega)}|
        &\stackrel{\textrm{Equation ~\ref{eqnmar6654}}}{=}|\langle u,\varphi\rangle_{[W^{s,p}_{loc}(\Omega)]^*\times W^{s,p}_{loc}(\Omega)} |\\
        &=|\langle u|_{W^{s,p}(\Omega)},\varphi\rangle_{[W^{s,p}_0(\Omega)]^*\times W^{s,p}_0(\Omega)}|\\
        &\leq \|u|_{W^{s,p}_0(\Omega)}\|_{[W^{s,p}_0(\Omega)]^*}\|\varphi\|_{W^{s,p}(\Omega)}
        \end{align*}
        So, by Corollary ~\ref{corapril2812}, we can conclude that $u\in W^{-s,p'}(\Omega)$. In the above we used the fact that $W^{s,p}_0(\Omega)\hookrightarrow W^{s,p}(\Omega)\hookrightarrow W^{s,p}_{loc}(\Omega)$ and so for $u\in [W^{s,p}_{loc}(\Omega)]^*$ we have $u|_{W^{s,p}_0(\Omega)}\in [W^{s,p}_0(\Omega)]^*$.\\
        Now we prove that $i^*: [W^{s,p}_{loc}(\Omega)]^*\rightarrow W^{-s,p'}_{comp}(\Omega)$ is the inverse of $R$. For all $u\in W^{-s,p'}_{comp}(\Omega)$ and $\varphi\in D(\Omega)$ we have
        \begin{align*}
        \langle (i^*\circ R) (u),\varphi\rangle_{D'(\Omega)\times D(\Omega)}&\stackrel{\textrm{Equation ~\ref{eqnmar6654}}}{=}
        \langle R u, \varphi\rangle_{[W^{s,p}_{loc}(\Omega)]^*\times W^{s,p}_{loc}(\Omega)} \\
        &\stackrel{\textrm{Definition of $R$}}{=} \langle u,\psi_u \varphi\rangle_{W^{-s,p'}_0(\Omega)\times W^{s,p}_0(\Omega)}\\
        &\stackrel{\textrm{Remark ~\ref{remapril28610}}}{=}\langle u,\psi_u \varphi\rangle_{D'(\Omega)\times D(\Omega)}\\
        &=\langle \psi_u u, \varphi\rangle_{D'(\Omega)\times D(\Omega)}\\
        &=\langle u, \varphi\rangle_{D'(\Omega)\times D(\Omega)}
        \end{align*}
        Therefore $(i^*\circ R) (u)=u$ for all $u\in W^{-s,p'}_{comp}(\Omega)$.\\
        Now we prove that $R\circ i^*$ is also the identity map. Considering Lemma ~\ref{lemmay12408}, since $D(\Omega)$ is dense in $W^{s,p}_{loc}(\Omega)]^*$, it is enough to show that for all $v\in [W^{s,p}_{loc}(\Omega)]^*$ and $f\in D(\Omega)$, $[R\circ i^*(v)](f)=v(f)$. We have
        \begin{align*}
        \langle (R\circ i^*)v, f\rangle_{[W^{s,p}_{loc}(\Omega)]^*\times W^{s,p}_{loc}(\Omega)}&\stackrel{\textrm{Definition of $R$}}{=}
        \langle i^* v, \psi_{i^*v} f\rangle_{W^{-s,p'}_0(\Omega)\times W^{s,p}_0(\Omega)} \\
        &\stackrel{\textrm{Remark ~\ref{remapril28610}}}{=}\langle i^* v, \psi_{i^*v} f\rangle_{D'(\Omega)\times D(\Omega)}\\
        &= \langle \psi_{i^*v}  i^* v, f\rangle_{D'(\Omega)\times D(\Omega)}\\
        &= \langle   i^* v, f\rangle_{D'(\Omega)\times D(\Omega)}\\
        &\stackrel{\textrm{Equation ~\ref{eqnmar6654}}}{=} \langle v, f\rangle_{[W^{s,p}_{loc}(\Omega)]^*\times W^{s,p}_{loc}(\Omega)}
        \end{align*}
        which shows $R\circ i^*(v)=v$.
   \item Let's denote the topology of $W^{-s,p'}_{comp}(\Omega)$ by $\tau$ and the strong topology on $[W^{s,p}_{loc}(\Omega)]^*$ by $\tau'$. Our goal is to show that $R: (W^{-s,p'}_{comp}(\Omega),\tau)\rightarrow ([W^{s,p}_{loc}(\Omega)]^*,\tau')$ is continuous. To this end we make use of Theorem ~\ref{thmapptvconvergence1}. Recall that $\tau$ is induced by the family of seminorms $\{q_{a,-s,p'}: W^{-s,p'}_{comp}(\Omega)\rightarrow \reals\}_{a\in S}$ where $q_{a,-s,p'}(u)=\sum_j a_j\|\psi_j  u\|_{W^{-s,p'}(\Omega)}$ (here we are using the notations introduced in Theorem ~\ref{thmmay3803}). Also $\tau'$ is induced by the family of seminorms
$\{p_B: [W^{s,p}_{loc}(\Omega)]^*\rightarrow \reals\}$ where $B$ varies over all bounded sets in $W^{s,p}_{loc}(\Omega)$ and $p_B(u)=\sup_{f\in B}{|u(f)|}$.\\
Let $B$ be a bounded subset of $W^{s,p}_{loc}(\Omega)$. Since $B$ is bounded, for all $\varphi\in C_c^\infty(\Omega)$, the set $\{\|\varphi f\|_{W^{s,p}(\Omega)}: f\in B\}$ is bounded in $\reals$ (see Theorem ~\ref{thmmay3814}). Thus for all $\varphi\in C_c^\infty(\Omega)$ there exists a positive integer $a_\varphi$ such that for all $f\in B$, $\|\varphi f\|_{W^{s,p}(\Omega)}<a_\varphi$. Recall that $\{\psi_j\}$ in the definition of $q_{a,-s,p'}$ denotes a fixed partition of unity. For each $j$ let $\varphi_j$ be a function in $C_c^\infty(\Omega)$ which is equal to $1$ on a neighborhood containing the support of $\psi_j$. For all $u\in W^{-s,p'}_{comp}(\Omega)$ we have
    \begin{align*}
    (p_B\circ R)(u)&=(p_B\circ R)(\sum_j \psi_j\, u)\leq \sum_j (p_B\circ R)(\psi_j u)=\sup_{f\in B}{|(R(\psi_ju))(f)|}\\
    &\stackrel{\textrm{Definition of $R$}}{=}\sum_j\sup_{f\in B}{|\langle \psi_j u, \varphi_j f\rangle_{W^{-s,p'}_{0}(\Omega)\times W^{s,p}_{0}(\Omega)}|}\\
    &\leq \sum_j \sup_{f\in B}\|\psi_j u\|_{W^{-s,p'}(\Omega)}\|\varphi_j f\|_{W^{s,p}(\Omega)}\\
    &\leq \sum_j a_{\varphi_j} \|\psi_j u\|_{W^{-s,p'}(\Omega)}\\
    &=q_{a,-s,p'}(u)
    \end{align*}
    where $a=(a_{\varphi_1},a_{\varphi_2},\cdots)$. Note that the inequality $(p_B\circ R)(\sum_j \psi_j\,u)\leq \sum_j (p_B\circ R)(\psi_j u)$ holds because $u$ has compact support and so only finitely many terms in the sum are nonzero, so we can use the subadditivity of the seminorm and linearity of $R$.\\
    It follows from Theorem ~\ref{thmapptvconvergence1} that $R: (W^{-s,p'}_{comp}(\Omega),\tau)\rightarrow [W^{s,p}_{loc}(\Omega)]^*$ is continuous.

\end{enumerateX}
\end{proof}

\begin{remark}\lab{remmay2757}
According to the previous two theorems, we have the following
\begin{itemizeX}
\item When $u\in W^{-s,p'}_{loc}(\Omega)$ is viewed as an element of $[W^{s,p}_{comp}(\Omega)]^*$ we have
\begin{equation*}
\forall\,f\in W^{s,p}_{comp}(\Omega)\qquad u(f)=\langle \psi_f u,f\rangle_{W^{-s,p'}_0(\Omega)\times W^{s,p}_0(\Omega)}
\end{equation*}
where $\psi_f$ is any function in $C_c^\infty(\Omega)$ that is equal to $1$ on a neighborhood containing $\textrm{supp}\,f$.
\item When $u\in W^{-s,p'}_{comp}(\Omega)$ is viewed as an element of $[W^{s,p}_{loc}(\Omega)]^*$ we have
    \begin{equation*}
\forall\,f\in W^{s,p}_{loc}(\Omega)\qquad u(f)=\langle  u,\psi f\rangle_{W^{-s,p'}_0(\Omega)\times W^{s,p}_0(\Omega)}
\end{equation*}
where $\psi$ is any function in $C_c^\infty(\Omega)$ that is equal to $1$ on a neighborhood containing $\textrm{supp}\,u$.
\end{itemizeX}
\end{remark}


\begin{corollary}\lab{cormay7826}
Suppose that $(s,p,\Omega)$ and $(-s,p',\Omega)$ are both smooth multiplication triples. As a direct consequence of the previous theorems, the bidual of $W^{s,p}_{comp}(\Omega)$ is itself. So $W^{s,p}_{comp}(\Omega)$ is semireflexive. It follows from Theorem ~\ref{thmmay6758} that $W^{s,p}_{comp}(\Omega)$ is reflexive and subsequently its dual $W^{-s,p'}_{loc}(\Omega)$ is reflexive.
\end{corollary}

Now we put everything together to build general embedding theorems for spaces of locally Sobolev-Slobodeckij functions.

\begin{theorem}[Embedding Theorem I]\lab{thmmay7827}
Let $\Omega$ be a nonempty open set in $\reals^n$. If $s_1,s_2\in \reals$ and $1<p_1,p_2<\infty$ are such that $W^{s_1,p_1}(\Omega)\hookrightarrow W^{s_2,p_2}(\Omega)$, then $W^{s_1,p_1}_{loc}(\Omega)\hookrightarrow W^{s_2,p_2}_{loc}(\Omega)$.
\end{theorem}
\begin{proof}
We have
\begin{align*}
u\in W^{s_1,p_1}_{loc}(\Omega)&\Longleftrightarrow \forall\,\varphi\in C_c^\infty(\Omega)\qquad \varphi u\in W^{s_1,p_1}(\Omega)\\
&\Longrightarrow \forall\,\varphi\in C_c^\infty(\Omega)\qquad \varphi u\in W^{s_2,p_2}(\Omega)
\\
&\Longleftrightarrow u\in W^{s_2,p_2}_{loc}(\Omega)
\end{align*}
So $W^{s_1,p_1}_{loc}(\Omega)\subseteq W^{s_2,p_2}_{loc}(\Omega)$. Now note that for all $\varphi\in C_c^\infty(\Omega)$
\begin{equation*}
|u|_{\varphi, s_2,p_2}=\|\varphi u\|_{W^{s_2,p_2}(\Omega)}\preceq \|\varphi u\|_{W^{s_1,p_1}(\Omega)}=|u|_{\varphi, s_1,p_1}
\end{equation*}
So it follows from Theorem ~\ref{thmapptvconvergence1} that the inclusion map from $W^{s_1,p_1}_{loc}(\Omega)$ to $W^{s_2,p_2}_{loc}(\Omega)$ is continuous.
\end{proof}

\begin{theorem}[Embedding Theorem II]\lab{thmmay7828}
Let $\Omega$ be a nonempty open set in $\reals^n$ that has the interior Lipschitz property. Suppose that  $s_1,s_2\in \reals$ and $1<p_1,p_2<\infty$ are such that $W^{s_1,p_1}(U)\hookrightarrow W^{s_2,p_2}(U)$ for all bounded open sets $U$ with Lipschitz continuous boundary. If $s_1<0$, further assume that $(-s_1,p_1',\Omega)$ is a smooth multiplication triple. If $s_2<0$, further assume that $(-s_2,p'_2,\Omega)$ is a smooth multiplication triple. Then $W^{s_1,p_1}_{loc}(\Omega)\hookrightarrow W^{s_2,p_2}_{loc}(\Omega)$.
\end{theorem}

\begin{proof}
Suppose $u\in W^{s_1,p_1}_{loc}(\Omega)$ and $\varphi\in C_c^\infty(\Omega)$. Let $\Omega'$ be an open set in $\Omega$ that contains $\textrm{supp}\,\varphi$ and has Lipschitz continuous boundary. We have
\begin{align*}
u\in W^{s_1,p_1}_{loc}(\Omega)&\Longrightarrow \varphi u\in W^{s_1,p_1}(\Omega)
\stackrel{\textrm{Theorem ~\ref{thmapril181}}}{\Longrightarrow} (\varphi u)|_{\Omega'}\in W^{s_1,p_1}(\Omega')\\
&\Longrightarrow (\varphi u)|_{\Omega'}\in W^{s_2,p_2}(\Omega')\\
&\stackrel{\textrm{Theorem ~\ref{thmapril181}}}{\Longrightarrow} \varphi u\in W^{s_2,p_2}(\Omega)
\end{align*}
Since $\varphi$ can be any element of $C_c^\infty(\Omega)$, we can conclude that if $u\in W^{s_1,p_1}_{loc}(\Omega)$, then $u\in W^{s_2,p_2}_{loc}(\Omega)$. In order to prove the continuity of the inclusion map we can proceed as follows: let $\varphi\in C_c^\infty(\Omega)$ and choose $\Omega'$ as before.
\begin{align*}
|u|_{\varphi, s_2,p_2}&=\|\varphi u\|_{W^{s_2,p_2}(\Omega)}\stackrel{\textrm{Theorem ~\ref{thmapril181}}}{\simeq} \|\varphi u\|_{W^{s_2,p_2}(\Omega')}\\
&
\preceq \|\varphi u\|_{W^{s_1,p_1}(\Omega')}\stackrel{\textrm{Theorem ~\ref{thmapril181}}}{\simeq} \|\varphi u\|_{W^{s_1,p_1}(\Omega)}\\
&=|u|_{\varphi, s_1,p_1}
\end{align*}
So it follows from Theorem ~\ref{thmapptvconvergence1} that the inclusion map from $W^{s_1,p_1}_{loc}(\Omega)$ to $W^{s_2,p_2}_{loc}(\Omega)$ is continuous.
\end{proof}

A version of compact embedding for spaces $H^{m}_{loc}$ with integer smoothness degree has been studied in \cite{Antonic1}. In what follows we will state the corresponding theorem and its proof for spaces of locally Sobolev-Slobodeckij functions.

\begin{lemma}\lab{lemmay2903}
Suppose that $(s,p,\Omega)$ and $(-s,p',\Omega)$ are smooth multiplication triples. If $u_m$ converges \textbf{weakly} to $u$ in $W^{s,p}_{loc}(\Omega)$ then
\begin{equation*}
\forall\,\varphi\in C_c^\infty(\Omega)\qquad \varphi u_m\rightharpoonup \varphi u\quad \textrm{in $W^{s,p}(\Omega)$}
\end{equation*}
\end{lemma}

\begin{proof}
The proof is based on the following well-known fact:\\
\textbf{Fact 1:} Let $X$ be a topological space and suppose that $x$ is a point in $X$. Let $\{x_m\}$ be a sequence in $X$. If every subsequence of $\{x_m\}$ contains a subsequence that converges to $x$, then $x_m\rightarrow x$.\\
Let $\varphi\in C_c^\infty(\Omega)$. By Fact 1, it is enough to show that every subsequence of $\varphi u_m$ has a subsequence that converges weakly to $\varphi u$ in $W^{s,p}(\Omega)$. Let $\varphi u_{m'}$ be a subsequence of $\varphi u_m$. We have
\begin{align*}
u_{m'}\rightharpoonup u\quad \textrm{in $W^{s,p}_{loc}(\Omega)$}&\stackrel{\textrm{Corollary ~\ref{cormay7322}}}{\Longrightarrow}
\textrm{$\{u_{m'}\}$ is bounded in $W^{s,p}_{loc}(\Omega)$}
\\
&\stackrel{\textrm{Corollary ~\ref{cormay13343}}}{\Longrightarrow}  \textrm{$\{\varphi u_{m'}\}$ is bounded in $W^{s,p}(\Omega)$}
\end{align*}
Since $W^{s,p}(\Omega)$ is reflexive, there exists a subsequence $\varphi u_{m''}$ that converges weakly to some  $F\in W^{s,p}(\Omega)$. To finish the proof it is enough to show that $F=\varphi u$. We have
\begin{align*}
u_{m''}\rightharpoonup u\quad \textrm{in $W^{s,p}_{loc}(\Omega)$}&\Longrightarrow
u_{m''}\rightarrow u \quad\textrm{in $(D'(\Omega), \textrm{weak}^*)$}\\
&\Longrightarrow\varphi u_{m''}\rightarrow \varphi u \quad\textrm{in $(D'(\Omega), \textrm{weak}^*)$}
\end{align*}
In the first line we used Theorem ~\ref{thmmay2} and the fact that $W^{s,p}_{loc}(\Omega)\hookrightarrow (D'(\Omega), \textrm{weak}^*)$ (see Corollary ~\ref{cormay13615}). In the second line we used the fact that multiplication by smooth functions is a continuous operator on $(D'(\Omega), \textrm{weak}^*)$.\\
Similarly, since $W^{s,p}(\Omega)\hookrightarrow (D'(\Omega), \textrm{weak}^*)$, it follows from Theorem ~\ref{thmmay2} that
\begin{align*}
\varphi u_{m''}\rightharpoonup F \quad\textrm{in $W^{s,p}(\Omega)$}\Longrightarrow
\varphi u_{m''}\rightarrow F \quad\textrm{in $(D'(\Omega), \textrm{weak}^*)$}
\end{align*}
Consequently, $\varphi u=F$ as elements of $D'(\Omega)$ and subsequently as elements of $W^{s,p}_0(\Omega)$.
\end{proof}

\begin{theorem}[Compact Embedding]\lab{thmmay7829}
Let $\Omega$ be a nonempty open set in $\reals^n$ that has the interior Lipschitz property. Suppose that $(s_1,p_1,\Omega)$ and $(-s_1,p'_1,\Omega)$ are smooth multiplication triples. If $s_2<0$, further assume that $(-s_2,p_2',\Omega)$ is a smooth multiplication triple. Moreover, suppose that  $s_1,s_2, p_1$, and $p_2$ are such that $W^{s_1,p_1}(U)$ is compactly embedded in $W^{s_2,p_2}(U)$ for all bounded open sets $U$ with Lipschitz continuous boundary. Then every bounded sequence in $W^{s_1,p_1}_{loc}(\Omega)$ has a convergent subsequence in $W^{s_2,p_2}_{loc}(\Omega)$.
\end{theorem}

\begin{proof}
The proof makes use of the following well-known fact:\\
\textbf{Fact 2:} Let $X$ and $Y$ be Banach spaces. Suppose that $T: X\rightarrow Y$ is a linear compact operator. If the sequence $x_m$ converges weakly (i.e. with respect to the weak topology) to $x$ in $X$, then $T(x_n)$ converges to $T(x)$ (with respect to the norm of $Y$) in $Y$.\\
Let $u_m$ be a bounded sequence in $W^{s_1,p_1}_{loc}(\Omega)$. By Theorem ~\ref{thmmay2858}, since $W^{s_1,p_1}_{loc}(\Omega)$ is a separable reflexive Frechet space, there exists $u\in W^{s_1,p_1}_{loc}(\Omega)$ and a subsequence $\{u_{m'}\}$ such that $u_{m'}\rightharpoonup u$ in $W^{s_1,p_1}_{loc}(\Omega)$. We claim that $\{u_{m'}\}$ converges to $u$ in $W^{s_2,p_2}_{loc}(\Omega)$, that is, for all $\varphi\in C_c^\infty(\Omega)$, $\varphi u_{m'}\rightarrow \varphi u$ in $W^{s_2,p_2}_{loc}(\Omega)$. Suppose that $\varphi\in C_c^\infty(\Omega)$ and let $K:=\textrm{supp}\,\varphi$. By Lemma ~\ref{lemmay2903} we have
\begin{equation*}
\varphi u_{m'}\rightharpoonup \varphi u \qquad \textrm{in $W^{s_1,p_1}(\Omega)$}
\end{equation*}
So by Theorem ~\ref{thmmay21220}
\begin{equation*}
\varphi u_{m'}\rightharpoonup \varphi u \qquad \textrm{in $W^{s_1,p_1}_K(\Omega)$}
\end{equation*}

Let $\Omega'$ be an open bounded set in $\Omega$ with Lipschitz continuous boundary such that $K \subseteq \Omega'$. By Theorem ~\ref{thmapril181}, the restriction map from $W^{s_1,p_1}_K(\Omega)$ to $W^{s_1,p_1}(\Omega')$ is well-defined and continuous. It follows from Theorem ~\ref{thmmay21022} that this restriction map is weak-weak continuous. So $\varphi u_{m'}\rightharpoonup \varphi u$ in $W^{s_1,p_1}_K(\Omega)$ implies that  $\varphi u_{m'}\rightharpoonup \varphi u$ in $W^{s_1,p_1}(\Omega')$. By assumption the identity map from $W^{s_1,p_1}(\Omega')$ to $W^{s_2,p_2}(\Omega')$ is compact, so it follows from Fact 2 that $\varphi u_{m'}\rightarrow \varphi u$ in $W^{s_2,p_2}(\Omega')$ which subsequently implies  $\varphi u_{m'}\rightarrow \varphi u$ in $W^{s_2,p_2}(\Omega)$ by Theorem ~\ref{thmapril181}.
\end{proof}

\section{Other Properties}
The main results of this section do not appear to be in the literature in the generality appearing here and they play a fundamental role in the study of the properties of differential operators between Sobolev spaces of sections of vector bundles on manifolds equipped with nonsmooth metric (see \cite{holstbehzadan2017c,holstbehzadan2018b}).
\begin{theorem}\lab{thmmay141107}
Let $\Omega$ be a nonempty open set in $\reals^n$, $s\geq 1$ and $1<p<\infty$. Then $u\in W^{s,p}_{loc}(\Omega)$ if and only if $u\in L^p_{loc}(\Omega)$ and for all $1\leq i\leq n$, $\displaystyle \frac{\partial u}{\partial x^i}\in W^{s-1,p}_{loc}(\Omega)$.
\end{theorem}

\begin{proof}
\begin{align*}
&u\in W^{s,p}_{loc}(\Omega)\Longleftrightarrow \forall\,\varphi\in C_c^\infty(\Omega)\quad \varphi u\in W^{s,p}(\Omega)\\
&\stackrel{\textrm{Theorem ~\ref{thmmay141029}}}{\Longleftrightarrow} \forall\,\varphi\in C_c^\infty(\Omega)\quad
\textrm{$\varphi u\in L^p(\Omega)$ and for all $1\leq i\leq n$, $\frac{\partial (\varphi u)}{\partial x^i}\in W^{s-1,p}(\Omega)$}
\end{align*}
Note that $\displaystyle \frac{\partial (\varphi u)}{\partial x^i}=\frac{\partial \varphi}{\partial x^i}u+\varphi \frac{\partial u}{\partial x^i}$. Since $\displaystyle \frac{\partial \varphi}{\partial x^i}u\in W^{s,p}(\Omega)\hookrightarrow W^{s-1,p}(\Omega)$, we have
\begin{equation*}
\frac{\partial (\varphi u)}{\partial x^i}\in W^{s-1,p}(\Omega) \Longleftrightarrow
\varphi \frac{\partial u}{\partial x^i}\in W^{s-1,p}(\Omega)
\end{equation*}
Therefore
\begin{align*}
&u\in W^{s,p}_{loc}(\Omega)\Longleftrightarrow \forall\,\varphi\in C_c^\infty(\Omega)\quad
\textrm{$\varphi u\in L^p(\Omega)$ and for all $1\leq i\leq n$, $\varphi\frac{\partial  u}{\partial x^i}\in W^{s-1,p}(\Omega)$}\\
&\Longleftrightarrow \textrm{$u\in L^p_{loc}(\Omega)$ and for all $1\leq i\leq n$, $\displaystyle \frac{\partial u}{\partial x^i}\in W^{s-1,p}_{loc}(\Omega)$}
\end{align*}
\end{proof}

\begin{theorem}\lab{thmmay141123}
Let $\Omega$ be a nonempty open set in $\reals^n$, $k\in\mathbb{N}$ and $1<p<\infty$. Then $u\in W^{k,p}_{loc}(\Omega)$ if and only if $\partial^\alpha u\in L^p_{loc}(\Omega)$ for all $|\alpha|\leq k$.
\end{theorem}

\begin{proof}
We prove the claim by induction on $k$. For $k=1$ we have
\begin{align*}
u\in W^{1,p}_{loc}(\Omega)
& \stackrel{\textrm{Theorem ~\ref{thmmay141107}}}{\Longleftrightarrow}  u\in L^{p}_{loc}(\Omega), \, \forall\,1\leq i\leq n \,\,\frac{\partial u}{\partial x^i}\in L^p_{loc}(\Omega)\\
&\Longleftrightarrow  \forall\,|\alpha|\leq 1\quad \partial^\alpha u\in L^p_{loc}(\Omega)
\end{align*}
Now suppose the claim holds for $k=m$. For $k=m+1$ we have
\begin{align*}
&u\in W^{m+1,p}_{loc}(\Omega)\stackrel{\textrm{Theorem ~\ref{thmmay141107}}}{\Longleftrightarrow}  u\in L^{p}_{loc}(\Omega), \, \forall\,1\leq i\leq n \,\,\frac{\partial u}{\partial x^i}\in W^{m,p}_{loc}(\Omega)\\
& \stackrel{\textrm{induction hypothesis}}{\Longleftrightarrow}
u\in L^{p}_{loc}(\Omega), \, \forall\,1\leq i\leq n \,\,\forall\, 0\leq |\alpha|\leq m\,\,\partial^\alpha \big[\frac{\partial u}{\partial x^i}\big]\in L^{p}_{loc}(\Omega)\\
&\Longleftrightarrow u\in L^{p}_{loc}(\Omega), \, \,\forall\, 0< |\beta|\leq m+1\,\,\partial^\beta u\in L^{p}_{loc}(\Omega)\\
&\Longleftrightarrow \forall\, |\beta|\leq m+1\,\,\partial^\beta u\in L^{p}_{loc}(\Omega)
\end{align*}
\end{proof}

\begin{theorem}\lab{thmapp8a}
 Let $s\in \reals$, $1<p<\infty$, and $\alpha\in
\mathbb{N}_0^n$. Suppose $\Omega$ is a nonempty bounded open set in $\reals^n$ with Lipschitz continuous boundary. Then
\begin{enumerateX}
\item the linear operator $\partial^\alpha: W^{s,p}_{loc}(\reals^n)\rightarrow
W^{s-|\alpha|,p}_{loc}(\reals^n)$ is well-defined and continuous;
\item for $s<0$, the linear operator $\partial^\alpha: W^{s,p}_{loc}(\Omega)\rightarrow
W^{s-|\alpha|,p}_{loc}(\Omega)$ is well-defined and continuous;
\item for $s\geq 0$ and $|\alpha|\leq s$, the linear operator $\partial^\alpha: W^{s,p}_{loc}(\Omega)\rightarrow
W^{s-|\alpha|,p}_{loc}(\Omega)$ is well-defined and continuous;
\item if $s\geq 0$, $s-\frac{1}{p}\neq \textrm{integer}$ (i.e. the fractional part of $s$ is not equal to $\frac{1}{p}$), then the linear operator $\partial^\alpha: W^{s,p}_{loc}(\Omega)\rightarrow W^{s-|\alpha|,p}_{loc}(\Omega)$ for $|\alpha|>s$ is well-defined and continuous.
\end{enumerateX}
\end{theorem}

\begin{proof}
This is the counterpart of Theorem \ref{winter88} for locally Sobolev functions. Here we will prove the first item. The remaining items can be proved using a similar technique.
\begin{itemizeX}
\item \textbf{Step 1:} First we prove by induction on $|\alpha|$ that if $u\in W^{s,p}_{loc}(\reals^n)$, then $\partial^\alpha u\in W^{s-|\alpha|,p}_{loc}(\reals^n)$. Let $\varphi\in C_c^\infty(\reals^n)$; we need to show that $\varphi \partial^\alpha u\in W^{s-|\alpha|,p}(\reals^n)$. If $|\alpha|=0$, there is nothing to prove. If $|\alpha|=1$, there exists $1\leq i\leq n$ such that $\partial^\alpha=\frac{\partial}{\partial x^i}$. We have
    \begin{equation*}
    \varphi\partial^\alpha u=\varphi \frac{\partial u}{\partial x^i}=\frac{\partial (\varphi u)}{\partial x^i}-\frac{\partial \varphi}{\partial x^i}u
    \end{equation*}
    By assumption $\varphi u\in W^{s,p}(\reals^n)$ and so it follows from Theorem \ref{winter88} that the first term on the right hand side is in $W^{s-1,p}(\reals^n)$. Also, since $u\in W^{s,p}_{loc}(\reals^n)$, the second term on the right hand side is in $W^{s,p}(\reals^n)\hookrightarrow W^{s-1,p}(\reals^n)$. Hence $\varphi\partial^\alpha u\in W^{s-1,p}(\reals^n)$. Now suppose the claim holds for all $|\alpha|\leq k$. Suppose $\alpha$ is a multi-index such that $|\alpha|=k+1$. Clearly there exists $1\leq i\leq n$ such that $\partial^\alpha=\frac{\partial}{\partial x^i}(\partial^\beta)$ where $\beta$ is a multi-index with $|\beta|=k$. By the induction hypothesis, $\partial^\beta u\in W^{s-|\beta|,p}_{loc}(\reals^n)$ and so by the argument that was presented for the base case we have $\frac{\partial}{\partial x^i}\partial^\beta u\in W^{s-|\beta|-1,p}_{loc}(\reals^n)=W^{s-|\alpha|,p}_{loc}(\reals^n)$.
\item \textbf{Step 2:} In this step we prove the continuity. Again we use induction on $|\alpha|$. Let $|\alpha|=1$. Choose $i$ as in the previous step. For every $\varphi\in C_c^\infty (\reals^n)$ we have
    \begin{align*}
    \|\varphi\frac{\partial}{\partial x^i}u\|_{s-1,p}&=\|\frac{\partial (\varphi u)}{\partial x^i}-\frac{\partial \varphi}{\partial x^i}u\|_{s-1,p}\\
    &\leq \|\frac{\partial (\varphi u)}{\partial x^i}\|_{s-1,p}+\|\frac{\partial \varphi}{\partial x^i}u\|_{s-1,p}\\
    &\preceq \|\varphi u\|_{s,p}+\|\frac{\partial \varphi}{\partial x^i}u\|_{s,p}
    \end{align*}
    On the right hand side we have sum of two of the seminorms that define the topology of $W^{s,p}_{loc}(\reals^n)$. It follows from item 2. of Theorem \ref{thmapptvconvergence1} that $\partial^\alpha: W^{s,p}_{loc}(\reals^n)\rightarrow W^{s-1,p}_{loc}(\reals^n)$ is continuous. Now suppose the claim holds for all $|\alpha|\leq k$. Suppose $\alpha$ is a multi-index such that $|\alpha|=k+1$. Clearly there exists $1\leq i\leq n$ such that $\partial^\alpha=\frac{\partial}{\partial x^i}(\partial^\beta)$ where $\beta$ is a multi-index with $|\beta|=k$. We have
{\fontsize{10}{10}{\begin{align*}
\|\varphi \partial^{\alpha} &u\|_{s-|\alpha|,p}=\|\varphi\frac{\partial}{\partial x^i}(\partial^\beta u)\|_{s-|\alpha|,p}\\
&\stackrel{\textrm{argument of the base case}}{\preceq}\|\varphi\partial^\beta u\|_{s-|\alpha|+1,p}+\|\frac{\partial \varphi}{\partial x^i}\partial^\beta u\|_{s-|\alpha|+1,p}\\
&\preceq\|\varphi\partial^\beta u\|_{s-|\beta|,p}+\|\frac{\partial \varphi}{\partial x^i}\partial^\beta u\|_{s-|\beta|,p}\\
&\stackrel{\textrm{induction hypothesis; Theorem \ref{thmapptvconvergence1}}}{\preceq} \max(\|\varphi_1 u\|_{s,p},\cdots,\|\varphi_k u\|_{s,p} )+\|\frac{\partial \varphi}{\partial x^i}\partial^\beta u\|_{s-|\beta|,p}\\
&\stackrel{\textrm{induction hypothesis; Theorem \ref{thmapptvconvergence1}}}{\preceq} \max(\|\varphi_1 u\|_{s,p},\cdots,\|\varphi_k u\|_{s,p} )+\max(\|\psi_1 u\|_{s,p},\cdots,\|\psi_l u\|_{s,p} )\\
&\preceq \max (\|\varphi_1 u\|_{s,p},\cdots,\|\varphi_k u\|_{s,p},\|\psi_1 u\|_{s,p},\cdots,\|\psi_l u\|_{s,p})
\end{align*}}}
for some $\varphi_1,\cdots,\varphi_k$ and $\psi_1,\cdots,\psi_l$ in $C_c^\infty(\reals^n)$. It follows from item 2. of Theorem \ref{thmapptvconvergence1} that $\partial^\alpha: W^{s,p}_{loc}(\reals^n)\rightarrow W^{s-|\alpha|,p}_{loc}(\reals^n)$ is continuous.
\end{itemizeX}
\end{proof}
Next we want to establish a counterpart of Theorem \ref{thm3.3} for locally Sobolev-Slobodeckij spaces. To this end, first we state and prove a simple lemma.

\begin{lemma}\lab{lemapp8c}
Let $\Omega$ be a nonempty open subset of $\reals^n$. Suppose $u:\Omega\rightarrow \reals$ and
$\tilde{u}:\Omega\rightarrow \reals$ are such that
$u=\tilde{u}\,\,a.e.$ If $\tilde{u}$ is continuous then
$\textrm{supp}\tilde{u}\subseteq \textrm{supp} u$.
\end{lemma}

\begin{proof}[Proof by Contradiction] Suppose $x\in \textrm{supp}\tilde{u}\setminus \textrm{supp}
u$. Since $x$ belongs to the complement of $\textrm{supp}\,u$,
which is an open set, there exists $\epsilon>0$ such that
$B_\epsilon(x)\subseteq \Omega$ and $B_\epsilon(x)\cap
\textrm{supp}\,u=\emptyset$. Since $x\in \textrm{supp}\tilde{u}$,
there exists $y\in B_{\epsilon/4}(x)$ such that $\tilde{u}(y)\neq
0$. $\tilde{u}$ is continuous, therefore there exists
$0<\delta<\frac{\epsilon}{4}$ such that $\tilde{u}(z)\neq 0$ for
all $z\in B_{\delta}(y)\subseteq B_\epsilon(x)$. But $u=0$ a.e. on
$B_\epsilon (x)$. This contradicts the fact that
$u=\tilde{u}\,\,a.e.$
\end{proof}

\begin{theorem}\lab{thmapp8b}
Let $\Omega$ be a nonempty bounded open set in $\reals^n$ with Lipschitz continuous boundary or $\Omega=\reals^n$. Suppose $u\in W^{s,p}_{loc}(\Omega)$ where $sp>n$. Then $u$ has a continuous
version.
\end{theorem}

\begin{proof}
Let $\{V_j\}_{j\in\mathbb{N}_0}$ and $\{\psi_j\}_{j\in \mathbb{N}_0}$ be as in Theorem \ref{winter3}.
Note that $u=\sum_j \psi_j u$. For all $j$, $\psi_j u\in
W^{s,p}(\Omega)$ so by Theorem \ref{thm3.3} there exists $\tilde{u}_j\in C(\Omega)$ such
that $\psi_j u=\tilde{u}_j$ on $\Omega\setminus A_j$ where $A_j$
is a set of measure zero.  Also by Lemma \ref{lemapp8c}
$\textrm{supp}\tilde{u}_j\subseteq \textrm{supp} \psi_j$.
 Therefore for any $x\in \Omega$ only finitely many of
$\tilde{u}_j(x)$'s are nonzero. So we may define
$\tilde{u}:\Omega\rightarrow \reals$ by $\tilde{u}=\sum_j
\tilde{u}_j$. Clearly $\tilde{u}=u$ on $\Omega\setminus A$ where
$A=\cup A_j$ (so $A$ is a set of measure zero). Consequently
$\tilde{u}=u\,\, a.e$. It remains to show that $\tilde{u}:\Omega
\rightarrow \reals$ is indeed continuous. To this end suppose
$a_m\rightarrow a$ in $\Omega$. We need to prove that
$\tilde{u}(a_m)\rightarrow \tilde{u}(a)$. Let $\epsilon>0$ be such that $\overline{B_\epsilon(a)}\subseteq \Omega$. So
$B_\epsilon(a)$ intersects only finitely many of
$\textrm{supp}\tilde{u}_j$'s; let's denote them by
$\tilde{u}_{r_1},\cdots,\tilde{u}_{r_l}$. Also since
$a_m\rightarrow a$ there exists $M$ such that for all $m\geq M$,
$a_m\in B_\epsilon(a)$. Hence
\begin{align*}
&\tilde{u}(a)=\sum_j\tilde{u}_j(a)=\tilde{u}_{r_1}(a)+\cdots+\tilde{u}_{r_l}(a)\\
&\forall\,m\geq M\quad
\tilde{u}(a_m)=\tilde{u}_{r_1}(a_m)+\cdots+\tilde{u}_{r_l}(a_m)
\end{align*}
Recall that $\tilde{u}_{r_1}+\cdots+\tilde{u}_{r_l}$ is a finite
sum of continuous functions and so it is continuous. Thus
\begin{equation*}
\lim_{m\rightarrow \infty}\tilde{u}(a_m)=\lim_{m\rightarrow
\infty}
(\tilde{u}_{r_1}+\cdots+\tilde{u}_{r_l})(a_m)=\tilde{u}_{r_1}(a)+\cdots+\tilde{u}_{r_l}(a)=\tilde{u}(a)
\end{equation*}
\end{proof}

\begin{remark}\lab{remmay7831}
In the above proof the only place we used the assumption of $\Omega$ being Lipschitz was in applying Theorem \ref{thm3.3}. We can replace this assumption by the weaker assumption that $\Omega$ has the interior Lipschitz property. Then, since $\textrm{supp}\,(\psi_j u)$ is compact, there exists $\Omega'$ with Lipschitz boundary that contains $\textrm{supp}\, (\psi_j u)$. Then by Theorem ~\ref{thmapril181}, $\psi_j u\in W^{s,p}(\Omega')$ and so it has a continuous version $\hat{u}_j\in C(\Omega')$. Since $\psi_j u=\hat{u}_j$ almost everywhere on $\Omega'$ and $\psi_j u=0$ outside of the compact set $\textrm{supp}\,\psi_j$, we can conclude that $\textrm{ext}^0_{\Omega',\Omega}\hat{u}_j$ is in $C(\Omega)$ and it is almost everywhere equal to $\psi_j u$. We set $\tilde{u}_j=\textrm{ext}^0_{\Omega',\Omega}\hat{u}_j$. The rest of the proof will be exactly the same as before.
\end{remark}

\begin{theorem}\lab{lemapp1}
Let $\Omega=\reals^n$ or $\Omega$ be a bounded open set in $\reals^n$ with Lipschitz continuous boundary. Suppose $s_1, s_2, s\in \reals$ and $1<p_1,p_2,p<\infty$ are such that
\begin{equation*}
W^{s_1,p_1}(\Omega)\times W^{s_2,p_2}(\Omega)\hookrightarrow W^{s,p}(\Omega)\,.
\end{equation*}
Then
\begin{enumerate}
\item $W^{s_1,p_1}_{loc}(\Omega)\times W^{s_2,p_2}_{loc}(\Omega)\hookrightarrow
W^{s,p}_{loc}(\Omega)$.
\item For all $K\in \mathcal{K}(\Omega)$, $W^{s_1,p_1}_{loc}(\Omega)\times W^{s_2,p_2}_{K}(\Omega)\hookrightarrow
W^{s,p}(\Omega)$. In particular, if $f\in W^{s_1,p_1}_{loc}(\Omega)$, then the mapping $u\mapsto fu$ is a well-defined continuous linear map from $W^{s_2,p_2}_{K}(\Omega)$ to $W^{s,p}(\Omega)$.
\end{enumerate}
\end{theorem}

\begin{remark}\lab{remmay7832}
In the above theorem, since the locally Sobolev spaces on $\Omega$ are metrizable, the
continuity of the mapping
\begin{equation*}
W^{s_1,p_1}_{loc}(\Omega)\times W^{s_2,p_2}_{loc}(\Omega)\rightarrow
W^{s,p}_{loc}(\Omega),\quad (u,v)\mapsto uv
\end{equation*}
can be interpreted as follows: if $u_i\rightarrow u$ in
$W^{s_1,p_1}_{loc}(\Omega)$ and $v_i\rightarrow v$ in
$W^{s_2,p_2}_{loc}(\Omega)$, then $u_iv_i\rightarrow uv$ in
$W^{s,p}_{loc}(\Omega)$. Also since $W^{s_2,p_2}_{K}(\Omega)$ is considered as a normed
subspace of $W^{s_2,p_2}(\Omega)$, we have a similar interpretation of the continuity of the mapping in item 2.
\end{remark}

\begin{proof}
\leavevmode
\begin{enumerateXALI}
\item Suppose $u\in W^{s_1,p_1}_{loc}(\Omega)$ and $v\in
W^{s_2,p_2}_{loc}(\Omega)$. First we show that $uv$ is in
$W^{s,p}_{loc}(\Omega)$. Clearly the set $A=\{\varphi^2: \varphi\in
C_c^{\infty}(\Omega)\}$ is an admissible family of test
functions. So in order to show that $uv\in W^{s,p}_{loc}(\Omega)$, it is
enough to show that for all $\varphi\in C_c^{\infty}(\Omega)$,
$\varphi^2 uv=(\varphi u)(\varphi v)$ is in $W^{s,p}(\Omega)$. This is
clearly true because $\varphi u\in W^{s_1,p_1}(\Omega)$, $\varphi v\in
W^{s_2,p_2}(\Omega)$, and by assumption
\begin{equation*}
W^{s_1,p_1}(\Omega)\times W^{s_2,p_2}(\Omega)\hookrightarrow W^{s,p}(\Omega)
\end{equation*}
In order to prove the continuity of the map $(u,v)\mapsto uv$,
suppose $u_i\rightarrow u$ in $W^{s_1,p_1}_{loc}(\Omega)$ and
$v_i\rightarrow v$ in $W^{s_2,p_2}_{loc}(\Omega)$. We need to
show that  $u_iv_i\rightarrow uv$ in $W^{s,p}_{loc}(\Omega)$.
That is, we need to prove that for all $\varphi\in
C_c^{\infty}(\Omega)$
\begin{equation*}
\varphi^2 u_iv_i\rightarrow \varphi^2 uv\quad \textrm{in
$W^{s,p}(\Omega)$}
\end{equation*}
We have
\begin{align*}
& \textrm{$u_i\rightarrow u$ in
$W^{s_1,p_1}_{loc}(\Omega)$}\Longrightarrow \textrm{$\varphi
u_i\rightarrow \varphi u$ in $W^{s_1,p_1}(\Omega)$}\\
& \textrm{$v_i\rightarrow v$ in
$W^{s_2,p_2}_{loc}(\Omega)$}\Longrightarrow \textrm{$\varphi
v_i\rightarrow \varphi v$ in $W^{s_2,p_2}(\Omega)$}\\
\end{align*}
By assumption $W^{s_1,p_1}(\Omega)\times W^{s_2,p_2}(\Omega)\hookrightarrow
W^{s,p}(\Omega)$, so
\begin{equation*}
(\varphi u_i)(\varphi v_i)\rightarrow (\varphi u)(\varphi v)\quad
\textrm{in $W^{s,p}(\Omega)$}
\end{equation*}
\item Suppose $u\in W^{s_1,p_1}_{loc}(\Omega)$ and $v\in
W^{s_2,p_2}_{K}(\Omega)$. First we show that $uv$ is in
$W^{s,p}(\Omega)$. To this end, let $\varphi\in C_c^\infty(\Omega)$ be such that $\varphi=1$ on a neighborhood containing $K$. We have
\begin{equation*}
uv=u(\varphi v)=\underbrace{(\varphi u)}_{\in W^{s_1,p_1}(\Omega)} \underbrace{v}_{\in W^{s_2,p_2}(\Omega)}\in W^{s,p}(\Omega)
\end{equation*}
Now we prove the continuity. Suppose $u_i\rightarrow u$ in $W^{s_1,p_1}_{loc}(\Omega)$ and $v_i\rightarrow v$ in $W^{s_2,p_2}_{K}(\Omega)$. Let $\varphi$ be as before. We have
\begin{align*}
& u_i\rightarrow u \quad \textrm{in $W^{s_1,p_1}_{loc}(\Omega)$}\Longrightarrow \varphi u_i\rightarrow \varphi u \quad \textrm{in $W^{s_1,p_1}(\Omega)$}\\
& v_i\rightarrow v\quad \textrm{in $W^{s_2,p_2}(\Omega)$}
\end{align*}
This together with the assumption that $W^{s_1,p_1}(\Omega)\times W^{s_2,p_2}(\Omega)\hookrightarrow
W^{s,p}(\Omega)$ implies $\varphi u_iv_i\rightarrow \varphi uv$ in $W^{s,p}(\Omega)$. Since $\varphi v=v$ and $\varphi v_i=v_i$, we conclude that $u_iv_i\rightarrow uv$ in $W^{s,p}(\Omega)$.
\end{enumerateXALI}
\end{proof}

\begin{remark}\lab{remmay7833}
In the above theorem the assumption that $\Omega$ is Lipschitz or $\reals^n$ was used only to ensure that we can apply Theorem ~\ref{thmapp7b} and to make sure that the locally Sobolev spaces involved are metrizable. For item (1) we can use the weaker assumption that $(s_1,p_1,\Omega)$, $(s_2,p_2,\Omega)$, and $(s,p,\Omega)$ are interior smooth multiplication triples. For item (2) we just need to assume that $(s_1,p_1,\Omega)$ is an interior smooth multiplication triple.
\end{remark}

\begin{corollary}\lab{corapp1}
Let $\Omega$ be the same as the previous theorem. If $sp>n$, then $W^{s,p}_{loc}(\Omega)$ is closed under multiplication.
Moreover if
\begin{equation*}
(f_1)_m\rightarrow f_1\quad \textrm{in $W^{s,p}_{loc}(\Omega)$},\cdots,
(f_l)_m\rightarrow f_l\quad \textrm{in $W^{s,p}_{loc}(\Omega)$}\,
\end{equation*}
then
\begin{equation*}
(f_1)_m\cdots (f_l)_m\rightarrow f_1\cdots f_l\quad \textrm{in
$W^{s,p}_{loc}(\Omega)$}
\end{equation*}
\end{corollary}

The next theorem plays a key role in the study of differential operators on manifolds equipped with nonsmooth metrics (see \cite{holstbehzadan2017c}).

\begin{theorem}\lab{lemapp7}
Let $\Omega=\reals^n$ or let $\Omega$ be a nonempty bounded open set in $\reals^n$ with Lipschitz continuous boundary. Let $s\in \reals$ and
$p\in (1,\infty)$ be such that $sp>n$. Let
$B:\Omega\rightarrow GL(k,\reals)$. Suppose for all $x\in \Omega$
and $1\leq i,j\leq k$, $B_{ij}(x)\in W^{s,p}_{loc}(\Omega)$. Then
\begin{enumerate}
\item $\textrm{det}\,B\in W^{s,p}_{loc}(\Omega)$.
\item Moreover if for each $m\in \mathbb{N}$ $B_m: \Omega\rightarrow
GL(k,\reals)$ and for all $1\leq i,j\leq k$
$(B_m)_{ij}\rightarrow B_{ij}$ in $W^{s,p}_{loc}(\Omega)$, then
$\textrm{det}\, B_m\rightarrow \textrm{det}\, B$ in
$W^{s,p}_{loc}(\Omega)$.
\end{enumerate}
\end{theorem}

\begin{proof}
\leavevmode
\begin{enumerateXALI}
\item By Leibniz formula we have
\begin{equation*}
\textrm{det}\,B(x)=\sum_{\sigma\in
S_n}\textrm{sgn}(\sigma)B_{\sigma (1),1}\cdots B_{\sigma (k),k}
\end{equation*}
By assumption for all $1\leq i\leq k$, $B_{\sigma (i),i}$ is in
 $W^{s,p}_{loc}(\Omega)$. Since $sp>n$, it follows from Corollary
 ~\ref{corapp1} that $\textrm{det}\,B\in W^{s,p}_{loc}(\Omega)$.
\item Since $(B_m)_{ij}\rightarrow B_{ij}$ in
$W^{s,p}_{loc}(\Omega)$ it again follows from Corollary
 ~\ref{corapp1} that for all $\sigma \in S_n$
 \begin{equation*}
(B_m)_{\sigma (1),1}\cdots (B_m)_{\sigma (k),k} \rightarrow
B_{\sigma (1),1}\cdots B_{\sigma (k),k}\quad \textrm{in
$W^{s,p}_{loc}(\Omega)$}
 \end{equation*}
 Thus $\textrm{det}\, B_m\rightarrow \textrm{det}\, B$ in $W^{s,p}_{loc}(\Omega)$.
\end{enumerateXALI}
\end{proof}

\begin{theorem}\lab{thmapp9}
Let $\Omega=\reals^n$ or let $\Omega$ be a nonempty bounded open set in $\reals^n$ with Lipschitz continuous boundary. Let $s\geq 1$ and
$p\in (1,\infty)$ be such that $sp>n$.
\begin{enumerate}
\item Suppose that $u\in W^{s,p}_{loc}(\Omega)$ and that $u(x)\in I$ for
all $x\in \Omega$ where $I$ is some interval in $\reals$. If
$F:I\rightarrow \reals$ is a smooth function, then $F(u)\in
W^{s,p}_{loc}(\Omega)$.
\item Suppose that $u_m\rightarrow u$ in $W^{s,p}_{loc}(\Omega)$
and that for all $m\geq 1$ and $x\in \Omega$, $u_m(x),u(x)\in I$
where $I$ is some open interval in $\reals$. If $F:I\rightarrow
\reals$ is a smooth function, then $F(u_m)\rightarrow F(u)$ in
 $W^{s,p}_{loc}(\Omega)$.
\item If $F:\reals\rightarrow \reals$ is a smooth function, then
the map taking $u$ to $F(u)$ is continuous from
$W^{s,p}_{loc}(\Omega)$ to $W^{s,p}_{loc}(\Omega)$.
\end{enumerate}
\end{theorem}
\begin{proof}
The proof of part~(1) generalizes the argument given in~\cite{dM05}.
Let $k=\floor{s}$. First we show that $F(u)\in
W^{k,p}_{loc}(\Omega)$. To this end we fix a multi-index
$|\alpha|=m\leq k$ and we show that $\partial^{\alpha}(F(u))\in
L^p_{loc}(\Omega)$ (see Theorem ~\ref{thmmay141123}).\\
It follows from the chain rule (and induction) that
$\partial^{\alpha}(F(u))$ is a sum of the terms of the form
\begin{equation*}
F^{(l)}(u)\partial^{\beta_1}u\cdots \partial^{\beta_r}u
\end{equation*}
where $l\in \mathbb{N}$ and $\sum_{i=1}^r |\beta_i|=m$. It is
a consequence of Theorem \ref{lemapp1} that if $s_1,s_2\geq s_3\geq 0$ and
$s_1+s_2-s_3>\frac{n}{p}$, then $W^{s_1,p}_{loc}(\Omega)\times
W^{s_2,p}_{loc}(\Omega)\hookrightarrow W^{s_3,p}_{loc}(\Omega)$. As a consequence
\begin{align*}
& W^{s-|\beta_1|,p}_{loc}(\Omega)\times
W^{s-|\beta_2|,p}_{loc}(\Omega)\hookrightarrow W^{s-|\beta_1|-|\beta_2|,p}_{loc}(\Omega)\\
& W^{s-|\beta_1|-|\beta_2|,p}_{loc}(\Omega)\times
W^{s-|\beta_3|,p}_{loc}(\Omega)\hookrightarrow W^{s-|\beta_1|-|\beta_2|-|\beta_3|,p}_{loc}(\Omega)\\
&\vdots\\
& W^{s-|\beta_1|-\cdots-|\beta_{r-1}|,p}_{loc}(\Omega)\times
W^{s-|\beta_r|,p}_{loc}(\Omega)\hookrightarrow W^{s-|\beta_1|-\cdots-|\beta_r|,p}_{loc}(\Omega)=W^{s-m,p}_{loc}(\Omega)
\end{align*}
Considering this and the
fact that $\partial^{\beta_i}u\in W^{s-|\beta_i|,p}_{loc}(\Omega)$, we
have
\begin{equation*}
\partial^{\beta_1}u\cdots \partial^{\beta_r}u\in W^{t,p}_{loc}(\Omega)
\end{equation*}
for all $0\leq t\leq s-m$. In particular, $\partial^{\beta_1}u\cdots
\partial^{\beta_r}u\in W^{0,p}_{loc}(\Omega)=L^p_{loc}(\Omega)$. Also, since $F$
is smooth and $u$ is continuous, $F^{(l)}(u)\in
L^{\infty}_{loc}(\Omega)$. Therefore
\begin{equation*}
F^{(l)}(u)\partial^{\beta_1}u\cdots \partial^{\beta_r}u\in
L^p_{loc}(\Omega)
\end{equation*}
So $F(u)\in W^{k,p}_{loc}(\Omega)$ where $k=\floor{s}$. Now, for noninteger $s$, we use
a bootstrapping argument to show that $F(u)$ in fact belongs to
$W^{s,p}_{loc}(\Omega)$.\\
$F'$ is smooth, therefore $F'(u)\in W^{k,p}_{loc}(\Omega)$. Also
$\frac{\partial u}{\partial x^i} \in W^{s-1,p}_{loc}(\Omega)$ (note that $s-1\geq 0$). By
Theorem \ref{lemapp1} we have
\begin{equation*}
W^{k,p}_{loc}(\Omega)\times
W^{s-1,p}_{loc}(\Omega)\hookrightarrow W^{t-1,p}_{loc}(\Omega)
\end{equation*}
provided that
\begin{align*}
k\geq t-1\geq 0,\quad s-1\geq t-1\geq 0,\quad k+(s-1)-(t-1)>\frac{n}{p}
\end{align*}
Therefore $\frac{\partial}{\partial x^i}(F(u))=F'(u)\frac{\partial u}{\partial x^i}\in
W^{t-1,p}_{loc}(\Omega)$ for all $1\leq t\leq s$ such that
$t<k+(s-\frac{n}{p})$. Consequently $F(u)\in
W^{t,p}_{loc}(\Omega)$ for all $1\leq t\leq s$ such that
$t<k+(s-\frac{n}{p})$ (see Theorem ~\ref{thmmay141107}). Now we can repeat this argument by starting
with "$F'$ is smooth, therefore $F'(u)\in W^{t,p}_{loc}(\Omega)$
for all $1\leq t\leq s$ such that $t<k+(s-\frac{n}{p})$". This results
in $F(u)\in W^{t,p}_{loc}(\Omega)$ for all $1\leq t\leq s$ such that
$t<k+2(s-\frac{n}{p})$. Repeating this a finite number of times
shows that $F(u)\in W^{s,p}_{loc}(\Omega)$.\\
Our next goal is to prove items 2 and 3. First we note that if $0\in I$ then WLOG we may assume that
$F(0)=0$. Indeed, the constant function $F(0)$ is an element of
$W^{s,p}_{loc}(\Omega)$. So
\begin{equation*}
F(u_m)\rightarrow F(u)\quad \textrm{in
$W^{s,p}_{loc}(\Omega)$}\Longleftrightarrow \tilde{F}(u_m)\rightarrow
\tilde{F}(u)\quad \textrm{in $W^{s,p}_{loc}(\Omega)$}
\end{equation*}
where $\tilde{F}(t)=F(t)-F(0)$. Thus WLOG we may assume that
$F(0)=0$.\\ Let $\{K_j\}_{j\in\mathbb{N}_0}$, $\{V_j\}_{j\in\mathbb{N}_0}$, and $\{\psi_j\}_{j\in \mathbb{N}_0}$ be as in Theorem \ref{winter3}. Clearly $\{\psi_j\}$ is an admissible family of
functions. Therefore in order to show that $F(u_m)\rightarrow
F(u)$ in $W^{s,p}_{loc}(\Omega)$ it is enough to prove that
\begin{equation*}
\forall\, r\in \mathbb{N}_0\qquad \psi_r(F(u_m)-F(u))\rightarrow
0\quad \textrm{in $W^{s,p}(\Omega)$ as $m\rightarrow \infty$}
\end{equation*}
Let $\psi_{r_1},\cdots,\psi_{r_k}$ be those admissible test
functions whose support intersects the support of $\psi_r$.
 So
\begin{equation*}
\forall\, x\in \textrm{supp}\,\psi_r\qquad \sum_{j\in
\mathbb{N}_0}\psi_j u=\psi_{r_1}u+\cdots+\psi_{r_k}u
\end{equation*}
Consequently
\begin{equation*}
\psi_r(F(u_m)-F(u))=\psi_rF(\psi_{r_1}u_m+\cdots+\psi_{r_k}u_m)-\psi_rF(\psi_{r_1}u+\cdots+\psi_{r_k}u)
\end{equation*}
Since $u_m\rightarrow u$ in $W^{s,p}_{loc}(\Omega)$, for all $1\leq i\leq
k$ we have
\begin{equation*}
\psi_{r_i}u_m\rightarrow \psi_{r_i}u\qquad \textrm{in $W^{s,p}(\Omega)$}
\end{equation*}
and so
\begin{equation*}
\psi_{r_1}u_m+\cdots+\psi_{r_k}u_m\rightarrow
\psi_{r_1}u+\cdots+\psi_{r_k}u\qquad \textrm{in $W^{s,p}(\Omega)$}
\end{equation*}
Since $W^{s,p}(\Omega)\hookrightarrow L^\infty (\Omega)$ also we
have
\begin{equation}\lab{equniform1}
\psi_{r_1}u_m+\cdots+\psi_{r_k}u_m\rightarrow
\psi_{r_1}u+\cdots+\psi_{r_k}u\qquad \textrm{in
$L^{\infty}(\Omega)$}
\end{equation}
Consequently for the continuous versions of
$\psi_{r_1}u_m+\cdots+\psi_{r_k}u_m$ and
$\psi_{r_1}u+\cdots+\psi_{r_k}u$ we have uniform convergence.
From this point, we work with these continuous versions. The
continuous function $\psi_{r_1}u+\cdots+\psi_{r_k}u$ attains its
max and min on the compact set $\textrm{supp}\psi_r$ which we
denote by $A_{max}$ and $A_{min}$, respectively. Note that
\begin{equation*}
\forall\,x\in \textrm{supp}\psi_r\quad
(\psi_{r_1}u+\cdots+\psi_{r_k}u)(x)=u(x)\in I
\end{equation*}
So $A_{max}$ and $A_{min}$ are in $I$ (that is
$[A_{min},A_{max}]\subseteq I$). Let $\epsilon>0$ be such that
$[A_{min}-2\epsilon,A_{max}+2\epsilon]\subseteq I$. By
(\ref{equniform1}) there exists $M$ such that
\begin{equation*}
\forall\, m\geq M,\,\,\forall\,x\in\textrm{supp}\psi_r\qquad
(\psi_{r_1}u_m+\cdots+\psi_{r_k}u_m)(x)\in
[A_{min}-\epsilon,A_{max}+\epsilon]\subseteq I
\end{equation*}
Let $\xi\in C_c^\infty (\reals)$ be such that $\xi=1$ on
$[A_{min}-\epsilon,A_{max}+\epsilon]$ and $\xi=0$ outside of
$[A_{min}-2\epsilon,A_{max}+2\epsilon]\subseteq I$. Define
$\hat{F}:\reals\rightarrow \reals$ by
\begin{equation*}
\hat{F}(t)=\begin{cases}
      \hfill \xi(t)F(t)    \hfill & \text{ if $t \in I$ } \\
      \hfill 0 \hfill & \text{ if $t\not \in I$}
  \end{cases}
\end{equation*}
Clearly $\hat{F}:\reals\rightarrow \reals$ is a smooth function
and $\hat{F}(0)=0$. Moreover $\hat{F}=F$ on
$[A_{min}-\epsilon,A_{max}+\epsilon]$. Also for all $x\in \Omega$
and $m\geq M$ we have
\begin{align*}
\psi_r(F(u_m)-F(u))&=\psi_rF(\psi_{r_1}u_m+\cdots+\psi_{r_k}u_m)-\psi_rF(\psi_{r_1}u+\cdots+\psi_{r_k}u)\\
&=\psi_r\hat{F}(\psi_{r_1}u_m+\cdots+\psi_{r_k}u_m)-\psi_r\hat{F}(\psi_{r_1}u+\cdots+\psi_{r_k}u)
\end{align*}
Indeed, if $x\not \in \textrm{supp}\psi_r$, then both sides are
equal to zero. If $x\in \textrm{supp}\psi_r$, then
\begin{align*}
& (\psi_{r_1}u+\cdots+\psi_{r_k}u)(x)\in
[A_{min},A_{max}] \\
& (\psi_{r_1}u_m+\cdots+\psi_{r_k}u_m)(x)\in
[A_{min}-\epsilon,A_{max}+\epsilon]
\end{align*}
and so
\begin{align*}
& F((\psi_{r_1}u+\cdots+\psi_{r_k}u)(x))=\hat{F}((\psi_{r_1}u+\cdots+\psi_{r_k}u)(x))\\
&
F((\psi_{r_1}u_m+\cdots+\psi_{r_k}u_m)(x))=\hat{F}((\psi_{r_1}u_m+\cdots+\psi_{r_k}u_m)(x))
\end{align*}
$\hat{F}$ is a smooth function and its value at $0$ is $0$. Also
by assumption $sp>n$. Therefore the mapping $v\rightarrow
\psi_r\hat{F}(v)$ from $W^{s,p}(\Omega)$ to $W^{s,p}(\Omega)$ is continuous. Hence
\begin{equation*}
\psi_r\hat{F}(\psi_{r_1}u_m+\cdots+\psi_{r_k}u_m)\rightarrow
\psi_r\hat{F}(\psi_{r_1}u+\cdots+\psi_{r_k}u) \qquad \textrm{in
$W^{s,p}(\Omega)$}
\end{equation*}
That is
\begin{equation*}
\psi_r(F(u_m)-F(u))\rightarrow 0\qquad \textrm{in $W^{s,p}(\Omega)$}
\end{equation*}
So we proved item 2. Finally we note that $W^{s,p}_{loc}(\Omega)$ is metrizable. So
continuity of the mapping $u\rightarrow F(u)$ is equivalent to
sequential continuity which was proved in item 2.
\end{proof}

\section*{Acknowledgments}
   \label{sec:ack}

MH was supported in part by NSF Awards~1262982, 1318480, and 1620366.
AB was supported by NSF Award~1262982.

\bibliographystyle{abbrv}
\bibliography{refsscratch}
\end{document}